\documentclass[a4paper]{article}
\usepackage{amsmath, amsthm, amsfonts}
\usepackage{amssymb}
\usepackage{latexsym}
\usepackage{verbatim}
\usepackage{url}
\usepackage{graphicx}

\usepackage[all,cmtip]{xy}
\usepackage{color}
\usepackage{hyperref}
\usepackage{booktabs}

\usepackage{enumerate}
\usepackage{array}
\usepackage{nicefrac}
\usepackage[section]{placeins}

\DeclareMathAlphabet{\mathfr}{U}{euf}{m}{n}

\newtheorem{theorem}{Theorem}[section]
\newtheorem{definition}[theorem]{Definition}
\newtheorem{conjecture}[theorem]{Conjecture}
\newtheorem{proposition}[theorem]{Proposition}

\newtheorem{lemma}[theorem]{Lemma}

\newtheorem{remark}[theorem]{Remark}

\numberwithin{equation}{section}

\newcommand{\Aut}{\operatorname{Aut}}
\newcommand{\Conj}{\operatorname{Conj}}
\newcommand{\dens}{\operatorname{dens}}
\newcommand{\End}{\operatorname{End}}
\newcommand{\Exp}{\mathrm{E}}
\newcommand{\Frob}{\operatorname{Frob}}
\newcommand{\Gal}{\mathrm{Gal}}
\newcommand{\Hom}{\operatorname{Hom}}
\newcommand{\Id}{\mathrm{Id}}
\newcommand{\Jac}{\operatorname{Jac}}

\newcommand{\topo}{\operatorname{top}}
\newcommand{\Trace}{\operatorname{Trace}}

\newcommand{\AST}{\operatorname{AST}}
\newcommand{\GL}{\mathrm{GL}}
\newcommand{\GSp}{\mathrm{GSp}}
\newcommand{\Hg}{\mathrm{Hg}}
\newcommand{\Lef}{\operatorname{L}}
\newcommand{\MT}{\mathrm{MT}}

\newcommand{\SL}{\mathrm{SL}}
\newcommand{\SO}{\mathrm{SO}}
\newcommand{\Sp}{\mathrm{Sp}}
\newcommand{\SU}{\mathrm{SU}}
\newcommand{\ST}{\mathrm{ST}}
\newcommand{\TL}{\operatorname{TL}}
\newcommand{\Unitary}{\mathrm{U}}
\newcommand{\USp}{\mathrm{USp}}
\newcommand{\Zar}{\mathrm{Zar}}

\newcommand{\smallmat}[4]{\bigl(\begin{smallmatrix}#1&#2\\#3&#4\end{smallmatrix}\bigr)}

\newcommand{\cyc}[1]{{\mathrm{C}_#1}}
\newcommand{\dih}[1]{{\mathrm{D}_#1}}
\newcommand{\alt}[1]{{\mathrm{A}_#1}}
\newcommand{\sym}[1]{{\mathrm{S}_#1}}

\newcommand{\bbn}{\hat{b}_n}
\newcommand{\bcn}{\sum_k\binom{n}{k}2^{n-k}b_kc_k}
\newcommand{\ccn}{\hat{c}_n}

\newcommand{\bA}{{\mathbf{A}}}
\newcommand{\bB}{{\mathbf{B}}}
\newcommand{\bC}{{\mathbf{C}}}
\newcommand{\bD}{{\mathbf{D}}}
\newcommand{\bE}{{\mathbf{E}}}
\newcommand{\bF}{{\mathbf{F}}}

\newcommand{\rhoA}{\rho_A}

\newcommand{\Aseq}[1]{\htmladdnormallink{A#1} {http://oeis.org/A#1}}

\newcommand{\C}{\mathbb C}
\newcommand{\F}{\mathbb F}
\newcommand{\HH}{\mathbb H}
\newcommand{\M}{\mathrm{M }}

\newcommand{\Q}{\mathbb Q}
\newcommand{\Qbar}{{\overline{\mathbb Q}}}
\newcommand{\R}{\mathbb R}
\newcommand{\Z}{\mathbb Z}

\newcommand{\lra}{\longrightarrow}
\newcommand{\ra}{\rightarrow}
\newcommand{\norm}[1]{\left\Vert#1\right\Vert}
\newcommand{\nothing}{{}}

\begin{document}
\title{Sato-Tate distributions and\\Galois endomorphism modules in genus 2}
\author{Francesc Fit\'e, Kiran S. Kedlaya,\\ V\'ictor Rotger, and Andrew V. Sutherland}
\date{January 27, 2012}

\maketitle


\begin{abstract}
For an abelian surface $A$ over a number field $k$, we study the
limiting distribution of the normalized Euler factors of the $L$-function of $A$.
This distribution is expected to correspond to taking characteristic polynomials of a uniform random
matrix in some closed subgroup of $\USp(4)$; this \emph{Sato-Tate group} may be obtained
from the Galois action on any Tate module of $A$. We show that the Sato-Tate group is limited to a particular list of
55 groups up to conjugacy. We then classify $A$ according to the Galois module structure on the
$\R$-algebra generated by endomorphisms of $A_{\Qbar}$ (the \emph{Galois type}),
and establish a matching with the classification of Sato-Tate groups;
this shows that there are at most 52 groups up to conjugacy which occur as Sato-Tate groups for suitable
$A$ and $k$, of which 34 can occur for $k = \Q$.
Finally, we exhibit examples of Jacobians of hyperelliptic curves exhibiting each Galois type (over $\Q$ whenever possible),
and observe numerical agreement with the expected
Sato-Tate distribution by comparing moment statistics.
\end{abstract}

\section{Introduction}\label{intro}
The celebrated \emph{Sato-Tate conjecture} concerns the distribution of
Euler factors of an elliptic curve over a number field. It specifically predicts that this distribution always takes
one of three forms, one occurring whenever the elliptic curve fails to have complex multiplication (the generic case),
and two exceptional cases arising for CM curves (only one of which occurs over $\Q$). Substantial progress has been made on
this conjecture only very recently (see \S~\ref{subsec:tractable cases}).

The purpose of this paper is to formulate a precise analogue of the Sato-Tate conjecture for abelian surfaces.
We present strong theoretical and experimental evidence that the distribution of Euler factors can take as many as
52 forms, 35 of which occur if we limit to fields with a real place and 34 of which occur if we limit to $\Q$.

In the rest of the introduction, we introduce the Sato-Tate problem for a general abelian variety,
describe a conjectural description of the distribution of Euler factors in terms of a certain compact Lie group (the \emph{Sato-Tate group}),
then discuss our analysis of this conjecture from three points of view:
\begin{enumerate}
\item[(a)] a classification of compact Lie groups that are compatible with restrictions on the Sato-Tate group imposed by existing conjectures;
\item[(b)] a classification of Galois module structures on the $\R$-algebra generated by endomorphisms
(which we call the \emph{Galois type});
\item[(c)] numerical computations testing the relationship between these classifications and observed
distributions of Euler factors.
\end{enumerate}
We conclude with some speculation about how much of the Sato-Tate conjecture for abelian surfaces may be tractable
in the near future.

Throughout the introduction, let $k$ denote a number field, let $G_k$ be an absolute Galois group of $k$,
let $\mathfrak{p}$ be a prime ideal
of $k$, and let $q := \norm{\mathfrak{p}}$ denote the absolute norm of $\mathfrak{p}$. When we make statements about averages over prime ideals,
we always sort in increasing order by norm; it will not matter how ties are broken.
For $A$ an abelian variety over $k$ and $k'$ a field containing $k$, write $A_{k'}$ for the base extension of
$A$ from $k$ to $k'$.

\subsection{The Sato-Tate conjecture and random matrices}

Let $E$ be an elliptic curve over $k$.
If $E$ has good reduction at $\mathfrak{p}$ (which excludes only finitely many primes), then Hasse's theorem implies that
the integer
$a_{\mathfrak{p}} = q + 1 - \#E(\mathbb{F}_q)$ has absolute value at most
$2 \sqrt{q}$. One observes in examples\footnote{To observe for yourself, see the animations at
\url{http://math.mit.edu/~drew}.} that the quantities
\[
\theta_{\mathfrak{p}} := \arccos \frac{a_{\mathfrak{p}}}{2 \sqrt{q}},
\]
appear to be equidistributed with respect to a certain measure on $[0, \pi]$.
\begin{itemize}
\item
If $E$ has complex multiplication defined over $k$, one takes the uniform measure.
Note that this case cannot occur if $k = \Q$, or even if $k$ has a real place.
\item
If $E$ has complex multiplication not defined over $k$, one takes half of the uniform measure plus a discrete measure of mass $1/2$ concentrated at $\pi/2$.
The discrete measure occurs because $a_{\mathfrak{p}} = 0$ whenever $\mathfrak{p}$ fails to split
 in the quadratic extension of $k$
over which the complex multiplication is defined.
\item
If $E$ fails to have complex multiplication, one takes the measure
$\frac{2}{\pi} \sin^2 \theta\,d\theta$.
\end{itemize}
In the first two cases, it is easy to prove equidistribution using the explicit description of $a_{\mathfrak{p}}$
in terms of Gr\"ossencharacters. The third case is subtler: it is the \emph{Sato-Tate conjecture},
which has recently been established when $k$ is totally real (see \S\ref{subsec:tractable cases}).

The measures described above admit the following interpretation. If one chooses a matrix uniformly at random from
$\SU(2)$ (with respect to the Haar measure), its eigenvalues have the form $e^{i \theta}, e^{-i \theta}$
for $\theta$ distributed according to the measure $\frac{2}{\pi} \sin^2 \theta\,d\theta$ on $[0, \pi]$.
If one replaces $\SU(2)$ by $\Unitary(1) \simeq \SO(2)$, one obtains the uniform measure;
if one instead takes the normalizer of $\Unitary(1)$ in $\SU(2)$, one gets half of the uniform measure
(from the identity connected component) plus a discrete measure of mass $1/2$ concentrated at $\pi/2$ (from the other
connected component).

\subsection{The Sato-Tate group}

The interpretation in terms of random matrices suggests a good formulation of the \emph{Sato-Tate problem}\footnote{In the
language of \cite[\S 8]{Ser11}, we only consider the \emph{weight $1$} case of the Sato-Tate problem. Some examples
of the weight 2 case are discussed in \cite[\S 8.5.6]{Ser11}. See also Remark~\ref{higher weight}.}
for an abelian variety $A$ of arbitrary dimension $g \geq 1$.
If $A$ has good reduction at $\mathfrak{p}$, then
there exists a polynomial $L_{\mathfrak{p}}(A, T) = \prod_{i=1}^{2g} (1 - T \alpha_i)$
over $\Z$ such that for each positive integer $n$,
\[
\# A(\mathbb{F}_{q^n}) = \prod_{i=1}^{2g} (1 - \alpha_i^n).
\]
For example, if $g=1$, then $L_{\mathfrak{p}}(A, T) = 1 - a_{\mathfrak{p}} T + q T^2$.
One can reinterpret $L_{\mathfrak{p}}(A, T)$ as follows: for any prime $\ell$, let
$V_\ell(A)=\Q\otimes T_\ell(A)$ denote the (rational) $\ell$-adic Tate module of $A$.
If $\mathrm{Frob}_{\mathfrak{p}}$ is an arithmetic Frobenius element of $G_k$ for the prime $\mathfrak{p}$, then
\[
L_{\mathfrak{p}}(A,T) = \det(1 - T \mathrm{Frob}_{\mathfrak{p}}, V_{\ell}(A)).
\]
This implies that $L_{\mathfrak{p}}(A,q^{-s})$ is the Euler factor at $\mathfrak{p}$
in the $L$-function of $A$ evaluated at $s$.
Define the \emph{normalized $L$-polynomial}
$\overline{L}_{\mathfrak{p}}(A, T) := L_{\mathfrak{p}}(A, q^{-1/2} T)$;
by a theorem of Weil, the roots of $\overline{L}_{\mathfrak{p}}(A, T)$
in $\C$ have norm 1 and are stable (as a multiset) under complex conjugation.
The polynomial $\overline{L}_{\mathfrak{p}}(A, T)$ thus corresponds to a unique element of the set $\Conj(\USp(2g))$ of conjugacy classes in the unitary symplectic
group $\USp(2g)$.

For generic $A$, Katz and Sarnak \cite{KS99} predict that
the polynomials $\overline{L}_{\mathfrak{p}}(A, T)$
are equidistributed\footnote{Note the analogy with the Chebotarev density theorem: in a finite Galois extension of number fields,
the distribution of Frobenius conjugacy classes over the Galois group is also governed by the image of the Haar measure.} with respect to the image on $\Conj(\USp(2g))$ of the normalized Haar measure
on $\USp(2g)$.
It is tempting to guess that in general, the $\overline{L}_{\mathfrak{p}}(A, T)$
are equidistributed with respect to the image on $\Conj(\USp(2g))$ of the normalized Haar measure
on a suitable closed subgroup $G$ of $\USp(2g)$ (depending on $A$);
however, to formulate a precise\footnote{It is unclear whether $G$ is even determined up to conjugacy, even though
the conjugacy class of an individual element of $\USp(2g)$ is determined by its characteristic polynomial.}
conjecture, one needs an explicit definition of the group $G$.

We will give in \S\ref{section:STdefinition} a definition of the \emph{Sato-Tate group} $\ST_A$,
using a construction in terms of $\ell$-adic monodromy groups described by Serre in \cite[\S 8.3]{Ser11}.
This construction relates strongly to an earlier description given by Serre in terms of motivic Galois groups
\cite{Ser95}, as well as to the definition of the \emph{Mumford-Tate group} of $A$; see \S\ref{subsec:MT} for more
on the relationship between the Mumford-Tate and Sato-Tate groups.
{}From the construction of $\ST_A$, the map $\mathfrak{p} \mapsto \overline{L}_{\mathfrak{p}}(A,T) \in \Conj(\USp(2g))$
will factor through an assignment $\mathfrak{p} \mapsto s(\mathfrak{p}) \in \Conj(\ST_A)$, enabling us to make the following
conjecture.

\begin{conjecture}[Refined Sato-Tate] \label{Refined Sato-Tate}
For $\ST_A$ the subgroup of $\USp(2g)$ defined in Definition~\ref{Sato-Tate group} and $\mu_{\ST_A}$ the image on $\Conj(\ST_A)$
of the normalized Haar measure on $\ST_A$, the classes $s(\mathfrak{p}) \in \Conj(\ST_A)$
are $\mu_{\ST_A}$-equidistributed.
\end{conjecture}

This is somewhat stronger than asserting equidistribution of the $\overline{L}_{\mathfrak{p}}(A,T)$
for the image measure on $\Conj(\USp(2g))$, but it is only this last conclusion that we test numerically.

\subsection{A group-theoretic classification}

For the remainder of this introduction, we assume\footnote{It should be possible to make a similar analysis
for larger values of $g$, but even for $g=3$ we have no idea how many groups to expect!} $g=2$.
We can further clarify the Sato-Tate conjecture
in this case by classifying groups which can occur as $\ST_A$ in
Conjecture~\ref{Refined Sato-Tate}.
\begin{theorem} \label{group classification}
Let $A$ be an abelian surface over $k$. Then $\ST_A$ is conjugate to one of $55$
particular groups (see Theorem~\ref{Sato-Tate axioms groups} for the list).
\end{theorem}

The proof of Theorem~\ref{group classification} is an exercise with Lie groups, given the limitations
placed on $\ST_A$ by properties of its construction.
Most of the cases arise when the connected part of $\ST_A$ is isomorphic to $\Unitary(1)$;
these cases are related to the finite subgroups of $\SO(3)$.

The alert reader will notice that Theorem~\ref{group classification} mentions a list of 55 groups,
whereas we claimed initially that only 52 groups are possible. That is because three cases survive the
group-theoretic analysis, but are ruled out by the comparison to Galois types
(Theorem~\ref{Galois type to ST group}).

\subsection{Galois types}

Our next step towards controlling the Sato-Tate group of an abelian surface is to study the Galois module structure on the
endomorphism algebra of $A$. This data is in general insufficient to control the Sato-Tate group
(this is related to Mumford's exotic examples of Hodge groups for abelian fourfolds \cite{Mu}); however, in dimension $2$ such pathologies do not
arise, and indeed we gain a complete understanding of the Sato-Tate group this way.

Let $\End(A)$ denote the (not necessarily commutative) ring of endomorphisms of $A$.
Note that these are assumed to be defined over $k$; if we mean to take endomorphisms of $A$ defined over a larger
field $k'$, we will write $\End(A_{k'})$ instead. In fact, we will often write $\End(A_k)$ instead of $\End(A)$
to emphasize rationality over $k$. For any field $L$, write $\End(A)_L$ for $\End(A) \otimes_{\Z} L$.

Let $K/k$ denote the minimal extension over which all endomorphisms of $A_{\Qbar}$ are defined; it is a finite Galois extension of $k$
(see Proposition~\ref{teo:SiRi}).

\begin{definition}\label{type-intro}
Consider pairs $[G,E]$ in which $G$ is a finite group and $E$ is a finite-dimensional
$\R$-algebra equipped with an action of $G$ by $\R$-algebra automorphisms.
An \emph{isomorphism} between two such pairs $[G,E]$ and $[G', E']$ consists of an isomorphism
$G \simeq G'$ of groups and an equivariant isomorphism $E \simeq E'$ of $\R$-algebras.
The \emph{Galois type} associated to $A$ is the isomorphism class of the pair
$[\Gal(K/k),\End(A_K)_{\R}]$.
Note that abelian surfaces defined over different number fields may have the same Galois type.
\end{definition}

\begin{theorem} \label{Galois type to ST group}
Let $A$ be an abelian surface over $k$.
Then the conjugacy class of the Sato-Tate group of $A$ is uniquely determined
by the Galois type and vice versa.
\end{theorem}

In Theorem~\ref{Galois type to ST group}, the passage from the Sato-Tate group to the Galois type is fairly explicit;
see Proposition~\ref{Galois type from ST group}. The reverse implication can be seen by computing the Galois types associated
to the 55 groups named in Theorem~\ref{group classification} and seeing that they are pairwise nonisomorphic. This requires
significantly less data than the full Galois type; for instance,
it is sufficient to keep track of the isomorphism class (as an $\R$-algebra) of the $\R$-subalgebra of $E$
fixed by one subgroup of $\Gal(K/k)$ in each conjugacy class.

We also make a more detailed analysis of the Galois type and its relationship to more apparent arithmetic
of the abelian surface of $A$, such as the shape of the simple factors of $A$.
This leads to the following result; for a more precise statement, see Theorem~\ref{main-section4}.

\begin{theorem} \label{Galois type classification}
There exist exactly $52$ Galois types of abelian surfaces over number fields.
Of these, exactly $35$ can be realized in such a way that
$k$ has a real place, and exactly $34$ can be realized in such a way that $k = \Q$.
\end{theorem}

It is worth pointing out that the definition of the Galois type had to be chosen rather carefully in order to make
Theorem~\ref{Galois type to ST group} valid. For example, in \cite{KS09} one finds examples of abelian surfaces $A=\Jac(C)/\Q$, $A'=\Jac(C')/\Q$ with the same Sato-Tate group such that $A$ is absolutely simple and $A'$ is isogenous to the product of two elliptic curves. This is because $\End(A_K)_\Q$ is a division quaternion algebra while $\End(A'_K)_\Q\simeq \M_2(\Q)$; this shows that in the definition of the Galois type, considering the $\Q$-algebra $\End(A_K)_\Q$ yields a classification which is
too fine.
On the other hand, the $\C$-algebra $\End(A_K)_\C$ gives a classification which is too coarse: there are examples of abelian surfaces $A=\Jac(C)/\Q$, $A'=\Jac(C')/\Q$ with different Sato-Tate groups such that $\End(A_{\Q})_\R\simeq \R  \times \R$ and $\End(A_{\Q})_\R\simeq \C$, and thus $\End(A_{\Q})_\C\simeq \End(A'_{\Q})_\C\simeq  \C \times \C$.

\subsection{Numerical computations}
\label{subsec:numerical computations}

Our final step is to test Conjecture~\ref{Galois type to ST group} numerically for Jacobians of curves of genus 2.
For a given curve $C$, it is a finite problem to identify the Galois type of its Jacobian.
This determines the Sato-Tate group $\ST_{\Jac(C)}$, yielding a predicted distribution of normalized $L$-polynomials.
To test convergence to the predicted distribution, we follow the methodology introduced in \cite{KS09}, and consider \emph{moment statistics} of the linear and quadratic coefficients $a_1$ and $a_2$ of the normalized $L$-polynomial.
We first compute the corresponding moments for all of the Sato-Tate groups that can arise in genus 2; see \S\ref{subsection:atlas}.
For any given curve~$C$, we may then use the methods of \cite{KS08} to compute a large quantity of $L$-polynomial data for $C$, from which we derive moment statistics for $a_1$ and $a_2$ that may be compared to the corresponding moments for the Sato-Tate group.
We may also compare histograms of the normalized $L$-polynomial coefficients with the corresponding density functions for the Sato-Tate group; see \S\ref{subsection:computations}.
Although here we present computational results only after giving the theoretical description of Sato-Tate groups and Galois types,
the order of discovery was the reverse: it would have been quite impossible to
establish the theoretical results without numerical evidence to lead the way.

It is worth mentioning that one can run the methodology in two different directions. Given an abelian surface in explicit form,
one can in principle compute the Galois type and then obtain a prediction for the Sato-Tate distribution. On the other hand,
in practice, it is often easier to compute the Sato-Tate distribution numerically and then use this to guess the Galois type!
This state of affairs may persist for larger $g$; for instance, it may be possible to identify examples of Mumford's
exotic fourfolds most easily from their Sato-Tate distributions.

\subsection{Sato-Tate for abelian surfaces}

Combining all of our ingredients, we now have a precise Sato-Tate conjecture for abelian surfaces,
including the following features.
\begin{enumerate}
\item[(a)]
The $\overline{L}_{\mathfrak{p}}(A, T)$ in $\Conj(\USp(2g))$  are conjectured to be equidistributed
according to the image of the Haar measure for a specific group $\ST_A$
(Definition~\ref{Sato-Tate group}).
\item[(b)]
The statement in (a) can be refined to an equidistribution conjecture on $\Conj(\ST_A)$
itself (Conjecture~\ref{Refined Sato-Tate}).
\item[(c)]
It is known exactly which groups $\ST_A$ can occur in general (there are 52 of them), which can occur for $k=\Q$ (there are 34 of them), and
how $\ST_A$ is related to the endomorphism ring of $A$
(by Theorem~\ref{Galois type to ST group} and Theorem~\ref{Galois type classification}).
\end{enumerate}

By contrast, an analysis of Sato-Tate groups in dimension 2 over $\Q$ had been previously attempted by the second and fourth
authors in \cite{KS09} based on results of an exhaustive search, but produced fewer groups
than we identify here. One reason is that in \cite{KS09},
curves were classified only using the moments of the first coefficient
of the $L$-polynomial, which turns out to be insufficient: over $\Q$, the 34 observed Sato-Tate groups only give rise
to 26 distinct trace distributions. Another reason is that in \cite{KS09}, the Sato-Tate groups were constructed
in a haphazard fashion, without the benefit either of a group-theoretic classification or an interpretation in terms
of Galois types. It was the combination of these additional ingredients that made it possible to identify some exceptional
cases missed in \cite{KS09}; with the benefit of this hindsight, we then made a larger exhaustive search to identify
examples with small coefficients (see \S\ref{subsection:search}).

\subsection{Tractable cases}
\label{subsec:tractable cases}

Having asserted a precise form of the Sato-Tate conjecture for abelian surfaces,
we conclude this introduction by discussing to what extent it may be possible to prove cases of this conjecture
in the near future (though not in this paper).

We first recall the paradigm for proving equidistribution described by Serre \cite[\S I.A.2]{Se68}.
Conjecture~\ref{Refined Sato-Tate} asserts that for any continuous function $f:\Conj(\ST_A) \to \C$,
\begin{equation}\label{equi}
\mu_{\ST_A}(f)=\lim_{n\ra \infty} \frac{\sum_{\norm{\mathfrak p} \leq n}f(s(\mathfrak p))}{\# \{\mathfrak p: \norm{\mathfrak p}\leq n\}}.
\end{equation}
By the Peter-Weyl theorem, the space of characters of $\ST_A$ is dense for the supremum norm on the space of continuous functions on $\Conj(\ST_A)$, so we need only check \eqref{equi} when $f = \chi$ is an irreducible character;
 we may omit the trivial character since \eqref{equi} is obvious in that case.
By emulating the proofs of the prime number theorem by Hadamard and de la Vall\'ee Poussin, one then obtains the following
result.
\begin{theorem}[Serre] \label{serre}
Suppose that for each nontrivial irreducible linear representation\footnote{It is enough to consider the restrictions of
representations of $\USp(2g)$, but then one finds a pole at $s=1$
of order equal to the multiplicity of the trivial representation in $\rho|_{\ST_A}$.} $\rho$ of $\ST_A$, the Dirichlet series
$$
L(A,\rho,s)=\prod_\mathfrak p \det(1-\rho(s(\mathfrak p))\norm{\mathfrak p}^{-s})^{-1}
$$
(which converges absolutely for $\mathrm{Re}(s) > 1$) extends to a holomorphic function on $\mathrm{Re}(s) \geq 1$
which does not vanish anywhere on the line $\mathrm{Re}(s) = 1$.
 Then Conjecture~\ref{Refined Sato-Tate} holds for $A$.
\end{theorem}

For the most part, the only known method for establishing meromorphic continuation of $L(A,\rho,s)$
is to show that it arises from an automorphic form. This has been achieved recently in many cases where
$A$ corresponds to an automorphic form of $\mathrm{GL}_2$-type, thanks to the work of
Barnet-Lamb, Clozel, Gee, Geraghty, Harris, Shepherd-Barron, and Taylor. For example,
Conjecture~\ref{Refined Sato-Tate} is known (with $\ST_A = \SU(2)$) whenever
$A$ is an elliptic curve over a totally real field without complex multiplication. (See \cite{BLGG11} for an even stronger result
covering Hilbert modular forms.)

For $g=2$, it is unclear at present how to approach the Sato-Tate problem
in even a single case with $\ST_A = \USp(4)$;
however, one can hope to treat all cases with $\ST_A \neq \USp(4)$ using existing technology.
For instance, it should be straightforward to verify Conjecture~\ref{Refined Sato-Tate} when $A$ is isogenous
to the product of two elliptic curves with complex multiplication. For $k$ totally real, it should also be possible to
use the result of \cite{BLGG11} to handle cases where $A$ is isogenous to the product of two elliptic curves, one of
which has complex multiplication. The case where $A$ is isogenous to the product of two nonisogenous non-CM elliptic
curves is harder, but has been treated by Harris \cite{Har} conditionally on some results on the stable trace formula
which have subsequently been verified; see \cite{BLGHT} for the appropriate references.
(Thanks to Toby Gee and Michael Harris for this explanation.)

\subsection{Acknowledgments}

Thanks to Jean-Pierre Serre for numerous helpful discussions about the general Sato-Tate conjecture, for providing
preliminary drafts of \cite{Ser11}, and for triggering the collaboration among the four authors by alerting the first and third
authors to the existence of \cite{KS08, KS09}. Thanks to Joan-Carles Lario for his help at the first stage of the project.
Thanks to Grzegorz Banaszak for helpful discussions with Kedlaya about the algebraic Sato-Tate group and the Mumford-Tate group,
leading to the paper \cite{BK}. Thanks to Kevin Buzzard and Marco Streng for helpful conversations with Rotger about the formulation of Conjecture \ref{Refined Sato-Tate} (resp. on complex multiplication).
Thanks to Amanda Clemm for the code used to generate Figure \ref{figure:Ta2hist} in Sage \cite{WS11}.
Thanks to Xavier Guitart, Joan-C. Lario, and Marc Masdeu for organizing the workshop ``Sato-Tate in higher dimension''
at the Centro de Ciencias de Benasque Pedro Pascual (CCBPP) in July-August 2011, which provided the first opportunity for the four authors
to collaborate in person. Thanks also to the CCBPP for its hospitality, and to Josep Gonz\'alez for his lecture at the workshop providing examples
of modular abelian surfaces leading to the correct definition of the Galois type.

During this work, Kedlaya received financial support from NSF CAREER grant DMS-0545904,
NSF grant DMS-1101343, DARPA grant HR0011-09-1-0048, MIT (NEC Fund,
Cecil and Ida Green professorship), and UC San Diego
(Stefan E. Warschawski professorship). Sutherland received financial support from NSF
grant DMS-1115455. Fit\'e and Rotger received financial support from DGICYT Grant MTM2009-13060-C02-01. Fit\'e also received financial support from Fundaci\'o Ferran Sunyer i Balaguer.

\section{Construction of the Sato-Tate group}
\label{section:STdefinition}

Let $A$ be an abelian variety over a number field $k$ of dimension $g \geq 1$.
In this section, we give an explicit definition of the Sato-Tate group $\ST_A$ that appears in our refined
form of the Sato-Tate conjecture (Conjecture~\ref{Refined Sato-Tate}).
This loosely follows the presentation given by Serre in \cite[\S 8.3]{Ser11}.
We then relate this group to the Mumford-Tate
group; for abelian varieties of dimension at most 3, this gives a precise description of the Sato-Tate group in terms of the endomorphisms
of~$A$ and the Galois action on them. This depends crucially on work of the second author with Grzegorz
Banaszak \cite{BK}.

\subsection{\texorpdfstring{$\ell$}{l}-adic monodromy and the Sato-Tate group}

Choose a polarization on $A$ and embeddings $k \hookrightarrow \Qbar \hookrightarrow \C$.
Choose a symplectic basis for the singular homology group $H_1(A^{\topo}_{\C}, \Q)$ and use it to equip this space with
an action of $\GSp_{2g}(\Q)$, where
$\GSp_{2g}/\Q$ denotes the reductive algebraic group over $\Q$ such that for any field $F$,
$$
\GSp_{2g}(F)=\{ \gamma \in \GL_{2g}(F): \gamma^t \begin{smallmat}0{-\mathrm{Id}}{\mathrm{Id}}0\end{smallmat} \gamma = \lambda_\gamma \cdot  \begin{smallmat}0{-\mathrm{Id}}{\mathrm{Id}}0\end{smallmat}, \,\lambda_\gamma \in F^\times \}.
$$
Now fix a prime $\ell$, let $V_{\ell}(A)$ denote the rational $\ell$-adic Tate module of $A$, and make the identifications
\[
V_{\ell}(A)
\simeq
H_{1,\operatorname{et}}(A_{\Qbar}, \Q_\ell) \simeq
H_{1,\operatorname{et}}(A_{\C}, \Q_\ell) \simeq
H_1(A^{\topo}_{\C}, \Q_\ell) \simeq H_1(A^{\topo}_{\C}, \Q) \otimes_{\Q} \Q_\ell.
\]
Under these identifications, the Weil pairing on the Tate module is identified with the cup product pairings in
\'etale and singular homology, so our chosen symplectic basis for $H_1(A^{\topo}_{\C}, \Q)$ maps to a symplectic basis of
$V_{\ell}(A)$. The action of $G_k$ on $V_{\ell}(A)$ thus defines a continuous homomorphism
\begin{equation}\label{rho-ell}
\varrho_{A,\ell}: G_k \lra \GSp_{2g}(\Q_\ell).
\end{equation}
Write $G_\ell = G_\ell(A)$ for the image of $G_k$ under $\varrho_{A,\ell}$.

\begin{definition}
Let $G_\ell^{\Zar}=G_\ell^{\Zar}(A)$ denote the Zariski closure of $G_\ell$ inside $\GSp_{2g}(\Q_\ell)$;
this is sometimes called the \emph{$\ell$-adic monodromy group} of $A$.
\end{definition}

A result of Bogomolov \cite{Bo80} (plus a bit of $p$-adic Hodge theory; see \cite[\S 8.3.2]{Ser11} for references) ensures that $G_\ell$ is open in $G_\ell^{\Zar}$,
and this construction behaves reasonably under enlargement of the field $k$. See
\cite[\S 3]{BK} for similar arguments.
\begin{remark} \label{R:STextendfield}
Let $k'$ be a finite extension of $k$.
Since $G_{\ell}(A_{k'})$ is an open subgroup of the compact group $G_{\ell}$ by Bogomolov,
$G_{\ell}^{\Zar}(A_{k'})$ is a subgroup of $G_{\ell}^{\Zar}$ of finite index;
in particular, the two groups have the same connected part.
On the other hand, by making $k'$ large enough, we can force $G_{\ell}(A_{k'})$ to lie within a neighborhood
of the identity in $\GSp_{2g}(\Q_\ell)$ that is small enough to miss all of the nonidentity connected components of $G_{\ell}^{\Zar}$.
Consequently, for any sufficiently large $k'$, $G_{\ell}^{\Zar}(A_{k'})$ equals the connected part of $G_{\ell}^{\Zar}$.
This means that on one hand, the connected part of $G_{\ell}^{\Zar}$ is an invariant of $A_{\Qbar}$;
on the other hand, the component group $\pi_0(G_{\ell}^{\Zar})$ receives a surjective continuous homomorphism from $G_k$, and so may be identified with  $\Gal(K/k)$ for some\footnote{It is not clear from this construction that $K$ is independent of
$\ell$, but this has been shown by Serre \cite{Ser81}.} finite extension $K$ of $k$.
\end{remark}

\begin{remark}
By a celebrated theorem of Faltings \cite{Fa}, one has
\begin{equation}\label{monodromy}
\End(A)_{\Q_\ell}= \End(V_{\ell}(A))^{G_\ell}.
\end{equation}
By \eqref{monodromy},
the $\ell$-adic monodromy groups of $A_{k'}$ for all finite extensions $k'$ of $k$
determine the $G_k$-module $\End(A_{\Qbar})_{\Q_{\ell}}$.
The converse is not true in general: there are examples due to Mumford \cite{Mu}
of simple abelian fourfolds $A/k$ such that $\End(A_{\Qbar})=\Z$ while $G_\ell^{\Zar}\subsetneq \GSp_{2g}(\Q_\ell)$.
Fortunately, since we only consider the case $g=2$ in this paper, we will avoid such pathologies; see
Theorem~\ref{mumford-tate genus 2}.
\end{remark}

To get the Sato-Tate group, we must modify the above construction slightly.
\begin{definition}
Let $G_k^1$ denote the kernel of the cyclotomic character $\chi_\ell: G_k \ra \Q_\ell^\times$, and let
$G_\ell^{1,\Zar} = G_\ell^{1,\Zar}(A)$ be the Zariski closure of $\varrho_{A,\ell}(G_k^1)$ in $\GSp_{2g}/\Q_\ell$.
By Bogomolov's theorem, the natural inclusion of $G_\ell^{1,\Zar}$ into the kernel of the composition $G_{\ell}^{\Zar} \to \GSp_{2g} \to \mathbb{G}_m$ is an isomorphism.
\end{definition}

\begin{remark} \label{R:STextendfield analogue}
We have an analogue of Remark~\ref{R:STextendfield}:
for $k'$ a finite extension of $k$, $G_{\ell}^{1,\Zar}(A_{k'})$ is a subgroup of $G_{\ell}^{1,\Zar}(A)$ with the same
connected part, and is itself connected when $k'$ is sufficiently large.
\end{remark}

\begin{definition} \label{Sato-Tate group}
Choose an embedding $\iota: \Q_\ell\hookrightarrow \C$ and use it to view $\varrho_{A,\ell}$ as having target $\GSp_{2g}(\C)$.
Put $G^1:=G_\ell^{1,\Zar} \otimes_\iota \C$. The {\em Sato-Tate group} $\ST_A$ of $A$ (for the prime $\ell$ and the embedding
$\iota$) is a maximal compact Lie subgroup of $G^1$ contained in $\USp(2g)$
(which exists because the latter is a maximal subgroup of $\Sp_{2g}(\C)$).
\end{definition}

\begin{remark}
For example, if $A$ is an elliptic curve, then by Serre's open image theorem \cite{Se72},
$\ST_A = \SU(2)$ if and only if $A$ has no complex multiplication.
\end{remark}

\begin{lemma}\label{l2.7}
The groups of connected components of the groups
\[
G_{\ell}^{\Zar}, G_{\ell}^{1,\Zar}, G^1, \ST_A
\]
are all canonically isomorphic.
\end{lemma}
As a corollary, it follows that for $k'$ a finite extension of $k$, $G_{\ell}^{\Zar}(A_{k'})$ is connected
if and only if $G_{\ell}^{1,\Zar}(A_{k'})$ is connected.
\begin{proof}
(from \cite[Theorem~3.4]{BK})
Apply Remark~\ref{R:STextendfield} and its analogue (Remark~\ref{R:STextendfield analogue})
to produce a finite extension $k'$ of $k$ for which $G_{\ell}^{1,\Zar}(A_{k'})$ and $G_{\ell}^{\Zar}(A_{k'})$ are both connected.
Then the diagram
\[
\xymatrix{
  & 1 \ar[d] & 1 \ar[d] & 1 \ar[d] &  \\
1 \ar[r] & G_{\ell}^{1,\Zar}(A_{k'}) \ar[r] \ar[d] & G_{\ell}^{\Zar}(A_{k'}) \ar[r] \ar[d] & \mathbb{G}_m \ar[r]\ar[d]  & 1  \\
1 \ar[r] & G_{\ell}^{1,\Zar}(A_{k}) \ar[r] \ar[d] & G_{\ell}^{\Zar}(A_{k}) \ar[r] \ar[d] & \mathbb{G}_m \ar[r] \ar[d] & 1  \\
  & \pi_0(G_{\ell}^{1,\Zar}) \ar[r] \ar[d] & \pi_0(G_{\ell}^{\Zar}) \ar[d]  & 1  \\
 & 1 & 1
}
\]
has exact rows and columns, and a diagram chase (as in the proof of the snake lemma)
shows that $\pi_0(G_{\ell}^{1,\Zar}) \to \pi_0(G_{\ell}^{\Zar})$ is an isomorphism.
On the other hand,
we have $\pi_0(G_{\ell}^{1,\Zar}) = \pi_0(G^1)$, because the $\Q_\ell$-rational points
of $G_{\ell}^{1,\Zar}$ are Zariski dense,
and $\pi_0(G^1) = \pi_0(\ST_A)$, because any maximal
compact subgroup of a connected complex Lie group is a connected real Lie group.
\end{proof}

We now have the group $\ST_A$ appearing in the refined Sato-Tate conjecture (Conjecture~\ref{Refined Sato-Tate}),
but we still need conjugacy classes corresponding to prime ideals.
This construction will imply that $\overline{L}_{\mathfrak{p}}(A,T)$ belongs to the image of
$\Conj(\ST_A) \to \Conj(\USp(2g))$.
See \cite[\S 8.3.3]{Ser11} for further discussion.
\begin{definition}
For $G =G_\ell^{\Zar} \otimes_\iota \C$,
we may identify $G/G^1$ with $\C^\times$ compatibly with the cyclotomic character. The image of
$g_{\mathfrak{p}} := \varrho_{A,\ell}(\Frob_{\mathfrak{p}}) \in G$ in $\C^\times$ is $q$, so $g'_{\mathfrak{p}} := q^{-1/2} g_{\mathfrak{p}}$ belongs to $G^1$.
The semisimple component of $g'_{\mathfrak{p}}$ for the Jordan decomposition is an element of $G^1$ with eigenvalues of norm $1$, and therefore belongs to a conjugate of $\ST_A$. We thus associate to $\mathfrak{p}$ a class $s(\mathfrak{p}) \in \Conj(\ST_A)$.
\end{definition}

\begin{remark}
It is generally expected that $g_{\mathfrak{p}}$ is already semisimple, in which case
$s(\mathfrak{p})$ would be the class of $g'_{\mathfrak{p}}$ itself.
\end{remark}

\subsection{Mumford-Tate groups and Sato-Tate groups}
\label{subsec:MT}

{}From the above definition of the Sato-Tate group, it is difficult to recover much information relating
the Sato-Tate group to the arithmetic of $A$. To go further, we must control the Sato-Tate group in terms of
the endomorphisms of $A$. This is impossible in general, as shown for instance
by Mumford's examples in dimension 4 \cite{Mu}; however, no such pathologies occur for $g \leq 3$.

\begin{definition}
Fix an embedding $k \hookrightarrow \C$.
\begin{enumerate}
\item[(i)]
The Mumford-Tate group $\MT_A$ of $A$ is the smallest algebraic subgroup $G$ of $\GL(H_1(A_\C^{\topo},\Q))$ {\em over $\Q$} such that $G(\R)$ contains $h(\C^\times)$, where $$h: \C \lra \End_{\R}(H_1(A_\C^{\topo},\R))$$ is the complex structure on the $2g$-dimensional real vector space $H_1(A_\C^{\topo},\R)$ obtained by identifying it with the tangent space of $A$ at the identity.

\item[(ii)] The Hodge group of $A$ is $\Hg_A :=(\MT_A \cap \SL_{2g})^0$.

\end{enumerate}
\end{definition}
These can be viewed as archimedean analogues of the groups $G_\ell^{\Zar}$ and $G_{\ell}^{1,\Zar}$.
By a theorem of Deligne \cite[I, Proposition~6.2]{Del82}, for $G_\ell^{\Zar, 0}$ the connected component of the identity of  $G_\ell^{\Zar}$,
we have $G_\ell^{\Zar,0} \subseteq \MT_A \otimes_{\Q} \Q_\ell$.

\begin{conjecture}[Mumford-Tate]\label{MTconj}
The inclusion $G_\ell^{\Zar,0} \subseteq \MT_A \otimes_{\Q} \Q_\ell$ is always an equality.
Equivalently, the induced inclusion $G_\ell^{1,\Zar,0} \subseteq \Hg_A \otimes_{\Q} \Q_\ell$ is also an equality.
\end{conjecture}

One cannot hope to describe $G_{\ell}^{\Zar}$ or the algebraic Sato-Tate group
any more closely than this using the Mumford-Tate group, because only the connected parts of these groups are
determined by $A_{\C}$.
This suggests the following refinement of the Mumford-Tate conjecture \cite[Conjecture~2.3]{BK}.

\begin{conjecture} \label{algebraic ST conjecture}
There exists an algebraic subgroup $\AST_A$ of $\GSp_{2g}/\Q$,
called the \emph{algebraic Sato-Tate group} of $A$,
such that for each prime $\ell$,
$G_{\ell}^{1,\Zar} = \AST_A \otimes_{\Q} \Q_\ell$.
\end{conjecture}

When Conjecture~\ref{algebraic ST conjecture} holds, we may interpret $\ST_A$ as a maximal compact
subgroup of $\AST_A \otimes_{\Q} \C$. In many cases, including all cases with $g \leq 3$,
one can resolve Conjecture~\ref{algebraic ST conjecture} by giving a full description of both the
Mumford-Tate group and the algebraic Sato-Tate group in terms of the complex structure on $H_1(A^{\topo}_{\C}, \R)$.
This has been carried out by the second author and Grzegorz Banaszak
in \cite{BK}, building on much existing literature on Mumford-Tate groups; we summarize the relevant results here.

\begin{definition}
The \emph{Lefschetz group} $\Lef_A$ is defined as
\[
\Lef_A := \{\gamma \in \Sp_{2g}: \gamma^{-1} \alpha \gamma = \alpha \mbox{ for all $\alpha \in \End(A_{\Qbar})_{\Q}$}\}^0.
\]
Here we view $\alpha$ as an endomorphism of $H_1(A_{\C}^{\topo}, \Q)$; consequently, $\Lef_A$ is an algebraic subgroup of
$\GSp_{2g}/\Q$. There is an obvious inclusion $\Hg_A \subseteq \Lef_A$.
\end{definition}

\begin{definition}
For each $\tau \in G_k$, define
\[
\Lef_A(\tau) := \{\gamma \in \Sp_{2g}: \gamma^{-1} \alpha \gamma = \tau(\alpha) \mbox{ for all $\alpha \in \End(A_{\Qbar})_{\Q}$}\}.
\]
The \emph{twisted Lefschetz group} $\TL_A$ is defined as
\[
\TL_A := \bigcup_{\tau \in G_k} \Lef_A(\tau).
\]
It is an algebraic group over $\Q$ with connected part $\Lef_A$.
\end{definition}

\begin{theorem} \label{mumford-tate genus 2}
\begin{enumerate}
\item[(a)]
Suppose that $\Hg_A = \Lef_A$ and that Conjecture~\ref{MTconj} holds for $A$. Then
Conjecture~\ref{algebraic ST conjecture} holds with $\AST_A = \TL_A$.
Consequently, $\AST_A$ is reductive and $\ST_A$ is a maximal compact subgroup of $\TL_A \otimes_{\Q} \C$.
\item[(b)]
The hypotheses of (a) hold when $g \leq 3$.
\end{enumerate}
\end{theorem}
\begin{proof}
See \cite[Theorem~6.1, Theorem~6.10]{BK}.
\end{proof}

\subsection{Extracting data from the Sato-Tate group}

Using Theorem~\ref{mumford-tate genus 2}, we can recover from $\ST_A$ much data about the endomorphisms of $A$,
starting with the minimal field of definition of endomorphisms.
\begin{proposition}\label{p2.17}
For $g\leq 3$, the component groups of $G_{\ell}^{\Zar}$, $G_{\ell}^{1,\Zar}$, and $\ST_A$
may be identified with $\Gal(K/k)$
for $K/k$ the minimal extension over which all endomorphisms of $A_{\Qbar}$ are defined.
Moreover, for $k'$ a finite extension of $k$, $\ST_{A_{k'}}$ is the inverse image of
$\Gal(Kk'/k')$ under the map $\ST_A \to \Gal(K/k)$.
\end{proposition}
\begin{proof}
Immediate from Theorem~\ref{mumford-tate genus 2} and Lemma~\ref{l2.7}.
\end{proof}

To extract more information, we use the following construction.
\begin{definition}
Recall that we have fixed a polarization on $A$, which defines an isogeny $\phi$ from $A$ to its dual variety $\widehat{A}$. The \emph{Rosati involution}
takes $\psi \in \End(A_{\C})_{\Q}$ to $\psi' = \phi^{-1} \circ \widehat{\psi} \circ \phi$.
It has the following positivity property: the function
\[
\psi \mapsto \Trace(\psi \circ \psi', H_1(A_\C^{\topo}, \Q))
\]
is a positive definite quadratic form on $\End(A_{\C})$ \cite[\S 21, Theorem~1]{Mu}.
We will call this form the \emph{Rosati form}.
\end{definition}

\begin{proposition} \label{Galois type from ST group}
For $g \leq 3$, the subgroup $\ST_A$ of $\USp(2g)$ uniquely determines the $\R$-algebra
$\End(A_K)_{\R}$ and its action by $\Gal(K/k)$.
\end{proposition}
\begin{proof}
Note that
\[
\End(A_K)_{\Q} = \End(A_{\C})_{\Q} = \End(H_1(A_\C^{\topo},\Q))^{\Hg_A}= \End(H_1(A_\C^{\topo},\Q))^{\MT_A}
\]
(see \cite{Mu}). By Theorem~\ref{mumford-tate genus 2}, for $g \leq 3$ we can recover from $\ST_A$ the action of $\Gal(K/k)$ on
$\End(A_K)_{\C}$ by taking the action of $\TL_A/\Lef_A$ on $$(\End(H_1(A_\C^{\topo},\Q)) \otimes_\Q \C)^{\Lef_A}.$$
We can then identify $\End(A_K)_{\R}$ as the unique $\R$-subspace of $\End(A_K)_{\C}$ of half the dimension
which is positive definite for the real part of the Rosati form.
\end{proof}

It will be important later to have an explicit description of the effect of twists on the Sato-Tate group.
\begin{definition}
Let $f: \Gal(L/k) \to \Aut(A_K)$ be a continuous $1$-cocycle, i.e., a function satisfying
\[
f(\tau_1 \tau_2) = f(\tau_1) \tau_1(f(\tau_2)) \qquad (\tau_1, \tau_2 \in G_k)
\]
and factoring through $\Gal(L/k)$ for some finite Galois extension $L$ of $k$ containing $K$.
Then there exists an abelian variety $A^f$ over $k$ equipped with an isomorphism $A^f_L \simeq A_L$ such that the action of
$\tau \in G_k$ on $A^f(\Qbar) \simeq A^f_L(\Qbar)$ corresponds to the action of $f(\tau) \tau$ on $A(\Qbar) \simeq A_L(\Qbar)$.
We call $A^f$ the \emph{twist} of $A$ by $f$.
In case $A$ is the Jacobian of a genus $2$ curve $C$ over $k$ and $f$ factors through a $1$-cocycle
$f_C: \Gal(L/k) \to \Aut(C_K)$, we can identify $A^f$ with the Jacobian of a twist $C^f$ of $C$.

The isomorphism $A^f_L \simeq A_L$ induces an isomorphism $\End(A^f_L) \simeq \End(A_L)$
in which corresponding
$\alpha \in \End(A^f_L)$ and $\beta \in \End(A_L)$ satisfy the relation
$\tau(\alpha) = f(\tau) \tau(\beta) f(\tau)^{-1}$.
Consequently, for $\tau \in G_k$, we have
\begin{align*}
\Lef_{A^f}(\tau) &= \{\gamma \in \Sp_{2g}: \gamma^{-1} \beta \gamma = f(\tau) \tau(\beta) f(\tau)^{-1} \mbox{ for all $\beta \in \End(A_{L})_{\Q}$}\} \\
&= \{\gamma \in \Sp_{2g}: (\gamma f(\tau))^{-1} \beta \gamma f(\tau) = \tau(\beta) \mbox{ for all $\beta \in \End(A_{L})_{\Q}$}\} \\
&= \Lef_A(\tau) f(\tau)^{-1}.
\end{align*}
\end{definition}

\section{A group-theoretic classification}\label{section:STgroups}

In this section, we record some necessary properties of the Sato-Tate group $\ST_A$ of an abelian surface $A$ over a
number field $k$, then classify all closed subgroups of $\USp(4)$ exhibiting these properties; there are 55 such groups
up to conjugacy.
We give explicit representations for each of these groups
that will be used to match the groups with Galois
types (\S\ref{section:Galois}), and to compute invariants that characterize
the distribution of characteristic polynomials (\S\ref{subsection:atlas}).
Note that the classification includes 3 spurious groups; these will be ruled out in \S\ref{s3}.

\subsection{Axioms for Sato-Tate groups}

We first record some necessary conditions for a closed subgroup of $\USp(2g)$ to occur as a Sato-Tate group.
Although we are only interested in the weight 1 case, for future reference we formulate these conditions in a manner
suitable for considering self-dual motives of arbitrary positive weight.
(One can get slightly stronger conditions by accounting for the base field $k$; see Remark~\ref{real place}.)

\begin{definition}
Fix a positive integer $w$ and some nonnegative integers $h^{p,q}$ for all $p,q\geq 0$ with $p+q = w$,
and put $d = \sum_{p,q} h^{p,q}$. If $w$ is odd, assume also that $h^{p,q} = h^{q,p}$ for all $p,q$.
For a group $G$ with identity connected component $G^0$, the
\emph{Sato-Tate axioms} are as follows.
\begin{enumerate}
\item[(ST1)]
The group $G$ is a closed subgroup of $\USp(d)$ (if $w$ is odd) or $\operatorname{O}(d)$ (if $w$ is even).
\item[(ST2)] (Hodge condition)
There exists a homomorphism $\theta: \Unitary(1) \to G^0$
such that $\theta(u)$ has eigenvalues $u^{p-q}$ with multiplicity $h^{p,q}$.
The image of $\theta$ is called a \emph{Hodge circle}.
\item[(ST3)] (Rationality condition)
For each component
$H$ of $G$ and each irreducible character $\chi$ of $\GL_{d}(\C)$,
the expected value (under the Haar measure) of $\chi(\gamma)$ over $\gamma\in H$ is an integer.
In particular, for any positive integers $m$ and $n$, the quantity $\Exp[\Trace(\gamma, \wedge^m \C^{d})^n: \gamma \in H]$ lies in $\Z$.
\end{enumerate}
\end{definition}

The numbers $h^{p,q}$ are meant to be the Hodge numbers of the motive for which one is investigating the Sato-Tate
conjecture. In the case of an abelian variety of dimension $g$, we should thus take $w=1$ and $h^{0,1} = h^{1,0} = g$.

\begin{proposition} \label{necessity of Sato-Tate}
Let $A$ be an abelian variety over $k$ of dimension $g$ satisfying the Mumford-Tate conjecture
(Conjecture~\ref{MTconj})
and the algebraic Sato-Tate conjecture (Conjecture~\ref{algebraic ST conjecture}).
Then $G = \ST_A$ satisfies the Sato-Tate axioms
for $w=1, h^{0,1} = h^{1,0} = g$.
\end{proposition}
\begin{proof}
Condition ST1 is clear from the construction.
Condition ST2 follows from the definition of the Mumford-Tate group;
it can also be derived using $p$-adic Hodge theory \cite[\S 8.2.3.6, \S 8.3.4]{Ser11}.

To check ST3, let $\rho: \GL_{2g}(\C) \rightarrow V_{\C}$ be the representation corresponding to $\chi$.
Since the algebraic Sato-Tate conjecture has been assumed, we may write
$V_{\C} = V \otimes_{\Q} \C$ for $V$ a representation of $\GL_{2g}/\Q$.
Since $\AST_A^0$ is reductive, we may split $V = V_1 \oplus V_2$ so that $\AST_A$ acts on $V_1$ and $V_2$,
$\AST_A^0$ acts trivially on $V_1$, and $V_2$ has no nonzero subquotient on which $\AST_A^0$ acts trivially.
Let $H_{\Q}$ be the component of $\AST_A$ for which $H \subseteq H_{\Q}(\C)$.

Put $t = \Exp[\Trace(\gamma, V_{\C}): \gamma \in H]$.
Since $G^0$ is a maximal compact subgroup of $\AST_A^0$, $V_{2,\C}$ also has no
$G^0$-trivial subrepresentations;
hence for $v_2 \in V_{2,\C}$, the $G^0$-invariant element $\Exp[\rho(\gamma)(v_2): \gamma \in H]$ must be zero.
It follows that $\Exp[\Trace(\gamma, V_{2,\C}): \gamma \in H]=0$.
On the other hand, since $V_1$ is a
trivial $\AST_A^0$-representation, the function $\gamma \mapsto \Trace(\gamma, V_{1,\C})$ is constant on $H_{\Q}(\C)$.
We deduce that $t = \Trace(\gamma_0, V_{1,\C})$ for any single
$\gamma_0 \in H_{\Q}(\C)$.
Since $G/G^0$ is finite, $t$ is a sum of roots of unity and therefore an algebraic integer.

Let $\ell$ be an arbitrary prime and choose an algebraic isomorphism $\Qbar_{\ell} \simeq \C$.
By our hypotheses, $G_{\ell}^{1,\Zar}$ is a compact open subgroup of $\AST_A \otimes_{\Q} \Q_{\ell}$;
let $H_{\ell}$ be the coset of $G_{\ell}^{1,\Zar,0}$ contained in $H_{\Q}(\Q_{\ell})$.
Put $t_{\ell} = \Exp[\Trace(\gamma, V_{\Q_{\ell}}): \gamma \in H_{\ell}]$.
Again, $\Exp[\Trace(\gamma, V_{2,\Q_{\ell}}): \gamma \in H_{\ell}]=0$, so
$t_{\ell} = \Trace(\gamma_0, V_{1,\C})$ for any single
$\gamma_0 \in H_{\Q}(\C)$. In other words, $t_{\ell} = t$.

However, by taking $\gamma_0 \in G_{\ell}^{1,\Zar}$,
we find that $t_{\ell} \in \Q_{\ell}$.
Since $\ell$ and the isomorphism $\Qbar_{\ell} \simeq \C$ were arbitrary and $t$ is an algebraic integer, $t$ must belong to an everywhere unramified number field.
By Minkowski's theorem, we have $t \in \Z$.
\end{proof}

\begin{remark} \label{higher weight}
For any fixed Hodge numbers, the Sato-Tate axioms limit the group $G$
to one of finitely many subgroups of the ambient group $H = \USp(d)$ (if $w$ is odd)
or $H = \operatorname{O}(d)$ (if $w$ is even) up to conjugacy, as follows.

Conditions ST1 and ST2 already suffice to limit the connected part $G^0$ of $G$ to one of finitely many
subgroups of $H$ up to conjugacy, as in the proof of Lemma~\ref{genus 2 connected components};
we may thus fix a choice of $G^0$ hereafter.
Let $N$ denote the normalizer of $G^0$ in $H$;
note that $N$ contains $G$.

Using Tannaka-Krein duality, we may find a representation $\rho:
\GL_d(\C) \to V_{\C}$ whose restriction to $N$ contains a factor on which $G^0$ is trivial
and $N/G^0$ acts faithfully.
Condition ST3 then implies that the exponent of $G/G^0$ may be bounded independently of $G$.
By this fact plus Jordan's theorem on finite linear groups, we may also bound the order of $G/G^0$
independently of $G$.

To finish the argument, it is enough to verify that
any compact Lie group $K$ only contains finitely many conjugacy classes with any given finite order
(and then apply this to $K = N/G^0$).
To see this, note first that if $K'$ is a subgroup of $K$ of finite index, then the claims for $K$ and $K'$
are equivalent.
We may thus reduce to the case where $K$ is the product of a torus with a semisimple Lie group,
in which case one can even compute the exact number of conjugacy classes of a given finite order.
This is trivial for the torus factor; for the semisimple factor, see \cite{Dj}.
\end{remark}

\subsection{Classification in dimension 2: overview}

Our next goal is to establish the following theorem, which will imply Theorem~\ref{group classification}
thanks to Proposition~\ref{necessity of Sato-Tate} and Theorem~\ref{mumford-tate genus 2}.
\begin{theorem} \label{Sato-Tate axioms groups}
Let $G$ be a group which satisfies the Sato-Tate axioms with $w=1$, $h^{0,1} = h^{1,0} = 2$.
Let $G^0$ be the connected part of $G$. Then $G$ is conjugate to one of the groups listed in
Table~\ref{ST-axioms-groups}. (The notation in this table is defined within the proof.)
\begin{table}[ht]
\caption{Groups satisfying the Sato-Tate axioms with $w=1$, $h^{0,1} = h^{1,0} = 2$.}\label{ST-axioms-groups}
\[
\begin{array}{c|c}
G^0 & G \\
\hline\\
\Unitary(1) &
\left\{ \begin{minipage}{8cm}
\begin{center}
$C_{1}, C_{2}, C_{3}, C_{4}, C_{6}, D_{2}, D_{3}, D_{4}, D_{6}, T, O$,
$J(C_1), J(C_2), J(C_3), J(C_4), J(C_6)$, $J(D_2), J(D_3), J(D_4), J(D_6), J(T), J(O)$,
$C_{2,1}, C_{4,1}, C_{6,1}, D_{2,1}, D_{3,2}, D_{4,1}, D_{4,2}, D_{6,1}, D_{6,2}, O_1$
\end{center}
\end{minipage} \right.\\&\\
\SU(2) & E_1, E_2, E_3, E_4, E_6, J(E_1), J(E_2), J(E_3), J(E_4), J(E_6)\\&\\
\Unitary(1) \times \Unitary(1) & F_\nothing, F_a, F_c,
F_{a,b}, F_{ab}, F_{ac}, F_{ab,c}, F_{a,b,c}\\&\\
\Unitary(1) \times \SU(2) & \Unitary(1) \times \SU(2),\medspace N(\Unitary(1) \times \SU(2)) \\&\\
\SU(2) \times \SU(2) & \SU(2) \times \SU(2),\medspace N(\SU(2) \times SU(2))\\&\\
\USp(4) & \USp(4)\\\hline
\end{array}
\]
\end{table}
\end{theorem}

During the course of the classification, we will obtain explicit presentations\footnote{The need for these presentations is the main
reason we do not make more use of the exceptional isomorphism $\USp(4)/\{\pm 1\} \simeq \SO(5)$ in our classification; it leads quickly
to the list of groups but not to these particular presentations.} for each group;
these will be used later to compute Galois types
(see \S\ref{section:Galois}) and moments
(see \S\ref{subsection:atlas}). These computations will imply that no two of the groups listed
in Theorem~\ref{Sato-Tate axioms groups} are conjugate to each other.

\begin{remark} \label{real place}
If the number field $k$ has a real place, one can strengthen the Hodge condition
(see \cite[\S 8.3.4]{Ser11}): there must exist $\gamma \in G$ such that $\gamma^2 = -1$,
$\Trace(\gamma)$ is equal\footnote{This must be modified if one considers weights greater than 1.
See \cite[\S 8.2.3.4]{Ser11}.} to $0$, and $\gamma \theta(u) \gamma^{-1} = \theta(u^{-1})$ for all $u \in \Unitary(1)$.
This extra condition implies that the groups
\[
C_{1}, C_{2}, C_{3}, C_{4}, C_{6}, D_{2}, D_{3}, D_{4}, D_{6}, T, O,
F_\nothing, F_a, F_c, \Unitary(1) \times \SU(2)
\]
cannot occur when $k$ has a real place. We will recover and refine this statement
in \S\ref{section:realizability over Q}
\end{remark}

\begin{remark} \label{shape of symplectic form}
For the purposes of this computation, we will assume that the symplectic form preserved by $\USp(2g)$ is defined by the block matrix
\[
S := \begin{pmatrix} 0 & \Id_2 \\
-\Id_2 & 0
\end{pmatrix}
\]
unless otherwise specified (as is the case in \S\ref{subsec:group odd cases}).
\end{remark}

\subsection{The identity connected component}

The remainder of this section  will be taken up with
the proof of Theorem~\ref{Sato-Tate axioms groups};
we thus assume until the end of \S \ref{section:STgroups} that $w=1$, $g = h^{1,0} = h^{0,1} = 2$.
We begin by enumerating the options for the identity component $G^0$;
this classification is well-known in the context of Mumford-Tate groups.
\begin{lemma} \label{genus 2 connected components}
If $G$ satisfies the Sato-Tate axioms for $w=1$, $g = h^{1,0} = h^{0,1} = 2$, then $G^0$ is conjugate to one of
\[
\Unitary(1),\medspace \SU(2),\medspace \Unitary(1) \times \Unitary(1),\medspace \Unitary(1) \times \SU(2),\medspace \SU(2) \times \SU(2),\medspace \USp(4).
\]
\end{lemma}
\begin{proof}
We may exploit the exceptional isomorphism $\USp(4)/\{\pm 1\} \simeq \SO(5)$.\footnote{Serre kindly pointed out to us that
given an element of $\USp(4)$ with characteristic polynomial $T^4 + a_1 T^3 + a_2 T^2 + a_1 T + 1$,
if $u$ and $v$ are the angles defining the corresponding element of $\SO(5)$, then $a_1^2=4(1+\cos u)(1+\cos v)$ and $a_2 = 2(1 + \cos u + \cos v)$.}
Let $G^0$ be a closed connected subgroup of $\SO(5)$.
Let $T$ be a maximal torus of~$G^0$; it is contained in a maximal torus of $\SO(5)$,
which is 2-dimensional. Let~$\mathfrak{h}$ denote the Lie algebra of $G^0$.
By the classification of Dynkin diagrams, the complexification $\mathfrak{h}_\C$ of
$\mathfrak{h}$ must be isomorphic to one of
\begin{center}
\begin{tabular}{ll}
$\mathfrak{t}_1,\medspace
\mathfrak{sl}_2 = \mathfrak{so}_3$&$ (\dim(T) = 1),$ \\
$\mathfrak{t}_2,\medspace
\mathfrak{t}_1 \times \mathfrak{sl}_2,\medspace
\mathfrak{sl}_2 \times \mathfrak{sl}_2 = \mathfrak{so}_4,\medspace
\mathfrak{sl}_3,\medspace
\mathfrak{so}_5,\medspace
\mathfrak{g}_2$ & $ (\dim(T) = 2)$.
\end{tabular}
\end{center}
The standard representation of $G$ gives rise to 5-dimensional self-dual orthogonal representations of
$\mathfrak{h}$ and $\mathfrak{h}_\C$.
This immediately rules out $\mathfrak{g}_2$, because the smallest dimension of a nontrivial
representation of $\mathfrak{g}_2$ is $7 > 5$.
It also rules out $\mathfrak{sl}_3$, because its only nontrivial representation of dimension at most 5 is
the standard 3-dimensional representation, which is not self-dual.
\end{proof}

\subsection{The case \texorpdfstring{$G^0 = \Unitary(1)$}{G0 = U(1)}}
\label{subsec:unitary1}

We now treat the case $G^0 = \Unitary(1)=\{u\in \C^\times: |u|=1\}$. In this case, $G^0$
must be equal to a Hodge circle, which we may take to be the image of
$\Unitary(1)$ under the homomorphism given in block form by
\[
u \mapsto \smallmat{\mathrm{diag}(u)}00{\mathrm{diag}(u^{-1})}.
\]
Note that the centralizer of $G^0$ within $\GL(4, \C)$ consists of block diagonal matrices
$\begin{pmatrix} A & 0 \\ 0 & D \end{pmatrix}$. For such a matrix to be unitary, we must have
$A,D \in \Unitary(2)$. For such a matrix to also be symplectic, we must have $D = \overline{A}$
(where the bar denotes complex conjugation).
We thus conclude that the centralizer $Z$ of $G^0$ in $\USp(4)$ is isomorphic to $\Unitary(2)$ via the map
\begin{equation} \label{U2 to USp4}
A \mapsto \begin{pmatrix} A & 0 \\ 0 & \overline{A} \end{pmatrix}.
\end{equation}
The normalizer $N$ of $G^0$ in $\USp(4)$ has the form
\[
N = Z \cup JZ, \qquad J := \begin{pmatrix} 0 & J_2 \\ -J_2 & 0 \end{pmatrix}, \qquad
J_2 := \begin{pmatrix} 0 & 1 \\ -1 & 0 \end{pmatrix}.
\]
Note that $J$ centralizes the copy of $\SU(2)$ inside our embedded $\Unitary(2)$: for any $A  = \begin{pmatrix} a & b \\ c & d \end{pmatrix} \in \SU(2)$ we have
\[
A= (\overline{A}^T)^{-1} = (\overline{A}^{-1})^T =
\begin{pmatrix} \overline{d} & -\overline{b} \\
-\overline{c} & \overline{a} \end{pmatrix}^T =
\begin{pmatrix} \overline{d} & -\overline{c} \\
-\overline{b} & \overline{a} \end{pmatrix},
\]
and therefore
\[
J \begin{pmatrix} A & 0 \\ 0 & \overline{A} \end{pmatrix} J^{-1} =
\begin{pmatrix}
J_2 \overline{A} J_2 & 0 \\
0 & J_2 A J_2
\end{pmatrix} = \begin{pmatrix} A & 0 \\ 0 & \overline{A}
\end{pmatrix}.
\]
Consequently, we have
\begin{equation} \label{centralizer of U1}
N/G^0\medspace \simeq \medspace \SU(2) / (\pm 1) \times \Z/2\Z\medspace \simeq\medspace \SO(3)
\times \Z/2\Z.
\end{equation}

We first enumerate the options for $G$ assuming that $G \subseteq Z$;
given \eqref{centralizer of U1}, this constitutes an enumeration over the
familiar list of finite subgroups of $\SO(3)$ up to conjugacy.
It will be convenient to identify $\SU(2)$
with the group of unit quaternions via the isomorphism
\[
a + b \mathbf{i} + c \mathbf{j} + d \mathbf{k} \mapsto \begin{pmatrix}
a + bi & c +di \\ -c+di & a- bi
\end{pmatrix},
\]
and to then use \eqref{U2 to USp4} to view the unit quaternions as a subgroup of $\USp(4)$.

\vspace{6pt}
\textbf{Cyclic groups:} A cyclic group of order $n$ within $\SO(3)$ lifts to a cyclic group of order
$2n$ in $\SU(2)$, which can be represented as $\langle \zeta_{2n} \rangle$, where
$\zeta_{2n} = \cos(\pi/n) + \sin(\pi/n) \mathbf{i}$.
By the rationality condition, the average over $r \in [0, 1]$ of the square of the trace of the matrix
\[
B = \begin{pmatrix} A & 0 \\ 0 & \overline{A} \end{pmatrix}, \qquad A =
e^{2 \pi ir} \zeta_{2n} = \begin{pmatrix} e^{2\pi ir + \pi i/n} & 0 \\ 0 & e^{2\pi ir -\pi i/n} \end{pmatrix}
\]
is an integer. However, this average equals $\left|\Trace(A)\right|^2 = (2 \cos (\pi/n))^2$,
so we must have $2 \cos (\pi/n) \in \{0,\pm 1, \pm \sqrt{2}, \pm \sqrt{3}, \pm 2\}$,
or $n \in \{1, 2, 3, 4, 6\}$. For these values of $n$, let $C_{n}$ denote the resulting subgroup of $Z$.

\vspace{6pt}
\textbf{Dihedral groups:}
A dihedral group of order $2n$ within $\SO(3)$ lifts to the group $\langle \zeta_{2n}, \mathbf{j} \rangle$ in $\SU(2)$.
For $n=2,3,4,6$, let $D_{n}$ denote the resulting subgroup of $Z$. We omit $D_{1}$ since it is conjugate to $C_{2}$.

\vspace{6pt}
\textbf{Tetrahedral group:}
In this case, $G/G^0$ is the tetrahedral group. Lifting to $\SU(2)$ yields the binary tetrahedral group, which has the standard presentation
\[
\{ \pm 1, \pm \mathbf{i}, \pm \mathbf{j}, \pm \mathbf{k},
\frac{1}{2}(\pm 1 \pm \mathbf{i} \pm \mathbf{j} \pm \mathbf{k}) \}.
\]
Let $T$ denote the resulting subgroup of $Z$.

\vspace{6pt}
\textbf{Octahedral group:}
In this case, $G/G^0$ is the octahedral group. Lifting to $\SU(2)$ yields the binary octahedral group, a standard presentation of
which consists of the given presentation of the binary tetrahedral group together with
\[
\frac{\sqrt{2}}{2}(\pm 1 \pm \mathbf{i}),
\frac{\sqrt{2}}{2}(\pm 1 \pm \mathbf{j}),
\frac{\sqrt{2}}{2}(\pm 1 \pm \mathbf{k}),
\frac{\sqrt{2}}{2}(\pm \mathbf{i} \pm \mathbf{j}),
\frac{\sqrt{2}}{2}(\pm \mathbf{i} \pm \mathbf{k}),
\frac{\sqrt{2}}{2}(\pm \mathbf{j} \pm \mathbf{k}).
\]
Let $O$ denote the resulting subgroup of $Z$.

\vspace{6pt}
\textbf{Icosahedral group:} This group cannot occur because it contains a cyclic group of order 5, but $C_5$
does not satisfy the rationality condition.

\vspace{6pt}
We thus obtain the following groups with $G^0 = \Unitary(1)$ and
$G \subseteq Z$:
\begin{equation} \label{eq:so2 groups1}
C_{1}, C_{2}, C_{3}, C_{4}, C_{6}, D_{2}, D_{3}, D_{4}, D_{6}, T, O.
\end{equation}

We now determine the options for $G \not\subseteq Z$.
Since $N/G^0 \simeq \SO(3) \times \Z/2\Z$,
the projection of a subgroup of $N/G^0$ to $\SO(3)$ is either two-to-one or one-to-one onto its image $H$.
In the former case, we get a subgroup of the form $H \times \Z/2\Z$;
these cases correspond to the groups
\[
J(C_1), J(C_2), J(C_3), J(C_4), J(C_6), J(D_2), J(D_3), J(D_4), J(D_6), J(T), J(O),
\]
obtained by adjoining $J$ to each group in \eqref{eq:so2 groups1}.
In the latter case, the subgroup forms the graph of a homomorphism $H \to \Z/2\Z$, which must be nontrivial because
$G \not\subseteq Z$. We need only consider these homomorphisms up to conjugation within
$\SO(3)$; this gives some additional groups as follows.

\vspace{6pt}
\textbf{Cyclic groups:}
For $n=2,4,6$, the nontrivial homomorphism $H \to \Z/2\Z$ gives
\[
C_{n,1} := \langle \Unitary(1), J (\cos(\pi/n) + \sin(\pi/n) \mathbf{i}) \rangle.
\]
Beware: $C_{6,1}$ contains $C_{2,1}$ but $C_{4,1}$ does not (its subgroup of order 2 is $C_2$).

\vspace{6pt}
\textbf{Dihedral groups:}
For $n = 2,4,6$, there are two nontrivial homomorphisms $H \to \Z/2\Z$ not killing the cyclic subgroup,
which are interchanged by an outer automorphism of $H$. These give rise to the group
\[
D_{n,1} := \langle \Unitary(1), J (\cos(\pi/n) + \sin(\pi/n) \mathbf{i}), \mathbf{j} \rangle
\]
containing $C_{n,1}$ with index 2. For $n = 3,4,6$, we may also use the nontrivial homomorphism
killing the cyclic subgroup to obtain
\[
D_{n,2} := \langle \Unitary(1), \cos(\pi/n) + \sin(\pi/n) \mathbf{i}, J\mathbf{j} \rangle;
\]
this is redundant for $n=2$ because all three of the
nontrivial homomorphisms $H \to \Z/2\Z$ are conjugated transitively by $\SO(3)$.

\vspace{6pt}
\textbf{Tetrahedral group:}
In this case, we have $H \simeq \alt 4$, which has no nontrivial homomorphisms to $\Z/2\Z$.

\vspace{6pt}
\textbf{Octahedral group:}
In this case, we have $H \simeq \sym 4$, so there is one nontrivial homomorphism to $\Z/2\Z$
with the tetrahedral group in its kernel. We obtain a new group $O_1$ by multiplying
each of the elements of $O \setminus T	$ by $J$.

\vspace{6pt}
This analysis thus adds the additional groups
\[
C_{2,1}, C_{4,1}, C_{6,1}, D_{2,1}, D_{3,2}, D_{4,1}, D_{4,2}, D_{6,1}, D_{6,2}, O_1.
\]

\subsection{The case \texorpdfstring{$G^0 = \SU(2)$}{G0 = SU(2)}}
\label{subsec:su2}

After the case $G^0 = \Unitary(1)$,
the next most complicated case is $G^0 = \SU(2)$.
Fortunately, we can reuse some of the analysis from \S\ref{subsec:unitary1} as follows.

Embed $G^0$ into $\USp(4)$ as in \eqref{U2 to USp4}.
Since $\SU(2)$ is centralized by $\Unitary(1)$, the normalizer is again $N = Z \cup JZ$.
This time, we see that $Z/G^0 \simeq \Unitary(2)/\SU(2)
\simeq \Unitary(1)/(\pm 1)$, and that conjugation by $J$ acts on $Z/G^0$ by inversion.
We thus need only list the finite subgroups of $\mathrm{O}(2)$, which is straightforward: for each positive
integer $n$, the cyclic subgroup of $\Unitary(1)$ of order $2n$ gives rise to the groups
\begin{align*}
E_n &:= \langle \SU(2), e^{\pi i/n} \rangle \\
J(E_n) &:= \langle \SU(2), e^{\pi i/n}, J \rangle.
\end{align*}
It is easy to check that these groups satisfy the
rationality condition if and only if $n = 1,2,3,4,6$. We thus have the groups
\[
E_1, E_2, E_3, E_4, E_6, J(E_1), J(E_2), J(E_3), J(E_4), J(E_6).
\]

\subsection{The remaining cases for \texorpdfstring{$G^0$}{G0}}
\label{subsec:group odd cases}

In order to complete the proof of Theorem~\ref{Sato-Tate axioms groups},
by Lemma~\ref{genus 2 connected components} it remains to consider the cases
$G^0 = \Unitary(1) \times \Unitary(1), \Unitary(1) \times \SU(2),
\SU(2) \times \SU(2), \USp(4)$.
The case $G^0 = \USp(4)$ is trivial because we must have $G = G^0$,
so we focus on the other three cases. For these cases, it is convenient to change basis to account for the
product structure of $G^0$, by interchanging the second and third
rows and columns; we are thus working with the new symplectic
form
\[
\begin{pmatrix} 0 & 1 & 0 & 0 \\ -1 & 0 & 0 & 0 \\ 0 & 0 & 0 & 1 \\ 0 & 0 & -1 & 0 \end{pmatrix}.
\]
In these coordinates, we may take $G^0$ to be embedded into $\USp(4)$ in block form. We take $\Unitary(1)$
to be embedded into $\SU(2)$ via the map
\[
u \mapsto \begin{pmatrix} u & 0 \\ 0 & \overline{u} \end{pmatrix}.
\]
For $G^0 = \Unitary(1) \times \Unitary(1)$,
the normalizer in $\USp(4)$ contains $\Unitary(1) \times \Unitary(1)$ with index~8, with the quotient
(isomorphic to a dihedral group) generated by
matrices
\[
a := \begin{pmatrix} J_2 & 0 \\ 0 & \Id_2 \end{pmatrix}, \qquad
b := \begin{pmatrix} \Id_2 & 0 \\ 0 & J_2 \end{pmatrix}, \qquad
c := \begin{pmatrix} 0 & \Id_2 \\ -\Id_2 & 0 \end{pmatrix},
\]
each of which defines an involution on the component group. We write $F_{*}$
for the group generated by $G^0$ and a list $*$ of matrices generated by $a,b,c$. In this notation,
up to conjugacy, we have the groups
\[
F_\nothing, F_a, F_c,
F_{a,b}, F_{ab}, F_{ac}, F_{ab,c}, F_{a,b,c}.
\]

For $G^0 = \Unitary(1) \times \SU(2)$, the normalizer in $\USp(4)$ equals $N(\Unitary(1)) \times \SU(2)$.
Thus $G^0$ and its normalizer are the only possible groups.

For $G^0 = \SU(2) \times \SU(2)$, the normalizer in $\USp(4)$ consists of $\SU(2) \times \SU(2)$
plus the coset generated by $J$. In this case, $G^0$ and its normalizer are the only possible groups.
This completes the proof of Theorem~\ref{Sato-Tate axioms groups}.

\section{Galois structures of abelian surfaces}\label{section:Galois}

In this section, we give the classification of \emph{Galois types} of abelian surfaces
(as introduced in Definition~\ref{type-intro})
and the relation of these to Sato-Tate groups. Our main result is Theorem~\ref{main-section4},
which implies both Theorem~\ref{Galois type to ST group} and Theorem~\ref{Galois type classification}.
It gives an alternate description of the Galois type in terms of arithmetic properties of the
abelian surface.  Strictly speaking, only a small part of this description is needed in order to obtain
Theorem~\ref{Galois type to ST group} and Theorem~\ref{Galois type classification}
(namely the analysis of cases corresponding to Sato-Tate groups with connected part $\Unitary(1) \times \Unitary(1)$).
However, we have chosen to provide the complete analysis in order to make it easier to recognize Galois types
and Sato-Tate groups of abelian surfaces occurring in nature.

Before stating Theorem~\ref{main-section4}, we recall the definition of the Galois type and set some associated notation.
For the moment, we take $A$ to be any abelian variety over a number field $k$.
\begin{proposition}\cite{Si}
\label{teo:SiRi}
There is a {\em unique minimal extension} $K/k$ over which all endomorphisms of $A_{\Qbar}$ are defined.
The extension $K/k$ is normal and unramified at the prime ideals
of $k$ at which $A$ has good or semistable reduction.
\end{proposition}

Taking $K$ as in Proposition~\ref{teo:SiRi}, $\End(A_K)_\Q$ is a semisimple algebra of finite rank over $\Q$ and thus decomposes as a product $\End(A_K)_\Q=\prod_i \M_{n_i}(D_i)$ of matrix algebras over division algebras, parallel to the decomposition of $A$ as a product $A\sim \prod_i A_i^{n_i}$ of simple varieties over $K$.

Note that the Galois group $\Gal (K/k)$ acts in a natural way on the $\Q$-algebra $\End(A_K)_\Q$ of endomorphisms of $A$ and induces
a Galois representation
$$
\rhoA: \Gal (K/k)\hookrightarrow \Aut_{\R-\mbox{alg}} (\End(A_K)_\R),
$$
which is faithful precisely because $K/k$ is the minimal extension over which the endomorphisms of $A_{\Qbar}$
are defined.

\begin{definition}\label{Galois-type} The {\em Galois type} of $A$ is the equivalence class of the representation $\rhoA$.
\end{definition}

As noted in Definition~\ref{type-intro},
two abelian varieties $A/k$ and $A'/k'$ defined over different number fields may have the same Galois type; the equivalence relation on representations is meant to see $\Gal(K/k)$ only as an abstract group, not as a quotient of $G_k$.

\subsection{Classification of Galois types: overview}\label{view}

We now restrict to the case where $A$ is an abelian surface over $k$, and formulate the classification theorem for Galois types.
In the process, we introduce alternate names for the Galois types corresponding more closely to their arithmetic.

To begin with, recall that Albert's classification of division algebras with involution (see \cite{Mum70}), together with the work of Shimura \cite{Shi63}, show that the $\R$-algebra $\End(A_K)_\R$ is isomorphic to one of:

\begin{enumerate}

\item[($\bA$)] $\R$, which is the generic case;

\item[($\bB$)] $\R\times \R$, which occurs when either
\begin{list}{\labelitemi}{\leftmargin=1em}
\item $A_K$ is isogenous to a product of nonisogenous elliptic curves without CM, or
\item $A_K$ is simple and $\End(A_K)$ is an order in a real quadratic field;
\end{list}

\item[($\bC$)] $\C\times \R$, which occurs when $A_K$ is isogenous to a product of (necessarily nonisogenous) elliptic curves, one with CM and the other without CM;

\item[($\bD$)]  $\C\times \C $, which occurs when either
\begin{list}{\labelitemi}{\leftmargin=1em}
\item $A_K$ is isogenous to a product of nonisogenous elliptic curves with CM, or
\item $A_K$ is simple and $\End(A_K)$ is an order in a quartic CM-field;
\end{list}

\item[($\bE$)] $\M_2(\R)$, which occurs when either
\begin{list}{\labelitemi}{\leftmargin=1em}
\item $A_K$ is isogenous to the square of an elliptic curve without CM, or
\item $A_K$ is simple and $\End(A_K)$ is an order in a division quaternion algebra over $\Q$;
\end{list}

\item[($\bF$)] $\M_2(\C)$,  which occurs when $A_K$ is isogenous to the square of an elliptic curve with CM.

\end{enumerate}

In case $\bD$, when $\End(A_K)$ is an order in a quartic CM-field $M$ we shall assume that a choice of an isomorphism $\iota: M \overset{\sim}{\lra} \End(A_K)_\Q$ has been made; this singles out a (primitive) CM-type $\Phi$ on $A$, to which we can associate the {\em reflex field} of the pair $(A,\Phi)$, which we denote as usual by $M^*$. Different choices of $\Phi$ give rise to conjugate reflex fields in the Galois closure of $M$, and our results depend only on the conjugacy class of $M^*$.

We shall refer to $\bA$, $\bB$, $\bC$, $\bD$, $\bE$, $\bF$ as the \emph{absolute type}, or simply the \emph{type}, of $A$. Note that the Galois type of $A$ is a much finer invariant.

The six absolute types are in one-to-one correspondence with the six {\em connected} Lie subgroups of $\USp(4)$ appearing in Lemma \ref{genus 2 connected components}, as indicated below.

\begin{center}
\begin{tabular}{llllllll}
$\bA$: & $\USp(4)$            & $\quad$ & $\bB$: &  $\SU(2)\times\SU(2)$     & $\quad$ & $\bC$: & $\Unitary(1)\times\SU(2)$\vspace{2pt}\\
$\bD$: & $\Unitary(1)\times\Unitary(1)$& $\quad$ & $\bE$: &  $\SU(2)$ & $\quad$ & $\bF$: & $\Unitary(1)$\\
\end{tabular}
\end{center}

There at least two ways of proving this.
One method, which we do not make explicit here but surely follows from existing results in the literature, uses Definition \ref{Sato-Tate group} to compute $\ST_A^0=\ST_{A_K}$ for any abelian surface $A$ of given absolute type.
Alternatively, we may work in the reverse direction: for each of the six possible connected Sato-Tate groups, use Proposition \ref{Galois type from ST group} to determine the corresponding Galois type.
These computations are made explicit in sections \S \ref{s1} through \S \ref{s4}.

\begin{theorem}\label{main-section4}
There are exactly $52$ different Galois types of abelian surfaces, and these correspond to $52$ of the $55$ Sato-Tate groups listed in Theorem~\ref{Sato-Tate axioms groups}, as indicated below (using notation defined in the proof).
Of the $52$ Galois types, exactly $34$ can (and do) arise from abelian surfaces defined over $\Q$; these are decorated with the symbol $\star$.

\begin{itemize}
\item $\bA[\cyc 1]^\star $, matching $\USp(4)$.
\item $\bB[\cyc 1]^\star$ and $\bB[\cyc 2]^\star $, matching $\SU(2)\times\SU(2)$ and $N(\SU(2)\times\SU(2))$.
\item $\bC[\cyc 1]$ and $\bC[\cyc 2]^\star$, matching $\Unitary(1)\times\SU(2)$ and $N(\Unitary(1)\times\SU(2))$.
\item $\bD[\cyc 1]$, $\bD[\cyc 2,\R\times \C]$, $\bD[\cyc 2,\R\times \R]$, $\bD[\cyc 4]^\star$, and $\bD[\dih 2]^\star$, matching $F_\nothing$, $F_a$, $F_{ab}$, $F_{ac}$, and $F_{a,b}$.
\item $\bE[\cyc 1]^\star$, $\bE[\cyc 2,\C]^\star$, $\bE[\cyc 3]^\star$, $\bE[\cyc 4]^\star$, and $\bE[\cyc 6]^\star$, matching $E_1$, $E_2$, $E_3$, $E_4$, and $E_6$.
\item $\bE[\cyc 2,\R\times \R]^\star$, $\bE[\dih 2]^\star$, $\bE[\dih 3]^\star$, $\bE[\dih 4]^\star$, and $\bE[\dih {6}]^\star$, matching $J(E_1)$, $J(E_2)$, $J(E_3)$, $J(E_4)$, and $J(E_6)$.
\item $\bF[\cyc 1]$, $\bF[\cyc 2]$, $\bF[\cyc 3]$, $\bF[\cyc 4]$, $\bF[\cyc 6]$, $\bF[\dih 2]$, $\bF[\dih 3]$, $\bF[\dih 4]$, $\bF[\dih 6]$, $\bF[\alt 4]$, and $\bF[\sym 4]$,
      matching $C_1$, $C_2$, $C_3$, $C_4$, $C_6$, $D_2$, $D_3$, $D_4$, $D_6$, $T$, and $O$.
\item $\bF[\cyc 2,\cyc 1,\mathbb H]$, $\bF[\dih 2,\cyc 2,\mathbb H]^\star$, $\bF[\cyc 6,\cyc 3,\mathbb H]$, $\bF[\cyc 4\times\cyc 2, \cyc 4]^\star$, $\bF[\cyc 6\times \cyc 2, \cyc 6]^\star$,
      $\bF[\dih 2\times \cyc 2,\dih 2]^\star$, $\bF[\dih 6,\dih 3,\mathbb H]^\star$, $\bF[\dih 4\times \cyc 2,\dih 4]^\star$, $\bF[\dih 6\times \cyc 2,\dih 6]^\star$,\newline $\bF[\alt 4\times \cyc 2,\alt 4]^\star$, and $\bF[\sym 4\times \cyc 2,\sym 4]^\star$, matching $J(C_1)$, $J(C_2)$, $J(C_3)$, $J(C_4)$, $J(C_6)$, $J(D_2)$, $J(D_3)$, $J(D_4)$, $J(D_6)$, $J(T)$, and $J(O)$.
\item $\bF[\cyc 2,\cyc 1, \M_2(\R)]^\star$, $\bF[\cyc 4,\cyc 2]$, $\bF[\cyc 6,\cyc 3,\M_2(\R)]^\star$, $\bF[\dih 2,\cyc 2,\M_2(\R)]^\star$,\newline $\bF[\dih 4,\dih 2]^\star$, $\bF[\dih 6,\dih 3,\M_2(\R)]^\star$,
	  $\bF[\dih 3,\cyc 3]^\star$, $\bF[\dih 4,\cyc 4]^\star$, $\bF[\dih 6,\cyc 6]^\star$, and $\bF[\sym 4,\alt 4]^\star$, matching $C_{2,1}$, $C_{4,1}$, $C_{6,1}$, $D_{2,1}$, $D_{4,1}$, $D_{6,1}$,
	  $D_{3,2}$, $D_{4,2}$, $D_{6,2}$, and $O_1$.
\end{itemize}

Moreover, for any abelian surface $A$ defined over a number field $k$, each of the following three invariants uniquely determines the other two:
\begin{enumerate}
\item[$(a)$] the conjugacy class of $\ST_A$ within $\USp(4)$;
\item[$(b)$] the Galois type of $A$;
\item[$(c)$] the isomorphism class of $\Gal(K/k)$ plus the function on the subgroup lattice of $\Gal(K/k)$
taking the subgroup $H$ to the isomorphism class of the $\R$-algebra $\End(A_K)_\R^H$ fixed by $H$.
\end{enumerate}
\end{theorem}

We devote the remainder of this section to proving Theorem~\ref{main-section4}. From \S \ref{s1} through \S \ref{s4}, we prove that there exist at most $52$ different Galois types over number fields.
Along the way, we describe the passage from (a) to (b) in Theorem~\ref{main-section4};
this amounts to making the proof of Proposition~\ref{Galois type from ST group} explicit
for each of the 55 groups named in Theorem~\ref{Sato-Tate axioms groups}.
{}From this computation, we see that the Galois types correspond to 52 of the 55 Lie groups listed in Theorem~\ref{Sato-Tate axioms groups}.
Since the Lie groups
$F_c$, $F_{ab,c}$, $F_{a,b,c}$ remain unmatched, they cannot occur as the Sato-Tate group of any abelian surface.

As a byproduct of this computation, we explicitly describe the passage from (a) to (c) in Theorem~\ref{main-section4}.
To do this, it suffices to compute $\End(A_K)^{\Gal(K/k)}_{\R} = \End(A_k)_{\R}$ for each Sato-Tate group; this data appears in Table~\ref{table:STgroups}.
With this data, it is also easy to go from (c) back to (a); see \S \ref{subsection:correspondence}.

In \S \ref{section:realizability over Q}, we verify that the $18$ Galois types that are {\em not} decorated with a $\star$ in Theorem~\ref{main-section4} {\em cannot} arise from an abelian surface defined over $\Q$.
Among these 18 Galois types, 14 correspond to Sato-Tate groups that were already ruled out over a field with a real place in Remark \ref{real place}.
Additional arguments are provided to show that the 3 Galois types $\bF[\cyc 2,\cyc 1,\mathbb H]$, $\bF[\cyc 4,\cyc 2]$ and $\bF[\cyc 6,\cyc 3,\mathbb H]$ cannot occur over a field with a real place  (Proposition~\ref{proposition:no real place}). Finally, we show that the Galois type $\bD[\cyc 2, \R\times\R]$, admissible over a field admitting a real place, cannot arise from an abelian surface defined over $\Q$, as a result of the discussion in \S \ref{s3}.

In \S \ref{section:examples}, we exhibit one proven example of an abelian surface for each of the $52$ Galois types.
These examples arise as Jacobians of curves of genus 2 over number fields;
for those $34$ Galois types decorated with a $\star$, the curve that we exhibit is defined over $\Q$. For the Galois type $\bD[\cyc 2, \R\times\R]$, the curve we present is defined over a totally real field.

We now proceed with the proof of Theorem 4.3 in \S \ref{s1}--\S \ref{subsection:correspondence}.

\subsection{Cases \texorpdfstring{$\bA$}{A} and \texorpdfstring{$\bB$}{B}}\label{s1}

In case $\bA$, the group $\Aut_{\Q} (\End(A_K)_\Q)$ is trivial and $\End(A_K)_\R=\R$.
Therefore $\Gal(K/k)=\cyc 1$ and $\rhoA=\chi_1$ is the trivial representation.
In case $\bB$, we have $\Aut_{\Q} (\End(A_K)_\R)\simeq \cyc 2$, $\End(A_K)_\R\simeq \R \times \R$, and $\End(A_K)_\R^{\cyc 2}\simeq \R$.
This yields three distinct Galois types: $\bA[\cyc 1]$, $\bB[\cyc 1]$, and $\bB[\cyc 2]$.

{}From the Lie group side, it is clear that if $\ST_A = \USp(4)$, then $\End(A_K)_{\R}$ is $\R$.
If $\ST_A^0 = \SU(2) \times \SU(2)$, then $\End(A_K)_{\R} \simeq \R \times \R$ and the normalizer of $\SU(2) \times \SU(2)$ interchanges the two factors, fixing $\R$.

Thus the Galois types of abelian surfaces with Sato-Tate groups $\USp(4)$, $ \SU(2) \times \SU(2)$, and $N(\SU(2) \times \SU(2))$ are $\bA[\cyc 1]$, $\bB[\cyc 1]$, and $\bB[\cyc 2]$, respectively.

\medskip

\subsection{Complex multiplication by a quartic CM-field}\label{s2}

Let $M$ be  a quartic CM-field, that is, a totally imaginary quadratic extension of a real quadratic field. Assume $A$ has complex multiplication by $M$ and fix an isomorphism $\iota: M \lra \End(A_K)_\Q$, which in turn induces an isomorphism
$$
M\otimes_\Q \R\simeq \C\times \C \simeq \End(A_K)_\R.
$$

Let $\Phi=\{\phi, \phi'\}$ denote the CM-type of the pair $(A,\iota)$, which is necessarily primitive as otherwise $\iota$ would only be a (nonsurjective) monomorphism. Let $M^*$ be the reflex field of $(M,\Phi)$. The extension $M/\Q$ is either

\begin{enumerate}

\item[($\cyc 4$)] normal, with $\Gal(M/\Q)\simeq \cyc 4$ and $M^*= M$, or

\item[($\dih 4$)] not normal, with the Galois group of the normal closure $\tilde M$ isomorphic to $\dih 4$ and $M^*$ a subfield of $\tilde M$ of degree $4$ over $\Q$, different from $M$ and not normal over $\Q$; see \cite[page 64]{Shi98} or \cite[Ch.\,I, \S 7]{Str}, for example.

\end{enumerate}

\begin{proposition} The field $K$ is the compositum of $k$ and $M^*$.
\end{proposition}

\begin{proof} This is Proposition 3 in \cite[page 515]{Shi71}.
\end{proof}

In case ($\dih 4$) we must have $|\Gal(K/k)|\leq 2$ because otherwise $\Aut_{\Q}(M)\supseteq \Gal(K/k)$ would have order at least $4$, implying that $M/\Q$ is Galois. Note that the condition $[kM^*:k]\leq 2$ implies that this case can not occur for $k=\Q$. In case ($\cyc 4$), we have $\Gal(K/k)=\Gal(kM^*/k)\subseteq \Gal(M^*/\Q)\simeq \cyc 4$.
In any case, $\Gal(K/k)\simeq \cyc n$ for $n=1$, $2$, or $4$, which gives rise to the following alternatives:

\begin{itemize}

\item $\Gal(K/k)= \cyc 1$ and $\rhoA$ is the trivial representation; this yields Galois type $\bD[\cyc 1]$.

\item $\Gal(K/k)= \cyc 2$ and $(\End(A_K)_\R)^{\cyc 2}\simeq\R\times \R$, as  $(\End(A_K)_\R)^{\cyc 2}\simeq M^{\cyc 2}\otimes_\Q \R$ is by Artin's Lemma an $\R$-vector space of dimension 2 and  $M$ contains a single quadratic subfield, which is real; this is Galois type $\bD[\cyc 2,\R\times \R]$.

\item $\Gal(K/k)= \cyc 4$, $(\End(A_K)_\R)^{\cyc 2} \simeq \R\times \R$ and $(\End(A_K)_\R)^{\cyc 4} \simeq \R$; this is Galois type $\bD[C_4]$, which is the only case that can occur when $k=\Q$.

\end{itemize}

Thus for case $\bD$ we have found three of the five Galois types listed in Theorem~\ref{main-section4}; we shall find the remaining two in the next section, arising from abelian surfaces that are isogenous to the product of two elliptic curves with CM by distinct imaginary quadratic fields.

Let us now analyze which Galois types correspond to an abelian surface $A$ such that $\ST_A^0=\Unitary(1) \times \Unitary(1)$.
Recall that in this case, we take the symplectic form to be in split form rather than block form.
With this in mind, the matrices in $\M_4(\C)$ commuting with $\Unitary(1) \times \Unitary(1)$ are
\[
\left\{ \begin{pmatrix}
a & 0 & 0 & 0 \\
0 & b & 0 & 0 \\
0 & 0 & c & 0 \\
0 & 0 & 0 & d
\end{pmatrix}: a,b,c,d \in \C\right\}
\]
and the Rosati form is a scalar multiple of $2ab  + 2cd  = \frac{1}{2}((a+b)^2 - (a-b)^2 +(c+d)^2 - (c-d)^2)$. Consequently,
\[
\End(A_K)_{\R} =
\left\{ \begin{pmatrix}
a+bi & 0 & 0 & 0 \\
0 & a-bi & 0 & 0 \\
0 & 0 & c+di & 0 \\
0 & 0 & 0 & c-di
\end{pmatrix}: a,b,c,d \in \R \right\} \simeq \C \times \C.
\]
The action of a generator of the component group of $F_a$ is
\[
\begin{pmatrix}
a+bi & 0 & 0 & 0 \\
0 & a-bi & 0 & 0 \\
0 & 0 & c+di & 0 \\
0 & 0 & 0 & c-di
\end{pmatrix} \mapsto
\begin{pmatrix}
a-bi & 0 & 0 & 0 \\
0 & a+bi & 0 & 0 \\
0 & 0 & c+di & 0 \\
0 & 0 & 0 & c-di
\end{pmatrix},
\]
so the fixed ring has $b=0$ and thus is $\R \times \C$. For $F_c$, we get
\[
\begin{pmatrix}
a+bi & 0 & 0 & 0 \\
0 & a-bi & 0 & 0 \\
0 & 0 & c+di & 0 \\
0 & 0 & 0 & c-di
\end{pmatrix} \mapsto
\begin{pmatrix}
c+di & 0 & 0 & 0 \\
0 & c-di & 0 & 0 \\
0 & 0 & a+bi & 0 \\
0 & 0 & 0 & a-bi
\end{pmatrix},
\]
so the fixed ring has $a=c, b=d$ and thus is $\C$. For both $F_{ab}$ and $F_{a,b}$, the fixed ring has $b=d=0$
and thus is $\R \times \R$. For $F_{ac}$, the fixed ring has $a=c$ and $b=d=0$ and thus is $\R$, and similarly
for the larger group $F_{a,b,c}$. For $F_{ab,c}$, the fixed ring is contained in both $\R \times \R$ (the fixed ring
of $F_{ab}$) and $\C$ (the fixed ring of $F_c$) and thus is $\R$.

Comparing this analysis with the Galois structure of the above three Galois types, we conclude that if the Sato-Tate group of $A$ is $F_\nothing$, $F_{ab}$, or $F_{ac}$, then the Galois type of $A$ is $\bD[\cyc 1]$, $\bD[\cyc 2,\R\times \R]$, or $\bD[\cyc 4]$, respectively.

\subsection{Products of nonisogenous elliptic curves}\label{s3}

\begin{proposition}\label{M1M2} For $i=1,2$, let $M_i$ be either $\Q$ or an imaginary quadratic field. Assume that at least one $M_i$ is quadratic and that, if both $M_1$ and $M_2$ are quadratic, then $M_1\not \simeq M_2$.
Let $A/k$ be an abelian surface such that $\End(A_{\Qbar})_{\Q} \simeq M_1\times M_2$.  The following hold:
\begin{enumerate}[$(i)$]
\item the minimal extension of $k$ over which the endomorphisms of $A_{\Qbar}$ are defined is $K=k M_1 M_2$;
\item there exist elliptic curves $\tilde E_1$, $\tilde E_2$ over $k$ for which $A\sim_k \tilde E_1\times \tilde E_2$.
\end{enumerate}
\end{proposition}

\begin{proof} $(i)$ Let us prove the statement under the assumption that both $M_1$ and $M_2$ are quadratic; when $M_1$ is quadratic and $M_2=\Q$, the proof is simpler and we leave the details to the reader.

Let $K\subset \overline\Q$ denote the minimal extension of $k$ over which all endomorphisms of $A_{\Qbar}$ are defined, and let $\Omega/k$ denote a minimal subextension of $K/k$ over which there exists an isogeny
\[
\psi: A \overset{\sim }\lra E_1\times E_2,
\]
defined over $\Omega$ onto a product of two elliptic curves $E_1/\Omega$ and $E_2/\Omega$.
Note that $\Omega$ might be properly contained in $K$, as we do not require that $\End(E_{i,\Omega})_{\Q} \simeq M_i$.

For $i=1,2$, we claim that there exists an elliptic curve $E'_i$ over $k$ such that $E_i$ and $E_i'$ are isogenous over $\Omega$.
Indeed, let $k\subseteq k_i \subseteq \Omega$ be a minimal subextension of $\Omega/k$ over which   such an $E_i'$ exists.
The abelian surface $A':=E'_1\times E'_2$ is thus defined over $k_1 k_2$. As an application of \cite[Thm.~8.2]{Ri}, it follows from the minimality of $k_i$ that $k_1 k_2$ is also a minimal subextension of $\Omega/k$ over which $A'$ admits a model up to isogenies over $\Omega$.
Indeed, if there were a proper subextension $k_0\subsetneq k_1 k_2$ over which $A'$ admits a model, there would exist a collection of isogenies $\{ \varphi_\sigma: (A')^\sigma \lra A'\}_{\sigma \in \Gal(k_1k_2/k_0)}$ defined over $\Omega$ such that $\varphi_\sigma {}^\sigma \varphi_\tau = \varphi_{\sigma \tau}$.
Since there are no isogenies between $E_1$ and any of the Galois conjugates of $E_2$, we would have $\varphi_\sigma = (\varphi^1_\sigma,\varphi^2_\sigma)$ where $\varphi^i_\sigma: E^{\sigma}_i \lra E_i$ are isogenies such that $\varphi^i_\sigma {}^\sigma \varphi^i_\tau = \varphi^i_{\sigma \tau}$.
Ribet's theorem would then imply that both $E_1$ and $E_2$ admit a model over $k_0$, contradicting the minimality of $k_1$ and $k_2$.

Since $A$ is one such model, we deduce that $k=k_1 k_2$ and thus $k=k_1=k_2$. Hence $A$ and $A'$ are abelian surfaces over $k$ that are isogenous over $\Omega$.

By the theory of complex multiplication on elliptic curves (see for example \cite[Thm.~2.2]{Sil}), the minimal extension of $k$ over which all endomorphisms of $A'_{\Qbar}$ are defined is $k M_1 M_2$.
Since $\End(A'_{\Qbar})_\Q=M_1\times M_2$ is commutative and $A$ is a twist of $A'$ in the category of abelian varieties up to isogenies, the isogeny class of $A$ over $k$ corresponds to a cocycle $c_A\in H^1(\Gal(\overline\Q/k), \End(A'_{\Qbar})_\Q^\times)=H^1(\Gal(\Qbar/k), M_1^\times \times M_2^\times)$; it follows that for any number field $k\subseteq F\subseteq \overline\Q$,
\begin{align*}
\End(A_F)_\Q &= \{ \alpha \in \End(A'_{\Qbar})_\Q: \alpha^\sigma c_A(\sigma ) = c_A(\sigma )\alpha, \,\, \forall \sigma \in G_F\}
\\
&=  \End(A'_F)_\Q \quad \text{(because $M_1\times M_2$ is commutative).}
\end{align*}
Hence the minimal extension of $k$ over which all endomorphisms of $A_{\Qbar}$ are defined is also $K=k M_1 M_2$.

$(ii)$ In the proof of $(i)$, we have seen that $A\sim_K E_1'\times E_2'$ for elliptic curves $E_1'$, $E_2'$ over $k$. Since $\Gal(K/k)=\cyc 1,\cyc 2, \dih 2$, the representation $\Hom(A_K,E_{1,K}')_\Q$ decomposes as a sum of characters of order at most $2$. Let $\chi$ denote any of these characters. Then $\Hom(A_K,(E_{1}'\otimes\chi)_K)_\Q$ contains the trivial representation and thus $E_{1}'\otimes\chi$ is a $k$-factor of $A$. This induces a decomposition $A\sim_k (E_1'\otimes\chi)\times \tilde E_2$, for some elliptic curve $\tilde E_2/k$. We may then take $\tilde E_1=E_1'\otimes\chi$.
\end{proof}

In case $\bC$, let $A\sim_K E_1\times E_2$, where $E_1$ and $E_2$ are elliptic curves defined over $k$ with CM by $M$ and without CM, respectively. Two cases arise:
\begin{enumerate}[(i)]
\item $M\subseteq k$. Then $\Gal(K/k)=\cyc 1$ and $\rhoA$ is the trivial representation. This is Galois type $\bC[\cyc 1]$, and it cannot occur when $k=\Q$.

\item $M$ is not contained in $k$. Then $\Gal(K/k)=\cyc 2$ and $\Trace \rhoA = 2 \chi_1+\chi_2$. This is Galois type $\bC[\cyc 2]$, which can occur when $k=\Q$.

\end{enumerate}

{}From the Lie group side, it is easy to check that if $\ST_A^0= \Unitary(1) \times \SU(2)$, then $\End(A_K)_{\R} = \R \times \C$ and the normalizer acts nontrivially on $\C$,
fixing $\R \times \R$. This shows that the Galois types corresponding to the Lie groups $\Unitary(1) \times \SU(2)$ and $N(\Unitary(1) \times \SU(2))$ are respectively $\bC[\cyc 1]$ and $\bC[\cyc 2]$.

\vspace{0.3cm}

In case $\bD$, let $A\sim_K E_1\times E_2$, where $E_1$ and $E_2$ are  nonisogenous elliptic curves defined over $k$ with CM by two different imaginary quadratic fields $M_1$ and $M_2$, respectively. Then four cases arise:

\begin{enumerate}[(i)]

\item $M_1,M_2\subseteq k$. Then $\Gal(K/k)=\cyc 1$ and $\rhoA$ is the trivial representation; this is Galois type $\bD[\cyc 1]$, which we already encountered in \S \ref{s2}.

\item $M_1\subseteq k$ and $M_2$ is not contained in $k$. Then $\Gal(K/k)=\cyc 2$ and $(\End(A_K)_\Q)^{\cyc 2}=M_1\times \Q$, thus $(\End(A_K)_\R)^{\cyc 2} \simeq \R\times\C$.
This is Galois type $\bD[\cyc 2,\R\times\C]$.

\item $M_1$ and $M_2$ are not contained in $k$ and $kM_1=kM_2$. Then $\Gal(K/k)=\cyc 2$ and $(\End(A_K)_\Q)^{\cyc 2}=\Q\times\Q$; we thus have $(\End(A_K)_\R)^{\cyc 2} \simeq \R\times \R$, yielding the Galois type $\bD[\cyc 2,\R\times \R]$ that we already met in \S \ref{s2}.

\item $M_1$ and $M_2$ are not contained in $k$ and $kM_1\not=kM_2$. Then $\Gal(K/k)\simeq \dih 2$ and the three subalgebras of $\End(A_K)_\Q$ fixed by each of the subgroups of order $2$ are $M_1\times\Q$, $M_2\times\Q$ and $\Q\times\Q$. This is Galois type $\bD[\dih 2]$.

\end{enumerate}

Among the four Galois types listed above, only the last can occur when $k=\Q$. The analysis in \S \ref{s2} implies that the Galois types corresponding to the Lie groups $F_a$ and $F_{a,b}$ are respectively $\bD[\cyc 2,\R\times\C]$ and $\bD[\dih 2]$.

As a byproduct, since we have now classified all the possible Galois types of an abelian surface $A$ for which $\ST_A^0=\Unitary(1) \times \Unitary(1)$, we deduce that the Lie groups $F_c$, $F_{ab,c}$, $F_{a,b,c}$ cannot occur as the Sato-Tate group of an abelian surface.

\subsection{Products of isogenous elliptic curves}\label{s4}

In case $\bE$ or $\bF$, the endomorphism ring $\End(A_K)_\Q$ is a quaternion algebra $B$ over $C=\Q$ or $M$, respectively. Write $B\ra B$, $\alpha \mapsto \alpha'$ for the canonical anti-involution on $B$, and $n(\alpha)=\alpha\alpha'\in C$.

\begin{proposition}\label{Minclos}
Assume that $A/k$ is of type $\bF$. Then $K$ contains $M$ and $\Gal(K/kM)$ acts trivially on the center of $\End(A_K)_\Q\simeq \M_2(M)$.
\end{proposition}

\begin{proof} Let $E/K$ be an elliptic curve such that $\End(E_K)_{\Q}\simeq M$ and $A\sim_K E^2$. Since any field of definition of the endomorphisms of $E$ contains $M$ (cf.\,e.g.\,\cite[Ch. II, Thm. 2.2]{Sil}), we have $M\subseteq K$. Fix embeddings $\iota: M\subseteq K \subset \C$.

As shown in \cite[Ch.~II, Prop.~1.1]{Sil}, there exists a unique isomorphism $[\,\,]_E: M \lra \End(E_K)_{\Q}$ such that for any invariant differential $\omega \in \Omega_E^1$, we have $[\alpha]_E^*(\omega) = \alpha \omega$ for all $\alpha \in M\subset \C$.
There is therefore an unique isomorphism $[\,\,]_{E^2}: \M_2(M) \lra \End(E_K^2)_{\Q}$ such that for any invariant differential $\omega = (\omega_1,\omega_2) \in \Omega_{E^2}^1$, we have $[\alpha]_{E^2}^*(\omega) = \alpha \omega$ for all $\alpha \in \M_2(M)\subset \M_2(\C)$. It follows as in the proof of \cite[Ch.~II, Thm.~2.2]{Sil} that ${}^\sigma[\alpha]_{E^2} = [{}^\sigma \alpha]_{({}^{\sigma}E)^2}$, hence
\begin{equation}\label{eq1}
{}^\sigma[\alpha]_{E^2} = [\alpha]_{({}^{\sigma}E)^2} \qquad \text{for $\sigma \in \Gal(K/k M)$}.
\end{equation}

Endomorphisms of $E^2$ induce $K$-linear endomorphisms of the space of regular differentials $\Omega^1_{E^2}$; any choice of a $K$-basis of $\Omega^1_{E^2}$ gives rise to a monomorphism $t:\M_2(M)\simeq \End(E^2_K)_{\Q}\hookrightarrow \M_2(K)\hookrightarrow \M_2(\C)$, $\alpha \mapsto [\alpha]^*_{E^2}$ whose restriction to the centers is $\iota $.

That $E^2$ admits the model $A$ over $k$ up to isogenies over $K$ implies, thanks again to Ribet's theorem \cite[Thm. 8.2]{Ri}, that there exists a collection of isogenies $\{\Phi_\sigma: ({}^{\sigma}E)^2 \lra E^2 \}_{\sigma \in \Gal(K/k)}$  such that $\phi_\sigma {}^\sigma \phi_\tau = \phi_{\sigma \tau}$.
The $\Gal(K/k)$-module $\End(A_K)_\Q$ is then isomorphic to the module $\End(E^2_K)_{\Q}$ equipped with the following action of the group $\Gal(K/k)$: an element $\sigma \in \Gal(K/k)$ acts on an endomorphism $[\alpha ]\in \End(E^2_K)_{\Q}$ by the rule $\sigma \cdot [\alpha ] = \phi_\sigma [{}^\sigma \alpha] \phi_\sigma^{-1}$.

Similarly, isogenies $\phi_\sigma$ induce $K$-linear isomorphisms $\pi_\sigma^*: \Omega^1_{({}^{\sigma}E)^2}\simeq \Omega^1_{E^2}$.
If $[\alpha]$ lies in the center of $\End(E^2_K)_{\Q}$, then $\sigma \cdot [\alpha ]^* = [{}^\sigma \alpha]^*$; if in addition $\sigma \in \Gal(K/kM)$, it follows from \eqref{eq1} that  $\sigma \cdot [\alpha ]^* = [\alpha]^*$. Since $t$ is a monomorphism, we deduce that $\sigma \cdot [\alpha ] = [\alpha]$; thus $\Gal(K/kM)$ acts trivially on the center of $\End(A_K)_\Q\simeq \M_2(M)$, as claimed.
\end{proof}

By the Skolem-Noether theorem, all automorphisms of $\End(A_K)_\Q$ that are the identity on $C$ are inner. Set $\mathbb P(\End(A_K)_\Q^\times)=\End(A_K)_\Q^\times/C^\times$.
A homothety class $[\alpha]\in \mathbb P(\End(A_K)_\Q^\times)$ induces the automorphism $c_\alpha$ of $\End(A_K)_\Q$ given by the rule $\gamma \mapsto \alpha \gamma \alpha^{-1}$, which is the identity on $C$.
This induces an isomorphism $\Aut_{C}(\End(A_K)_\Q)\simeq \mathbb P(\End(A_K)_\Q^\times)$.

By the previous proposition, $\Gal(K/k C)$ is isomorphic to a subgroup of $\Aut_{C}(\End(A_K)_\Q)$.
The list of finite subgroups of $\mathbb P(\End(A_K)_\Q^\times)$ is well-known (see \cite{Be,ChFr}, for example), and this allows us to conclude that $\Gal(K/k C)$ is isomorphic to one of $\cyc n$, $\dih {n}$, $\alt 4$ or $\sym 4$, where $n\in \{ 1,2,3,4,6\}$.
The groups $\alt 4$ or $\sym 4$ arise only when $-1$ can be written as a sum of two squares in $C$, and thus only occur in case $\bF$.

We make now a first analysis of the action of $\Gal(K/k C)$ on $\End(A_K)_\Q$, and therefore also on $\End(A_K)_\R$.
This will address all the Galois types in case $\bE$ and is a first step towards classifying the Galois types in case $\bF$, which we will conclude in \S \ref{secF}.

\begin{proposition}
\label{ChFr}

Let $A/k$ be an abelian surface that is isogenous over $K$ to the square of an elliptic curve.
\begin{enumerate}[$(i)$]

\item Assume $\Gal(K/k C)\simeq \cyc n$ for $n=1$, $2$, $3$, $4$ or $6$.

If $n=1$, then $\Gal(K/k C)$ acts trivially on $\End(A_K)_\Q$; if $C=\Q$ we denote the resulting Galois type $\bE[\cyc 1]$.

If $n=2$, then $\End(A_K)_\Q^{\Gal(K/k C)}$ is a quadratic extension of $C$. If $C=\Q$, this gives rise to two Galois types, according to whether the extension is real or imaginary; we label them $\bE[\cyc 2,\R\times \R]$ and $\bE[\cyc 2,\C]$, respectively.

If $n> 2$, then $\End(A_K)_\Q^{\Gal(K/k C)}=C$ and for any nontrivial subgroup $H\subseteq \Gal(K/k C)$ we have
$$
\End(A_K)_\Q^{H} = C+C\cdot (1+\zeta_n), \quad \zeta_n^n=1, \zeta_n^{n_0}\ne 1 \mbox{ for } n_0<n.
$$

If $C=\Q$, for any $k\subseteq k'\subsetneq K$ we have $\End(A_{k'})_{\Q}=\Q(\zeta_n)$, which is an imaginary quadratic extension of $\Q$. This gives rise to a single Galois type, which we denote $\bE[\cyc n]$.

\item Assume that $\Gal(K/k C)\simeq \dih 2=\cyc 2\times \cyc 2$ and $\End(A_K)_\Q^{\Gal(K/k C)}=C$, and that the three algebras fixed by each of the subgroups of $\Gal(K/k C)$ of order $2$ are quadratic extensions of $C$.

If $C=\Q$, then two of these quadratic extensions are real and one is imaginary; hence a single Galois type arises, which we denote $\bE[\dih 2]$.

\item Assume that $\Gal(K/k C)\simeq \dih {n}$ for $n=3$, $4$ or $6$. Write $\Gal(K/k C) = \langle r, s\rangle $, with $r^n=1$, $s^2=1$, and $s r s = r^{-1}$.

Then $\End(A_K)_\Q^{\langle s\rangle } = C+C\cdot \sqrt{m}$ for some $m\in C$ and, for any nontrivial subgroup $H\subseteq \langle r\rangle \subset \Gal(K/k C)$, we have $\End(A_K)_\Q^{H} = C+C\cdot (1+\zeta_n)$.

The algebra fixed by any other nontrivial subgroup of $\Gal(K/k C)$ is $C$.

If $C=\Q$, then $\Q(\sqrt{m})$ is real quadratic and $\Q(\zeta_n)$ is imaginary quadratic; we thus obtain the single Galois type $\bE[\dih n]$.

\end{enumerate}
\end{proposition}

We have relegated the study of the case $C=M$ and $\Gal(K/k M)\simeq \alt 4$ or $\sym 4$ to \S \ref{secF}: see Proposition \ref{cesc} and the discussion following it.

\begin{proof} Write $G=\Gal(K/k C)$. Let us first consider case ($i$), in which $\rhoA$ induces an isomorphism between $G$ and a cyclic subgroup of $B^\times/C^\times$ of order $n$.
More precisely, $G=\langle c_\alpha \rangle$ for some $\alpha\in B^\times \setminus C^\times$ such that $\alpha^2=d\in C^\times$ if $n=2$, or $\alpha=1+\zeta_n$ if $n>2$, where $\zeta_n$ is an element in $B^\times \setminus C^\times$ of order $n$.
One checks that $\alpha^{n_0}\not \in C$ for $n_0<n$, thus the subalgebra of $B$ fixed by any of the nontrivial subgroups of $G$ is precisely $C(\alpha)$, which is quadratic over $C$.

For ($ii$), we now assume that $G\simeq \dih {2}$.
By \cite[Lemma 2.3]{ChFr}, any subgroup of $B^\times/C^\times$ isomorphic to $\dih 2$ is of
the form $\langle [\alpha], [\beta]\rangle \subset B^\times/C^\times$ with $\alpha, \beta \in B^\times\setminus C^\times$ satisfying $\alpha^2=d$, $\beta^2=m\in C^\times$ and $\alpha \beta = -\beta \alpha$.
In particular, $(\alpha \beta)^2= -dm$. It follows that the subalgebra of $B$ fixed by $G$ (resp.\ by each of the three subgroups of $G$ of order $2$) is the center $C$ (resp.\ the quadratic extension $C(\alpha)$, $C(\beta)$, $C(\alpha \beta)$ of $C$). If $C=\Q$, we know that $B$ is indefinite, that is to say
$$
B\otimes_\Q \R = \left(\frac{d,m}{\R}\right)\simeq \M_2(\R),
$$
which amounts to saying that at least one of $d$, $m$ is positive. This implies that exactly two of $d$, $m$, $-dm$ are positive and one is negative.

Similarly, in case ($iii$), by \cite[Lemma 2.3]{ChFr} all subgroups of $B^\times/C^\times$ isomorphic to $\dih {n}$
are of the form $\langle [\alpha], [\beta ]\rangle$, where $\alpha=1+\zeta _n$ and $\beta \in B^\times\setminus C^\times$ satisfy $\beta ^2=m\in C^\times$ and $\zeta_n \beta = \beta\overline{\zeta}_n$. The subalgebra of $B$ fixed by any of the nontrivial subgroups of $\langle [1+\zeta _n]\rangle$ is $C(\zeta_n)$; the subalgebra fixed by $\langle [\beta ]\rangle$ is $C(\beta)$.
If $C=\Q$, then $\Q(\zeta_n)$ is either $\Q(\sqrt{-1})$ or $\Q(\sqrt{-3})$, which are both imaginary, and we necessarily have $m>0$ because $B$ is indefinite.
All the claims in ($iii$) follow.
\end{proof}

\begin{remark}
{}From the proof above, one also deduces that in case $\bE$ the $\Gal(K/k)$-module structure of $\End(A_K)_\R$ is given by the rule
$$\Trace \rho_A(\sigma)=2+\zeta_r+\overline \zeta_r\,,$$
where $r$ is the order of $\sigma \in \Gal(K/k)$ and $\zeta_r$ denotes a primitive $r$th root of unity.
Thus in case $\bE$, the $\Gal(K/k)$-module structure of  $\End(A_K)_\R$ is completely determined by $\Gal(K/k)$ (compare this result with Proposition \ref{cesc}).
\end{remark}

\subsubsection{Galois types and Sato-Tate groups in case \texorpdfstring{$\bE$}{E}}\label{secE}

For case $\bE$ we have found a total of ten Galois types: $\bE[\cyc n]$ for $n=1$, $3$, $4$, $6$; $\bE[\cyc 2,\R \times \R]$ and $\bE[\cyc 2,\C]$; and $\bE[\dih n]$ for $n=2$, $3$, $4$, $6$.

Let us recover this classification from the Lie group side, matching ten of the groups named in Theorem \ref{Sato-Tate axioms groups} with these Galois types.
Assume now that $A/k$ is an abelian surface such that $\ST_A^0 = \SU(2)$. In this case, the matrices in $\M_4(\C)$ commuting with $\SU(2)$ are
\[
\left\{ \begin{pmatrix}
a \Id_2& bJ_2 \\
cJ_2 & d \Id_2
\end{pmatrix}: a,b,c,d \in \C\right\}.
\]
The Rosati form is given up to a scalar multiple by
\begin{align*}
\psi \mapsto & \Trace(\Psi S^T \Psi^T S)  \\
&=
\Trace
\left(
\begin{pmatrix}
a \Id_2& bJ_2 \\
cJ_2 & d \Id_2
\end{pmatrix}
\begin{pmatrix}
0 & -\Id_2 \\
\Id_2 & 0
\end{pmatrix}
\begin{pmatrix}
a \Id_2& -cJ_2 \\
-bJ_2 & d \Id_2
\end{pmatrix}
\begin{pmatrix}
0 & \Id_2 \\
-\Id_2 & 0
\end{pmatrix}
\right) \\
&=
\Trace \left( \begin{pmatrix}
b J_2& -a\Id_2 \\
d\Id_2 & -c J_2
\end{pmatrix}
\begin{pmatrix}
c J_2& a\Id_2 \\
-d\Id_2 & -b J_2
\end{pmatrix} \right) \\
&= 2(ad - bc) = \frac{1}{2}((a+d)^2 - (a-d)^2 - (b+c)^2 + (b-c)^2).
\end{align*}
The positive definite subspace is defined by the conditions $a+d,b-c \in \R$ and $a-d,b+c \in i\R$, so
\[
\End(A_K)_{\R} = \left\{ \begin{pmatrix}
(a+bi) \Id_2& (c+di) J_2 \\
(-c+di) J_2 & (a-bi) \Id_2
\end{pmatrix}: a,b,c,d \in \R \right\} \simeq \M_2(\R).
\]
For $n>1$, the action of a generator of the component group of $E_n$ is
\[
\begin{pmatrix}
(a+bi) \Id_2& (c+di) J_2 \\
(-c+di) J_2 & (a-bi) \Id_2
\end{pmatrix} \mapsto
\begin{pmatrix}
(a+bi) \Id_2& e^{2 \pi i/n} (c+di) J_2 \\
e^{-2 \pi i/n} (-c+di) J_2 & (a-bi) \Id_2
\end{pmatrix},
\]
so the fixed ring consists of matrices with $c=d=0$ and is isomorphic to $\C$. The action of $J$ is
\[
\begin{pmatrix}
(a+bi) \Id_2& (c+di) J_2 \\
(-c+di) J_2 & (a-bi) \Id_2
\end{pmatrix} \mapsto
\begin{pmatrix}
(a-bi)\Id_2 & (c-di)J_2\\
(-c-di)J_2 & (a+bi)\Id_2
\end{pmatrix},
\]
so the fixed ring of $J(E_1)$ consists of matrices with $b=d=0$ and is isomorphic to $\R \times \R$.
For $n>1$, the fixed ring of $J(E_n)$ is isomorphic to $\R$ because we must have both $c=d=0$ and $b=d=0$.

It follows from the previous discussion that the correspondence between Galois types in case $\bE$ and Sato-Tate groups with connected component $\SU(2)$ is as indicated in the statement of Theorem~\ref{main-section4}.

\subsubsection{Galois types in case \texorpdfstring{$\bF$}{F}}\label{secF}

In this section we assume that $A/k$ is of type $\bF$, so that\footnote{This was indicated without proof in \S\ref{view},
but it is now a formal consequence of the above: Theorem \ref{mumford-tate genus 2}, Proposition \ref{necessity of Sato-Tate}, and Lemma \ref{genus 2 connected components} imply that $\ST_A^0$ is one of the six connected Lie groups listed in Lemma \ref{genus 2 connected components}.
If $\ST_A^0$ were not conjugate to $\Unitary(1)$, the computations of \S \ref{s1}--\ref{secE} would imply that the type of $A$ is one of $\bA$, $\bB$, $\bC$, $\bD$, or $\bE$, rather than $\bF$.}  $\ST_A^0\simeq \Unitary(1)$.

\begin{proposition}\label{cesc}
The $\Gal(K/k)$-module structure of $\End(A_K)_\R$ is determined by the pair $(\Gal(K/k),\Gal(K/kM))$. More precisely, it is given by the following rule: for $\sigma \in \Gal(K/k)$,
$$
\Trace \rhoA(\sigma)=
\begin{cases}
2(2+\zeta_r+\overline \zeta_r) & \text{if $\sigma \in \Gal(K/k M)$,}\\
0 & \text{otherwise.}
\end{cases}
$$

Here $r$ is the order of $\sigma$ and $\zeta_r$ stands for a primitive $r$th root of unity.

\end{proposition}

\begin{proof}
Suppose first that $\sigma\in\Gal(K/kM)$.
Recall that, except for a set of density zero, a prime $\mathfrak p$ of $k$ is supersingular if and only if $\mathfrak p$ is inert in $k M$. Let~$\mathfrak p$ be a split (i.e., not supersingular) prime in $k M$  of good reduction for~$A$.
Let~$\mathfrak P$ be a prime of $k M$ over $\mathfrak p$.
We first show that if $\Frob_{\mathfrak P}$ lies in the conjugacy class of $\sigma$ in $\Gal(K/kM)$, then the roots of $L_{\mathfrak P}(A_{k M},T)$ are $\alpha$, $\overline\alpha$, $\zeta_r\alpha$, and $\overline\zeta_r\overline\alpha$, for a certain $\alpha \in \C$.
Indeed, let $\alpha$ be a root of $L_{\mathfrak P}(A_{k M},T)$.
Since $\mathfrak P$ is not supersingular, $\alpha/\overline\alpha$ is a root of unity. Observe that the eigenvalues of $\rhoA(\sigma)$ are quotients of roots of $L_{\mathfrak P}(A_{k M},T)$ (see for example \cite{F10}). Suppose that~$\sigma$ is not the trivial element (as otherwise the proposition is trivially true).
Then $\rhoA(\sigma)$ has an eigenvalue $\omega\ne 1$, and the roots of $L_{\mathfrak p}(A/ k M,T)$ are $\alpha$, $\overline\alpha$, $\omega\alpha$, and $\overline\omega\overline\alpha$.
It follows that the set of eigenvalues of $\rhoA(\sigma)$ is $\{1,\,\omega,\,\overline\omega\}$.
Finally, one observes that $\omega$ must be a primitive $r$th root of unity since the order of $\sigma$ is~$r$.

But again the eigenvalues of $\rhoA(\sigma)$ are quotients of the roots $\alpha$, $\overline\alpha$, $\zeta_r\alpha$, $\overline\zeta_r\overline\alpha$ of $L_{\mathfrak P}(A_{k M},T)$ and, since $\mathfrak P$ is not supersingular, $\alpha/\overline\alpha$ is not a root of unity.
Thus, among the $16$ possible quotients between the roots of $L_{\mathfrak P}(A_{k M},T)$, only the following $8$ are roots of unity: $1$, $1$, $1$, $1$, $\zeta_r$, $\zeta_r$, $\overline \zeta_r$, $\overline\zeta_r$.
Since $\rhoA$ has dimension $8$, we have $\Trace\rhoA(\sigma)=2(2+\zeta_r+\overline\zeta_r)$.

Suppose now that $k\not= kM$ and $\sigma\not\in\Gal(K/kM)$.
Let $\chi$ denote the quadratic character of $\Gal(K/k)$ associated to the extension $kM/k$. Let $\mathfrak p$ be any prime of $k$ which does not split completely in $kM$.
Since $\mathfrak p$ is supersingular, we must have $\Trace V_\ell(A)(\Frob_{\mathfrak p})=0$.
It follows that $V_\ell(A)\simeq V_\ell(A)\otimes \chi$, that is, $A\sim_k A \otimes \chi$. Therefore $\End(A_K)_\R=\End(A_K)_\R\otimes \chi$, which implies the claim.
\end{proof}

\begin{remark}\label{rem}
It is not true that the pair $(\Gal(K/k),\Gal(K/kM))$ determines the Galois type of $A$: there exist examples of abelian surfaces $A/k$, $A'/k'$ for which there is an isomorphism of abstract groups
\[
(\Gal(K/k),\Gal(K/kM))\overset{\varphi}{\simeq }(\Gal(K'/k'),\Gal(K'/k'M')),
\]
but for a certain subgroup $H\subseteq \Gal(K/k)$ of order~$2$, the rings $\End(A_K)_\R^H$
and $\End(A'_K)_\R^{\varphi(H)}$ are not isomorphic as $\R$-algebras.
The list of Galois types that are ambiguous in this way can be found in Table \ref{table:matchingU1exc}.
\end{remark}

The computations performed in \S \ref{subsec:unitary1} allow us to determine the possible isomorphism classes for the pair of groups $(\Gal(K/k),\Gal(K/kM))$.
Indeed, observe that $\ST_A$ not only determines $\Gal(K/k)\simeq \ST_A/\ST_A^0$ but also the subgroup $\Gal(K/kM)$: since a prime $\mathfrak p$ of $k$ is supersingular if and only if it does not split completely in $kM$, the group $\Gal(K/kM)$ is isomorphic to the component group of $\ST_A^{ns}$, where $\ST_A^{ns}$ denotes the index 2 subgroup of $\ST_A$ obtained by removing from $\ST_A$ those components for which all elements have the same characteristic polynomial. For each Lie group $G$ with $G^0=\Unitary(1)$ appearing in the list of Theorem \ref{Sato-Tate axioms groups}, the pair $(G/G^0,G^{ns}/G^{ns,0})$ is shown in Table \ref{table:matchingU1}.

\begin{table}
\begin{center}
\small
\setlength{\extrarowheight}{0.5pt}
\caption{Pairs $(\Gal(K/k), \Gal(K/kM))$.}\label{table:matchingU1}
\vspace{6pt}
\begin{tabular}{@{\extracolsep{-4pt}}llllllll}
$\ST_A$&      &$\qquad\quad$& $\ST_A$  &                  &$\qquad\quad$& $\ST_A$   & \\\hline\vspace{-8pt}\\
$C_1$  & $(\cyc 1,\cyc 1)$ && $J(C_1)$ & $(\cyc 2,\cyc 1)$             && $C_{2,1}$ & $(\cyc 2,\cyc 1)$\\
$C_2$  & $(\cyc 2,\cyc 2)$ && $J(C_2)$ & $(\dih 2,\cyc 2)$             && $C_{4,1}$ & $(\cyc 4,\cyc 2)$ \\
$C_3$  & $(\cyc 3,\cyc 3)$ && $J(C_3)$ & $(\cyc 6,\cyc 3)$             && $C_{6,1}$ & $(\cyc 6,\cyc 3)$\\
$C_4$  & $(\cyc 4,\cyc 4)$ && $J(C_4)$ & $(\cyc 4\times\cyc 2,\cyc 4)$ && $D_{2,1}$ & $(\dih 2,\cyc 2)$\\
$C_6$  & $(\cyc 6,\cyc 6)$ && $J(C_6)$ & $(\cyc 6\times\cyc 2,\cyc 6)$ && $D_{4,1}$ & $(\dih 4,\dih 2)$\\
$D_2$  & $(\dih 2,\dih 2)$ && $J(D_2)$ & $(\dih 2\times\cyc 2,\dih 2)$ && $D_{6,1}$ & $(\dih 6,\dih 3)$\\
$D_3$  & $(\dih 3,\dih 3)$ && $J(D_3)$ & $(\dih 6,\dih 3)$             && $D_{3,2}$ & $(\dih 3,\cyc 3)$\\
$D_4$  & $(\dih 4,\dih 4)$ && $J(D_4)$ & $(\dih 4\times\cyc 2,\dih 4)$ && $D_{4,2}$ & $(\dih 4,\cyc 4)$\\
$D_6$  & $(\dih 6,\dih 6)$ && $J(D_6)$ & $(\dih 6\times\cyc 2,\dih 6)$ && $D_{6,2}$ & $(\dih 6,\cyc 6)$\\
$T$    & $(\alt 4,\alt 4)$ && $O_1$    & $(\sym 4,\alt 4)$             && $J(T)$    & $(\alt 4\times\cyc 2,\alt 4)$\\\vspace{2pt}
$O$    & $(\sym 4,\sym 4)$ && $J(O)$   & $(\sym 4\times\cyc 2,\sym 4)$ && &\\\hline
\end{tabular}
\end{center}
\end{table}

By Propositions \ref{Minclos} and  \ref{cesc}, there are eleven Galois types for the case $\bF$ in which $M\subseteq k$, and thus $\Gal(K/k)=\Gal(K/kM)$.
Indeed, note first that for any subgroup $H\subseteq \Gal(K/k M)$, the isomorphism class of the $\R$-algebra $\End(A_K)_\R^H$ depends only on its dimension: by Proposition \ref{Minclos} we know that $\End(A_K)_\Q^H$ is either $M$, a quadratic extension of $M$, or $\M_2(M)$, and so upon tensoring with $\R$ becomes $\C$, $\C\times \C$, or $\M_2(\C)$, respectively.
In addition, Proposition \ref{cesc} shows that the $\Gal(K/kM)$-module structure of $\End(A_K)_\Q$ is uniquely determined by the isomorphism class of the group $\Gal(K/k M)$.
This yields the Galois types $\bF[\cyc n]$ for $n=1$, $2$, $3$, $4$, $6$; $\bF[\dih n]$ for $n=2$, $3$, $4$, $6$; $\bF[\alt 4]$; and $\bF[\sym 4]$.
{}From Table \ref{table:matchingU1}, it follows that these correspond to the Sato-Tate groups $C_n$, $D_n$, $T$, and $O$, respectively.

If $M\not \subseteq k$, then, by Proposition \ref{cesc}, the pair $(\Gal(K/k), \Gal(K/kM))$ still determines the $\Gal(K/k)$-module structure of $\End(A_K)_\R$, but, as we warned in Remark \ref{rem}, more data is needed to determine the Galois type. We now describe these data. A glance at Table \ref{table:matchingU1} shows that each pair $(\Gal(K/k), \Gal(K/kM))$ gives rise to exactly one Galois type, which we denote $\bF[\Gal(K/k),\Gal(K/k M)]$, \emph{except} for the four pairs:
\begin{equation}\label{exc-pairs}
(\cyc 2,\cyc 1), (\dih 2,\cyc 2), (\cyc 6,\cyc 3), (\dih 6, \dih 3).
\end{equation}

In each of these cases we have $\Gal(K/k)\simeq \Gal(K/kM)\times \cyc 2$. Choose such an isomorphism and write $\sigma$ for the nontrivial involution of $\Gal(K/k)$ generating that cyclic subgroup of order $2$. By the Skolem-Noether theorem, $\sigma$  acts on $\End(A_K)_\R \simeq \M_2(\C)$ by
\begin{equation}\label{action-of-s}
x\in \M_2(\C) \mapsto \sigma(x) = \gamma \overline x \gamma^{-1},
\end{equation}
for some $\gamma \in \GL_2(\C)$ satisfying $\gamma \overline{\gamma}\in \C^\times \mathrm{Id}$.
It is obvious from this description that $\End(A_K)_\R^{\langle \sigma \rangle} \cap \C\cdot  \mathrm{Id} = \R\cdot  \mathrm{Id}$.
A further inspection (by writing down the linear equation resulting from \eqref{action-of-s}) shows that $\End(A_K)_\R^{\langle \sigma \rangle}$ is an $\R$-algebra of rank $4$. It cannot be commutative because if it were, it would be a maximal commutative subalgebra of $\M_2(\C)$ and should thus contain the center.
Hence $\End(A_K)_\R^{\langle \sigma \rangle}$ is a quaternion algebra over $\R$, which must be isomorphic to either $\M_2(\R)$ or Hamilton's division quaternion algebra $\mathbb H:= (\frac{-1,-1}{\R})$.

For each pair in \eqref{exc-pairs}, we label these two Galois types as $$\bF[\Gal(K/k),\Gal(K/k M),\M_2(\R)] \mbox{ and } \bF[\Gal(K/k),\Gal(K/k M),\mathbb H].$$
We next apply Proposition \ref{Galois type from ST group} to determine $\End(A)_\R$ for each Sato-Tate group $G$ with $G^0=\Unitary(1)$. This computation shows, in particular, that for the four preceding pairs, the one-to-one correspondence with Sato-Tate groups is as indicated in Table \ref{table:matchingU1exc}.

\begin{table}
\begin{center}
\setlength{\extrarowheight}{0.5pt}
\caption{Galois types for the exceptional pairs.}\label{table:matchingU1exc}
\vspace{6pt}
\begin{tabular}{@{\extracolsep{-3pt}}rrrrr}
ST group & Galois type & & ST group & Galois type  \\\hline\vspace{-8pt}\\
$J(C_1)$ & $\bF[\cyc 2,\cyc 1,\mathbb H]$ && $C_{2,1} $ & $\bF[\cyc 2,\cyc 1,\M_2(\R)]$ \\
$J(C_2) $ & $\bF[\dih 2,\cyc 2,\mathbb H]$ &&  $D_{2,1} $ & $\bF[\dih 2,\cyc 2,\M_2(\R)]$  \\
$J(C_3)$ & $\bF[\cyc 6,\cyc 3,\mathbb H]$ &&$C_{6,1} $ & $\bF[\cyc 6,\cyc 3,\M_2(\R)]$ \\\vspace{2pt}
$J(D_3) $ & $\bF[\dih 6,\dih 3,\mathbb H]$ && $D_{6,1} $ & $\bF[\dih 6,\dih 3,\M_2(\R)]$ \\\hline
\end{tabular}
\end{center}
\end{table}

Any $\psi \in \End(A_K)_{\R}$ acts via a block diagonal matrix
\[
\Psi = \begin{pmatrix} A & 0 \\ 0 & B
\end{pmatrix} \qquad (A,B \in \M_2(\C)).
\]
Since $\psi'$ acts via the matrix $S^{-1} \Psi^T S$, the Rosati form is given up to scalars by
\[
\psi \mapsto \Trace(\Psi S^T \Psi^T S) =
2 \Trace(AB^T).
\]
By positivity of the Rosati form, we must have $B = \overline{A}$, whence
\[
\End(A_K)_{\R} = \left\{ \begin{pmatrix} A & 0 \\ 0 & \overline{A}
\end{pmatrix}: A \in \M_2(\C) \right\} \simeq \M_2(\C).
\]

For $n>1$, the action of a generator of the component group of $C_n$ is
\[
\begin{pmatrix} a & b \\ c & d \end{pmatrix} \mapsto
\begin{pmatrix} a & e^{2\pi i/n} b \\ e^{-2 \pi i/n} c & d \end{pmatrix}.
\]
The fixed ring consists of those matrices with $b=c=0$, and is thus isomorphic to $\C \times \C$.
If we reinterpret $C_2$ as $D_1$, the action of the generator becomes
\[
\begin{pmatrix} a & b \\ c & d \end{pmatrix} \mapsto
\begin{pmatrix} d & -c \\ -b & a \end{pmatrix}.
\]
For $n>1$, the fixed ring under $D_n$ consists of matrices with $a=d$ and $b=c=0$, and is isomorphic to $\C$.
The same is true for $T$ and $O$.

The action of $J$ is
\[
\begin{pmatrix} a & b \\ c & d \end{pmatrix} \mapsto
\begin{pmatrix} \overline{d} & -\overline{c} \\ -\overline{b} & \overline{a} \end{pmatrix},
\]
so the fixed ring under $J(C_1)$ is isomorphic to the Hamilton quaternion ring $\HH$. For $n>1$,
the fixed ring under $J(C_n)$ consists of matrices with $b=c=0$ and $d = \overline{a}$,
and is isomorphic to $\C$. The fixed ring under $J(D_n)$ has the additional condition $d = a$, and is isomorphic to $\R$;
the same is true for $J(T)$ and $J(O)$.

The action of a generator of $C_{2,1}$ is
\[
\begin{pmatrix} a & b \\ c & d \end{pmatrix} \mapsto
\begin{pmatrix} \overline{d} & \overline{c} \\ \overline{b} & \overline{a} \end{pmatrix},
\]
so the fixed ring is $\M_2(\R)$. The fixed ring under $D_{2,1}$ consists of matrices with
$a = d = \overline{d}$ and $b = -c = \overline{c}$, and is isomorphic to $\R \times \R$.

The action of a generator of $C_{4,1}$ is
\[
\begin{pmatrix} a & b \\ c & d \end{pmatrix} \mapsto
\begin{pmatrix} \overline{d} & i \overline{c} \\ -i \overline{b} & \overline{a} \end{pmatrix},
\]
so the fixed ring consists of matrices with $b=c=0$ and $d = \overline{a}$, which is isomorphic to $\C$,
and similarly for $C_{6,1}$. The fixed ring under $D_{4,1}$ adds the condition $d = a$, hence it is isomorphic to $\R$,
and similarly for $D_{6,1}$ and $O_1$ (which contains $D_{4,1}$).

Since $D_{3,2}$ contains $C_3$,
its fixed ring only contains matrices with $b=c=0$.
Since $D_{3,2}$ also contains $J(D_1)$, we must also have $a = \overline{a}$ and $d = \overline{d}$,
so the fixed ring is isomorphic to $\R \times \R$. The same holds for $D_{4,2}$ and $D_{6,2}$.

\subsection{Correspondence with Sato-Tate groups} \label{subsection:correspondence}

Having completed the description of the 52 Galois types,
let us now address the correspondence with Sato-Tate groups included in Theorem \ref{main-section4},
that is, the equivalence between the three sets of data (a), (b), and (c) named in the theorem.
Note that Proposition~\ref{Galois type from ST group} implies that (a) determines (b),
and this has been made explicit in the preceding sections.
Since it is clear that (b) determines (c), it remains only to show that (c) determines (a).

We first note that the six choices for $\ST_A^0$ give rise to six distinct isomorphism classes $\bA$, $\bB$, $\bC$, $\bD$, $\bE$, $\bF$ for the $\R$-algebra
$\End(A_K)_{\R}$, so $\ST_A^0$ and $\End(A_K)_\R$ determine each other. Thus to prove that (c) determines (a),
it is sufficient to distinguish Sato-Tate groups with the same connected part.

To finish the argument, let us inspect Table~\ref{table:STgroups} more closely.
We find that the data given by $(\Gal(K/k), \End(A_K)_{\R}, \End(A_k)_{\R})$ alone is sufficient to distinguish Sato-Tate groups but for three exceptional pairs of groups where an ambiguity arises.
The first ambiguous pair is $J(C_2)$ and $D_2$; these may be distinguished by considering $\End(A_L)_{\R}$, where $L/k$ runs over the three quadratic subextensions of $K/k$.
For $J(C_2)$ one obtains $\HH$, $\HH$, and $\C\times\C$, since the index 2 subgroups of $J(C_2)$ are conjugate to $J(C_1)$, $J(C_1)$ and $C_2$, whereas for $D_2$ one obtains $\C\times \C$ in all cases, since all index 2 subgroups of $D_2$ are conjugate to $C_2$.
The second ambiguous pair is $J(C_3)$ and $C_{6,1}$; these may be distinguished by
passing from the cyclic group $\Gal(K/k)$ of order 6 to its unique subgroup of order 2,
thus reducing to the distinction between $J(C_1)$ and $C_{2,1}$.
The third ambiguous pair is $J(D_3)$ and $D_{6,1}$; these may be distinguished by
passing from the dihedral group $\Gal(K/k)$ of order 12 to its unique cyclic subgroup of order 6,
thus reducing to the distinction between $J(C_3)$ and $C_{6,1}$.

\subsection{Realizability over \texorpdfstring{$\Q$}{Q}}\label{section:realizability over Q}

We now show that certain Galois types cannot occur over $\Q$, or more generally over a field with a real place.

As determined at the end of \S\ref{s3}, the group $F_c$ does not occur as a Sato-Tate group of an abelian surface over any number field.  The remaining 14 Galois types corresponding to Sato-Tate groups ruled out by Remark \ref{real place} over a field with a real place are
\begin{gather*}
\bF[\cyc n]\text{ with $n\in\{1,2,3,4,6\}$, }\bF[\dih n] \text{ with $n\in\{2,3,4,6\}$, } \\
\bF[\alt 4],\, \bF[\sym 4],\,\bD[\cyc 1],\,\bD[\cyc 2,\C],\,\bC[\cyc 1]\,.
\end{gather*}
This list can be recovered immediately from the discussion from \S \ref{s1} to \S \ref{s4}: for all of these Galois types, $k$ must contain either a quadratic imaginary field or a quartic CM field.

The Galois type $\bD[\cyc 2,\R\times\R]$, corresponding to $F_{ab}$, cannot occur over $\Q$ since it corresponds to an abelian surface $A$ over $k$ such that $A\sim_k E_1\times E_2$, where $E_i$ is an elliptic curve over $k$ with CM by a quadratic imaginary field $M_i=\Q(\sqrt{-d_i})$, $i=1,2$, such that $M_1\not\simeq M_2$ and $kM_1=kM_2$. Of course, this last condition does not hold if $k=\Q$, but it can hold over a totally real field (e.g., $k=\Q(\sqrt{d_1d_2})$).

In order to complete the proof of Theorem \ref{main-section4}, we still need to prove that three other Galois types do not occur over $\Q$. In fact, we show that they cannot arise over a field with a real place.

\begin{proposition}\label{proposition:no real place}
The Galois types  $\bF[\cyc 2,\cyc 1,\mathbb H]$, $\bF[\cyc 4,\cyc 2]$, and $\bF[\cyc 6,\cyc 3,\mathbb H]$ cannot occur over a field with a real place.
\end{proposition}

\begin{proof}
Let $A$ be an abelian surface over $k$ with Galois type $\bF[\cyc 2,\cyc 1,\mathbb H]$ (or equivalently, with $\ST_A=J(C_1)$).
Note that $K=kM/k$ is a quadratic extension and write $\Gal(K/k)=\{e,\tau\}$, where $e$ and $\tau$ denote the identity and complex conjugation, respectively.
Extend scalars from $k$ to $\R$. Thus, $K\otimes \R = \C$ and $A_\C$ is isogenous to the square of some CM elliptic curve $E$ over $\C$.
Let $\mathcal O$ be the endomorphism ring of $E$, and take the curve $E'$ corresponding to the lattice $\mathcal O$ in $\C$; this is defined over $\R$ because the lattice is stable under complex conjugation.
Then $E$ and $E'$ are isogenous over an algebraic closure of $\C$, which is again $\C$.

Thus $A_\R$ is a twist of $(E')^2$ by some 1-cocycle $f$. If we normalize $f(e) = 1$, then $f$ sends the complex conjugation $\tau$ to some endomorphism $\alpha$ of $(E')^2$ for which $\alpha \tau(\alpha) = 1$.
If we translate the action of $\tau$ on $\End(A_\C)$ into an action on $\End((E'_\C)^2)_\R \simeq M_2(\C)$, then it is described by conjugating the complex conjugation on $\M_2(\C)$ by~$\alpha$.
By Hilbert-Speiser Theorem 90 (for $H^1(\Gal(\C/\R), \GL_2(\C))$, we can factor~$\alpha$ as $\beta  \tau(\beta)^{-1}$ for some $\beta$ in $\GL_2(\C)$. Using this, we can then calculate that the fixed subring of $\End(A_\C)_\R$ under $\tau$ is isomorphic to $\M_2(\R)$, just as for $(E'_\C)^2$.
However, for $J(C_1)$ this $\R$-algebra must be isomorphic to $\HH$.

Suppose now that the Galois type of $A$ is $\bF[\cyc 4,\cyc 2]$  (equivalently, that $\ST_A=C_{4,1}$). Note that $kM/k$ is the unique quadratic extension of $K/k$. Again extend scalars from $k$ to $\R$, and note that $M\otimes \R = \C$. By the argument of the previous paragraph, $\End(A_\R)_\R$ has rank 4 as an $\R$-algebra. However, in case $C_{4,1}$, no extension $L$ of $k$ contained in $\R$ can have endomorphism algebra of rank greater than 2 (namely, if $L$ doesn't contain $M$, then $\End(A_L)_\R = \End(A_k)_\R = \C)$.

We can reduce the case of an abelian surface $A$ with Galois type $\bF[\cyc 6,\cyc 3,\mathbb H]$ to the case of an abelian surface with Galois type $\bF[\cyc 2,\cyc 1,\mathbb H]$ by considering $A_L$, where $L/k$ is the only cubic subextension of $K/k$ (note that $L$ preserves the property of having a real place).
\end{proof}

\subsection{Examples}\label{section:examples}

{}In \S \ref{s1} through \S \ref{s4} we have shown that there are at most $52$ Galois types that can arise for an abelian surface over a number field.
In \S \ref{section:realizability over Q} we showed that $18$ (resp. 17) of these cannot occur over $\Q$ (resp. over a  field with a real place).
To complete the proof of Theorem \ref{main-section4}, it remains to show that each of the $52$ Galois types is actually realized by an abelian surface $A/k$, and that for the $34$ Galois types admissible over $\Q$, this can be achieved with $k=\Q$.

Here we accomplish this goal by exhibiting explicit examples that are Jacobians of genus 2 curves;
these curves have been chosen so that the Sato-Tate groups can be explained entirely using automorphisms of the curves themselves, without having to study endomorphisms of the Jacobians.
The 52 curves are listed in Table~\ref{table:curves} together with a corresponding Sato-Tate group, which, as shown in \S\ref{subsection:correspondence}, uniquely determines a Galois type.
In this section we prove that each of these curves has the Sato-Tate group listed in Table~\ref{table:curves}, and thus realizes the corresponding Galois type.
We note that the curves corresponding to Galois types admissible over $\Q$ are all defined over $\Q$, and that the curve given for $\bD[\cyc 2,\R\times\R]$ is defined over a totally real field.

We begin with the Galois types in cases $\bE$ and $\bF$.
The first step is to compute the field $K$ for each curve.
Let $\alpha$ and $\gamma$ be automorphisms of a genus $2$ curve~$C$, such that $\alpha$ is a nonhyperelliptic involution, and $\alpha$ and $\gamma$ do not commute.
Let $L/k$ denote the minimal field extension over which $\alpha$ and $\gamma$ are defined.
The quotient $E:=C/\langle\alpha\rangle$ is an elliptic curve over $L$.
Recall that $M$ denotes $\Q$ if $E$ does not have CM and the CM field of $E$ otherwise.
We claim that $K=LM$. Indeed, there exists an elliptic curve $E'$ over $L$ such that $\Jac(C)\sim_{L} E\times E'$.
Since $\Aut(C_{L})$ is nonabelian and injects into $\End(\Jac(C)_{L})_{\Q}$, we must have $\End(\Jac(C)_{L})_{\Q} \simeq \M_2(\End(E_{L})_{\Q})$ and $E\sim_{L} E'$.
It follows that $K=LM$.

In Table \ref{table:automorphisms}, we list $\alpha$, $\gamma$, and $M$ for each of the $42$ curves in cases $\bE$ and $\bF$.
{}From this data one can immediately compute the fields $K$, which are listed in Table~\ref{table:curves}; note that when writing the coefficients of $\alpha$ and $\gamma$ in Table~\ref{table:automorphisms}, there is an implicit reference to the generators of $K$ given in Table~\ref{table:curves}.

Except for $J(E_1)$ and $E_2$, the Sato-Tate groups in case $\bE$ are uniquely determined by $\Gal(K/k)$, and in each such case this implies that the claimed Sato-Tate group is correct.
One finds that the curve $C\colon y^2=x^5 + x^3 + x$ (resp.\ $C: x^6 + x^5 + 3x^4 + 3x^2 - x + 1$) listed for $J(E_1)$ (resp.\ $E_2$) has $\End(\Jac(C)_\Q)_\Q\simeq \Q\times\Q$ (resp.\ $\Q(\sqrt{-2})$), from which it follows that $\End(\Jac(C)_\Q)_\R\simeq \R\times \R$ (resp.\ $\C$), as desired.

Except for the pairs $J(C_1)$ and $C_{2,1}$, $J(C_2)$ and $D_{2,1}$, $J(C_3)$ and $C_{6,1}$, and $J(D_3)$ and $D_{6,1}$, the Sato-Tate groups in case $\bF$ are uniquely determined by $\Gal(K/k)$ and $\Gal(K/kM)$; these Galois groups can be directly computed from the data in Tables~\ref{table:curves} and Table~\ref{table:automorphisms}, and in each case one finds that the claimed Sato-Tate group is correct.
We now address the eight ambiguous cases.

\begin{enumerate}

\item $J(C_1)$: The curve $C\colon y^2=x^5-x$ over $k=\Q(i)$ has an automorphism $\beta(x,y)= (\frac{1}{x},\frac{iy}{x^3})$ with $\beta^2=-1$.
Since we also have $\gamma^2=-1$ and $\gamma\circ\beta=-\beta\circ\gamma$, we conclude that $\End(\Jac(C)_k)_\R\simeq\mathbb H$.

\item $C_{2,1}$: We observe that the curve $C\colon y^2=x^6+1$ over $k=\Q$ has $\alpha^2=\gamma^2=1$ and $\alpha\circ\gamma=-\gamma\circ\alpha$.
It follows that $\End(\Jac(C)_k)_\R\simeq\mathbb M_2(\R)$.

\item $J(C_2)$: The curve $C\colon y^2=x^5-x$ over $k=\Q$ has an automorphism $\beta(x,y)= (-\frac{1}{x},\frac{y}{x^3})$ of order $4$, from which we deduce that $\End(\Jac(C)_k)_\Q\simeq\Q(i)$, and therefore $\End(\Jac(C)_k)_\R\simeq\C$.

\item $D_{2,1}$: Since the nonhyperelliptic involution $\alpha$ of $C\colon y^2=x^5+x$ over $k=\Q$ is defined over $\Q$, we have $\End(\Jac(C)_\Q)_\Q\simeq \Q\times\Q$, and therefore $\End(\Jac(C)_k)_\R\simeq \R\times\R$.

\item $J(C_3)$: It is enough to show that the curve $C\colon y^2=x^6 + 10x^3 - 2$ over $k=\Q(\sqrt{-3})$ has $\End(\Jac(C)_L)_\R\simeq \mathbb H$, where $L=\Q(a,\sqrt{-3})$ is the unique subfield of $K$ with index 2 (here $a=\sqrt[3]{-2}$).
It suffices to find automorphisms $\beta$ and $\delta$ defined over $L$ that satisfy $\beta^2=\delta^2=-1$ and $\beta\circ\delta=-\delta\circ\beta$.
These automorphisms are:

{\small
\begin{gather*}
\beta(x,y)=\left(\frac{(-\sqrt{-3}+1)a^2x+2(\sqrt{-3}+1)a}{4x+(\sqrt{-3}-1)a^2},\frac {-3\cdot 2^5\sqrt{-3}y}{(4x+(\sqrt{-3}-1)a^2)^3}\right)\,, \\
\delta(x,y)=\left(\frac{-a^2x-2a}{2x+a^2},\frac{12\sqrt{-3}y}{(2x+a^2)^3}\right)\,.
\end{gather*}
}

\item $C_{6,1}$: For $C\colon y^2 = x^6 + 6x^5 - 30x^4 + 20x^3 + 15x^2 - 12x + 1$ over $k=\Q$ it suffices to show that $\End(\Jac(C)_{\Q(a)})_\R=\M_2(\R)$, where $\Q(a)$ is the unique subfield of $K$ with index 2.
But this is clear: the nonhyperelliptic involution $\alpha$ is defined over $\Q(a)$ and the noncommuting element $\gamma$ is also defined over $\Q(a)$ (in fact over $\Q$), and thus $\End(\Jac(C)_{\Q(a)})_\Q=\M_2(\Q)$.

\item $J(D_3)$: The same argument used for $J(C_3)$ shows that the Jacobian of the curve $y^2 = x^6 + 10x^3 - 2$ over $\Q$ has Sato-Tate group $J(D_3)$.

\item $D_{6,1}$: For $C\colon y^2 = x^6 + 6x^5 - 30x^4 - 40x^3 + 60x^2 + 24x - 8$ over $k=\Q$, it suffices to show $\End(\Jac(C)_L)_\R=\M_2(\R)$, where $L$ is determined in the following way: it is the field fixed by the unique subgroup of order 2 in $\Gal(K/k)$ contained in the unique cyclic subgroup of order 6 in $\Gal(K/k)$. More explicitly, $L=\Q(a,\sqrt 6)$. It is enough to observe that

{
\footnotesize
\begin{gather*}
\beta(x,y)=\left(\frac{\left(-\frac{1}{2}a\sqrt 6+\frac{1}{2}(a^2+a-2)\right)x+2}{x+\frac{1}{2}a\sqrt6+\frac{1}{2}(-a^2-a+2)},\frac{((9a^2+6a-8)\sqrt 6-9a^2-63a+54)y}{(x+\frac{1}{2}a\sqrt6+\frac{1}{2}(-a^2-a+2))^3}\right), \\
\delta(x,y)=\left(\frac{(-a^2-a+8)x+2}{x+a^2+a-8},\frac{(18a^2+12a-154)\sqrt 6 y}{(x+a^2+a-8)^3}\right)
\end{gather*}
}
do not commute and satisfy $\beta^2=\delta^2=1$.

\end{enumerate}

We now consider the five Sato-Tate groups corresponding to Galois types in case $\bD$, which are addressed in Table~\ref{table:curves} using just two curve equations (over various number fields $k$).
It is enough to observe the following.
First, the Jacobian of the curve $C\colon y^2=x^6+3x^4+x^2-1$ splits over $\Q$ as the product of elliptic curves with CM by $\Q(i)$ and $\Q(\sqrt{-2})$, respectively.
Indeed, using an algorithm of Gaudry and Schost \cite{GS01}, one proves that the the $j$-invariants of the elliptic quotients of this curve are $1728$ and $2000$, which correspond to (nonisogenous) elliptic curves $E_1$ and $E_2$ defined over $\Q$ with CM by $\Q(i)$ and $\Q(\sqrt{-2})$, respectively.
{}From Proposition \ref{M1M2}, it follows that $K=\Q(i,\sqrt 2)$ and that $\Jac(C)$ is isogenous over $\Q$ to $E_1\times E_2$. Second, for $C\colon y^2=x^5+1$, it is well-known that $\End(\Jac(C)_K)_{\Q} \simeq\Q(\zeta_5)$, where $K=\Q(\zeta_5)$.

For the two Sato-Tate groups in case $\bC$, which are both addressed using the same curve equation, the above procedure shows that the Jacobian of the curve $y^2=x^6+3x^4-2$ splits over $\Q$ as the product of an elliptic curve without CM and an elliptic curve with CM by $\Q(i)$ (now the $j$-invariants of the elliptic quotients are $3456$ and $1728$, respectively).

For the two Sato-Tate groups in case $\bB$, it is enough to check that the Jacobian of the curve $C_1\colon y^2 = x^6 + x^2 + 1$ splits over $\Q$
as the product of two nonisogenous curves without CM, while the Jacobian of the curve $C_2\colon y^2 = x^6 + x^5 + x - 1$
splits over $\Q(i)$ as the product of two nonisogenous curves without CM that are Galois conjugates. Indeed, the nonhyperelliptic involution $\alpha(x,y)=(-x,y)$ guarantees that $\Jac(C_1)$ splits over $\Q$ as the product of two elliptic curves $E$ and $E'$ over $\Q$, which do not have CM since their $j$-invariants are $-256/31$ and $6912/31$. Moreover, by Lemma \ref{lemma:nonisofac} below, $E$ and $E'$ are not $\Qbar$-isogenous, because
$$L_5(\Jac(C_1),T)=(1 - T + 5T^2)  (1 + 3T + 5T^2)\,.$$
The curve $C_2$ has a nonhyperelliptic involution $\alpha(x,y)=\left(\frac{-1}{x},\frac{-iy}{x^3}\right)$, which shows that $\Jac(C_2)$ splits over $\Q(i)$ as the product of two elliptic curves $E$ and $E'$ over $\Q(i)$. These two elliptic curves do not have CM, since their $j$-invariants are the roots of the polynomial
$$j^2 - \frac{5328000}{107}j + \frac{9826000000000}{11449}\,.$$
Moreover, they are not $\Qbar$-isogenous, because
$$L_{13}(\Jac(C_2),T)=(1-T+T^2)(1+5T+T^2)\,.$$

\begin{lemma}\label{lemma:nonisofac}
Let $A$ be an abelian surface over $k$ for which there exists a field extension $L/k$ such that $A_L\sim_L E\times E'$, where $E$ and $E'$ are elliptic curves defined over $L$ without complex multiplication.
Suppose there exists a prime $\mathfrak p$ of $k$, of good reduction for $A$ and of residue degree $1$ in $L$, such that $L_{\mathfrak p}(A,T)$ is not of the form $P(T)\cdot P(\pm T)$ for any degree $2$ polynomial $P(T)\in\Q[T]$.
Then $E$ and $E'$ are not $\Qbar$-isogenous.
\end{lemma}

\begin{proof} Suppose that $E$ and $E'$ are $\Qbar$-isogenous. Then there exists a quadratic extension $L'/L$ such that $E'\sim_L E\otimes \chi$, where $\chi$ is a character (either trivial or quadratic) of $\Gal(L'/L)$. Let $\mathfrak P$ be a prime of $L$ lying over $\mathfrak p$. It follows that
$$L_{\mathfrak P}(E',T)=L_{\mathfrak P}(E,\pm T)\,.$$
Since $\mathfrak p$ has residue degree 1 in $L$, we have
$$L_{\mathfrak p}(A,T)=L_{\mathfrak P}(A_L,T)=L_{\mathfrak P}(E,T)\cdot L_{\mathfrak P}(E,\pm T),$$
which proves the claim.
\end{proof}

For case $\bA$, we note that the Galois group of $x^5-x+1$ is the symmetric group $S_5$.  It follows from a theorem of Zarhin \cite{Zar00} that for the curve $C$ defined by $y^2=x^5-x+1$ over $\Q$ we have $\End(\Jac(C)_K)_\Q=\Q$.

\subsubsection{Sato-Tate groups over field extensions}

The curves in Table~\ref{table:curves} were chosen to minimize the degree of their fields of definition; this makes it necessary to use 34 distinct curves (some considered over multiple number fields). However, if one relaxes this restriction
on the field of definition, one can reduce the number of curves by using the fact that
every Sato-Tate group that can arise in genus 2 is conjugate (in $\USp(4)$) to a subgroup of one the groups
\begin{equation*}
J(D_6), J(O), J(E_6), J(E_4), F_{ac}, F_{a,b}, N(\Unitary(1)\times\SU(2)), N(\SU(2)\times\SU(2)), \USp(4).
\end{equation*}
One can thus realize all 52 Sato-Tate groups by considering just the 9 curves of Table~\ref{table:curves} corresponding to these groups over the appropriate number field.
To justify this, we recall that for an abelian surface $A$ defined over $k$, Proposition~\ref{p2.17} asserts that there is a bijection between the conjugacy classes of subgroups of $\Gal(K/k)$ and the conjugacy classes of subgroups of $\ST_A$ containing $\ST_A^0$.
This bijection is given by sending the class of the subgroup $H$ to the conjugacy class of the group $\ST_{A_{L}}$, where $L$ is the fixed field $K^H$.

In Table \ref{table:lattice}, we illustrate how the Jacobian $A=\Jac(C)$ of the curve $C$ of Table~\ref{table:curves} corresponding to $J(O)$ realizes 24 other Sato-Tate groups when considering subextensions of $K/k$. We note that in this example there are 33 conjugacy classes of subgroups of $\Gal(K/k)$, and thus 33 conjugacy classes of subgroups of $J(O)$ containing $\Unitary(1)$.
However, considering conjugation in $\USp(4)$ yields exactly the 25 nonconjugate subgroups listed in Table \ref{table:lattice}.  We leave the corresponding exercise for the other groups listed above to the reader.

\begin{table}
\begin{center}
\setlength{\extrarowheight}{2pt}
\caption{Sato-Tate groups realized by the Jacobian $A$ of the curve defined by $y^2=x^6 - 5x^4 + 10x^3 - 5x^2 + 2x - 1$ over extensions of $\Q$, where $a^3-4a+4=0$, $b^4+22b+22=0$, and $c^2+a+4=0$.}\label{table:lattice}
\vspace{6pt}
\begin{tabular}{llll}
$\ST_{A_L}$ & $L$ & $\ST_{A_L}$ & $L$ \\\hline
$J(O)$    & $\Q$                       & $J(C_3)$  & $\Q(b,\sqrt{-11})$            \\
$O$       & $\Q(\sqrt{-2})$            & $D_{3,2}$  & $\Q(b,\sqrt{22})$             \\
$J(T)$    & $\Q(\sqrt{-11})$           & $C_4$     & $\Q(c,\sqrt{-2})$             \\
$O_1$     & $\Q(\sqrt{22})$            & $J(C_2)$  & $\Q(c,\sqrt{-11})$            \\
$J(D_4)$  & $\Q(a)$                    & $C_{4,1}$ & $\Q(c,\sqrt{22})$             \\
$T$       & $\Q(\sqrt{-2},\sqrt{-11})$ & $D_{2,1}$ & $\Q(c\sqrt{-2},\sqrt{-11})$   \\
$J(D_3)$  & $\Q(b)$                    & $D_2$     & $\Q(c\sqrt{-11},\sqrt{-2})$   \\
$D_4$     & $\Q(a,\sqrt{-2})$          & $C_3$     & $\Q(b,\sqrt{-2},\sqrt{-11})$  \\
$J(D_2)$  & $\Q(a,\sqrt{-11})$         & $C_2$     & $\Q(a,b,\sqrt{-2})$           \\
$D_{4,1}$ & $\Q(a,\sqrt{22})$          & $J(C_1)$  & $\Q(a,b,\sqrt{-11})$          \\
$J(C_4)$  & $\Q(c)$                    & $C_{2,1}$ & $\Q(a,b,\sqrt{22})$           \\
$D_{4,2}$ & $\Q(c\sqrt{-2})$           & $C_1$     & $\Q(a,b,\sqrt{-2},\sqrt{-11})$\\\vspace{2pt}
$D_3$     & $\Q(b,\sqrt{-2})$          & \\\hline
\end{tabular}
\end{center}
\end{table}

\section{Numerical verification}

In this section, we describe some numerical verifications
of the refined Sato-Tate conjecture for abelian surfaces.

\subsection{Densities and moments}\label{subsection:atlas}

Using numerical computations, one can both provisionally identify the Sato-Tate group associated to a particular
abelian surface (which can then be confirmed through analysis of the Galois type), and then test the equidistribution
property predicted by the Sato-Tate conjecture. In order to do this, however, we need a way to numerically compare
the observed distribution of normalized $L$-polynomials to the Sato-Tate prediction.

To facilitate this, we compute the distributions of the first and second coefficients
of the characteristic polynomial of a random conjugacy class in each of the 55 groups named in Theorem \ref{Sato-Tate axioms groups}
(including the three groups excluded by the comparison to Galois types in Section \ref{s2}), under the image of the Haar measure. These distributions can be described in two
equivalent ways, via their \emph{density functions} or their \emph{moment sequences}. By computing these for the
55 groups, we see that the separate distributions of the first and second coefficients are already sufficient to
distinguish the groups.  Thus no joint statistics are needed, but as a matter of interest we give joint density functions for the 6 connected cases.

\begin{remark}
The moment statistics we associate to a closed subgroup $G$ of $\USp(2g)$ are integer symmetric polynomials in the eigenvalues of a matrix chosen uniformly over $G$.
They are thus forced to be integers \cite[Proposition~2]{KS09}.
\end{remark}

Table \ref{table:STgroups} lists the real dimension $d$ and the number of connected components $c=|G/G^0|$ for each group $G$, along with the associated endomorphism algebra $\End(A)_{\R}$.
We also list invariants $z_1=z_{1,0}$ and $z_2=[z_{2,-2},z_{2,-1},z_{2,0},z_{2,1},z_{2,2}]$ defined by
\[
\Pr[a_i=j] = z_{i,j}/c,
\]
where the random variables $a_1$ and $a_2$ denote the linear and quadratic coefficients, respectively, of the characteristic polynomial of a random conjugacy class in~$G$.
Additionally, we give the first three nontrivial moments $\Exp[a_1^2]$, $\Exp[a_1^4]$, $\Exp[a_1^6]$ and $\Exp[a_2], \Exp[a_2^2], \Exp[a_2^3]$ of $a_1$ and $a_2$.
We note that the invariants $d$, $c$, $z_1$, $z_2$, and $\Exp[a_2]$ already suffice to uniquely distinguish each Sato-Tate group in genus 2.

To save space, we use the notations $G_1$, $G_3$, $G_{1,1}$, $G_{1,3}$, and $G_{3,3}$ to identify the connected subgroups $G^0=\Unitary(1)$, $\SU(2)$, $\Unitary(1)\times\Unitary(1)$, $\Unitary(1)\times\SU(2)$, and $\SU(2)\times\SU(2)$ of $\USp(4)$, respectively.

\subsubsection{Computing the distributions of \texorpdfstring{$a_1$}{a1} and \texorpdfstring{$a_2$}{a2}}\label{subsection:moments}
Tables \ref{table:a1moments} and \ref{table:a2moments} give explicit formulas for the moments of $a_1$ and~$a_2$.
Here we describe the derivation of these formulas, as well as the computation of probability density functions for $a_1$ and $a_2$.

For $G=\USp(4)$, the moments of $a_1$ and~$a_2$ may be directly computed using the Weyl integration formula, as in \cite{KS08}.
A bit of calculus shows that the joint density function of $a_1$ and $a_2$ is given by $\frac{1}{4\pi^2}\sqrt{\max\{\rho(a_1,a_2),0\}}$, where
\begin{equation}\label{eq:rho_a1a2}
\rho(a_1, a_2) = (a_1^2 - 4a_2 + 8)(a_2 - 2a_1 + 2)(a_2 + 2a_1 + 2).
\end{equation}
The support of the joint density function is the region\footnote{Serre points out that the parabolic arc bounding the top of this
region corresponds precisely to the Hodge circle, and that the factorization of $\rho$ is a special case of a general formula
giving the product of the differentials of fundamental characters \cite[Lemma~8.2]{St}.} where $\rho$ is nonnegative:
\[
S=\{(a_1, a_2) \in \R^2: a_2 \geq 2a_1 - 2, a_2 \geq -2a_1 - 2, a_2 \leq \frac{1}{4} a_1^2 + 2\}.
\]
One recovers the density function for $a_1$ by integrating with respect to $a_2$, and vice versa; the results can be expressed
in terms of complete elliptic integrals, because $\rho(a_1, a_2)$ is a polynomial of degree $4$ in $a_1$ and of degree $3$ in $a_2$.
A plot of the joint density function for $\USp(4)$ can be found in Figure~\ref{figure:USp4joint}.
A similar analysis can be applied to the groups $\Unitary(1)\times\Unitary(1)$, $\Unitary\times\SU(2)$, and $\SU(2)\times\SU(2)$, using products of the appropriate measures; the results are tabulated in Table~\ref{table:joint}.
In each of these cases the support of the joint density function is the two-dimensional region $S$.

\begin{table}
\begin{center}
\small
\setlength{\extrarowheight}{4pt}
\caption{Joint density functions $c\sqrt{\max\{\rho(a_1,a_2),0\}}$ for the Sato-Tate groups $G_{1,1}=\Unitary(1)\times\Unitary(1)$, $G_{1,3}=\Unitary(1)\times\SU(2)$, $G_{3,3}=\SU(2)\times\SU(2)$, and $\USp(4)$.}\label{table:joint}
\vspace{6pt}
\begin{tabular}{lcc}
$G$&$c$&$\rho(a_1,a_2)$\\\hline
$G_{1,1}$ & $\frac{2}{\pi^2}$ & $1\ /\ \bigl((a_1^2-4a_2+8)(a_2-2a_1+2)(a_2+2a_1+2)\bigr)$\\
$G_{1,3}$ & $\frac{1}{2\pi^2}$ & $(4+2a_2-a_1^2)^2\ /\  \bigl((a_1^2-4a_2+8)(a_2-2a_1+2)(a_2+2a_1+2)\bigr)$\\
$G_{3,3}$ & $\frac{1}{2\pi^2}$ & $(a_2-2a_1+2)(a_2+2a_1+2)\ /\ (a_1^2-4a_2+8)$\\\vspace{2pt}
$\USp(4)$ & $\frac{1}{4\pi^2}$ & $(a_2-2a_1+2)(a_2+2a_1+2)(a_1^2-4a_2+8)$\\\hline
\end{tabular}
\end{center}
\end{table}

In all other cases the support is one-dimensional, as may be seen in Table~\ref{table:components}, which lists all the component density functions for $a_1$ and $a_2$ that can arise in genus 2.  With the exception of $\USp(4)$, these distributions are all derived from the distributions $a_{1,\Unitary(1)}$ and $a_{1,\SU(2)}$ that arise for the two connected Sato-Tate groups $\Unitary(1)$ and $\SU(2)$ in genus 1 (where $\Unitary(1)$ is embedded in $\SU(2)=\USp(2)$).  We recall that the even moments are given by
\begin{align}
\Exp[a_{1,\Unitary(1)}^{2n}] &= \binom{2n}{n},\\
\Exp[a_{1,\SU(2)}^{2n}] &=\frac{1}{n+1}\binom{2n}{n},
\end{align}
while the odd moments are zero, and we have the density functions
\begin{align}
\dens(a_{1,\Unitary(1)}=t) &= \frac{1}{\pi\sqrt{4-t^2}}\qquad |t| < 2,\\
\dens(a_{1,\SU(2)}=t) &= \frac{\sqrt{4-t^2}}{2\pi}\qquad\medspace\medspace\medspace |t| <2.
\end{align}
For convenience, we define the sequences
\[
b_n = [X^n](X^2+1)^n,\qquad\text{and}\qquad c_n = b_n/(\nicefrac{n}{2}+1),
\]
where $[X^n](X^2+1)^n$ denotes the coefficient of $X^n$ in the expansion of $(X^2+1)^n$, so that $\Exp[a_{1,\Unitary(1)}^n]=b_n$ and $\Exp[a_{1,\SU(2)}^n]=c_n$.  These are sequences \Aseq{126869} and \Aseq{126120}, respectively, in the On-line Encyclopedia of Integer Sequences \cite{OEIS}.

Each group $G\subsetneq\USp(4)$ appearing in Theorem \ref{Sato-Tate axioms groups} may be expressed in the form $\langle G^0, H \rangle$, where $H$ is a finite subgroup of $\USp(2g)$ whose intersection with $G^0$ is $\{\pm 1\}$, so that $G/G^0\simeq H/\{\pm 1\}$ (in most cases $H$ may be constructed by simply omitting $G^0$ from the list of generators given for $G$ in \S \ref{section:STgroups}).
After picking a representative $h\in H$ for each coset of $H/\{\pm 1\}$, the distributions of the coefficients of the characteristic polynomial $\sum a_iT^i$ of a random matrix $gh$ may then be computed in terms of $a_{1,\Unitary(1)}$ and $a_{1,\SU(2)}$, where $g\in G^0$ is distributed according to the Haar measure.
This allows the moments and density functions of $a_1$ and $a_2$ to be computed for the component $hG^0$; averaging over the components yields results for $G$.

The derivation of the component distributions in the split cases, where $G^0$ is $\Unitary(1)\times\Unitary(1)$, $\Unitary(1)\times\SU(2)$, or $\SU(2)\times\SU(2)$,
is straightforward.  The results are tabulated in Table~\ref{table:components} (see Table~\ref{table:joint} for the joint distribution on the identity components).
We now focus on the cases $G^0=\Unitary(1),\SU(2)$, which account for 42 of the 55 Sato-Tate groups in genus 2, including the most complicated cases.
There are 76 distinct components contained in these groups, but only 20 different component distributions that arise, each of which is determined by a triple $(G^0,s,r)$ that we now define.
Let $\varphi:G/G^0\to\{\pm 1\}$ be the (possibly trivial) homomorphism with kernel $G\cap Z$.  For each component $hG^0$ of $G$, let $k=k(h)\in \{1,2,3,4,6\}$ be the order of $hG^0$ in $G/G^0$ when $s=s(h)=\varphi(hG^0)=1$, and let $k$ be the order of $JhG^0$ in $\langle G,J \rangle/G^0$ when $s=-1$ (see \S \ref{subsec:unitary1} for the definitions of~$Z$ and~$J$).
Finally, we define $r=r(h)=\zeta_{2k}+\zeta_{2k}^{-1}\in\{-2,0,1,\sqrt{2},\sqrt{3}\}$.
Table \ref{table:components} gives the component distributions of $a_1$ and $a_2$ for each triple $(G^0,s,r)$, as well as the 10 component distributions that arise among the 13 remaining groups (30~component distributions in total).
The notation $a_{1,\Unitary(1)}'$ denotes an independent random variable with the same distribution as $a_{1,\Unitary(1)}$, and similarly for $a_{1,\SU(2)}'$.

\begin{table}
\begin{center}
\setlength{\extrarowheight}{2pt}
\caption{Component distributions of $a_1$ and $a_2$.  Here $G_1=\Unitary(1)$, $G_3=\SU(2)$,  $G_{1,1}=\Unitary(1)\times\Unitary(1)$, $G_{1,3}=\Unitary(1)\times\SU(2)$, and $G_{3,3}=\SU(2)\times\SU(2)$.  See \S \ref{subsec:group odd cases} for defintions of the matrices $a$, $b$, $c$, and $J$.}\label{table:components}
\vspace{6pt}
\begin{tabular}{lcccc}
component &$\quad$& $a_1$ &$\quad$& $a_2$\\\hline
$(G_1,1,r)$ && $ra_{1,\Unitary(1)}$ && $a_{1,\Unitary(1)}^2+r^2-2$\\
$(G_1,-1,r)$ && $0$ && $2-r^2$\\
$(G_3,1,r)$ && $ra_{1,\SU(2)}$ && $a_{1,\SU(2)}+r^2-2$\\
$(G_3,-1,r)$ && $0$ && $2-a_{1,\SU(2)}^2$\\
$G_{1,1}$ && $a_{1,\Unitary(1)}+a_{1,\Unitary(1)}'$ && $a_{1,\Unitary(1)}a_{1,\Unitary(1)}'+2$\\
$aG_{1,1},bG_{1,1},(ac)^2G_{1,1}$ && $a_{1,\Unitary(1)}$ && 2\\
$abG_{1,1}$ && $0$ && $2$\\
$cG_{1,1}$ && $0$ && $a_{1,\Unitary(1)}$\\
$acG_{1,1},(ac)^3G_{1,1}$ && $0$ && $0$\\
$G_{1,3}$  && $a_{1,\Unitary(1)}+a_{1,\SU(2)}$ && $a_{1,\Unitary(1)}a_{1,\SU(2)}+2$\\
$aG_{1,3}$ && $a_{1,\Unitary(1)}$ && $2$\\
$G_{3,3}$ && $a_{1,\SU(2)}+a_{1,\SU(2)}'$ && $a_{1,\SU(2)}a_{1,\SU(2)}'+2$\\
$JG_{3,3}$ && $0$ && $a_{1,\SU(2)}$\\\vspace{2pt}
$\USp(4)$ && $a_{1,\USp(4)}$ && $a_{2,\USp(4)}$\\\hline
\end{tabular}
\end{center}
\end{table}

Using Table \ref{table:components}, one can determine the moments and probability density functions of $a_1$ and $a_2$ on each component,
and then compute moments and density functions for any particular group $G$.  The invariants $z_1$ and $z_2$, which simply count
components on which $a_1$ or $a_2$ has a particular fixed value, are also easily derived from Table \ref{table:components}.

To simplify the moment formulas, for $i=0,1,2,3,4$ we define the sequences
\[
b_{i,n} = [X^n](X^2+iX+1)^n,
\]
which for $i=0,1,2,3,4$ correspond to sequences \Aseq{126869}, \Aseq{0002426}, \Aseq{000984}, \Aseq{026375}, \Aseq{081671} in the OEIS (respectively),
and we note $b_n=b_{0,n}$.
We also define the sequences
\[
d_{i,n} = [X^n](X^2+iX+1)^n - [X^{n+1}](X^2+iX+1)^n,
\]
which for $i=0,1,2,3,4$ correspond to sequences \Aseq{126930}, \Aseq{005043}, \Aseq{000108}, \Aseq{007317}, \Aseq{064613}, and we may write $d_n$ for $d_{0,n}$.
Additionally, we let
\[
\bbn = \sum_k\binom{n}{k}2^{n-k}b_k^2 \qquad\text{and}\qquad \ccn = \sum_k\binom{n}{k}2^{n-k}c_k^2.
\]

\subsubsection{Component distributions in the case \texorpdfstring{$G^0=\Unitary(1)$}{G0=U1}}

We now consider the component distributions of $a_1$ and $a_2$ when $G^0=\Unitary(1)$.
{}From Table \ref{table:components} we see that $a_1=ra_{1,\Unitary(1)}$ when $s=1$, and $a_1=0$ otherwise.
On a component $hG^0$ with $s=1$, the moments of $a_1$ are given by $\Exp_h[a_1^n]=r^nb_n$, where the subscript $h$ identifies the component and thus determines $r$ and $s$.  The density function for $a_1$ on $hG^0$ is
\begin{equation}
\dens_h(a_1=t) =
\begin{cases}
\left(\pi\sqrt{4r^2-t^2}\right)^{-1} & \quad\text{if } t^2<4r^2 \text{ and } s=1,\\
\delta_0 & \quad\text{if } r = 0 \text{ or } s=-1,\\
0 & \quad\text{otherwise.}
\end{cases}
\end{equation}
where $\delta_k$ denotes the Dirac delta function centered at $k$.

For $a_2$ we have $a_2=a_{1,\Unitary(1)}^2+r^2-2$ when $s=1$ and $a_2=2-r^2$ otherwise.
The moments of $a_2$ on $hG^0$ are given by
\begin{equation}
\Exp_h[a_2^n]=
\begin{cases}
\sum_{k=0}^n\binom{n}{k}b_{2k}(r^2-2)^{n-k} = b_{r^2,n} &\qquad\text{if } s=1,\\
(2-r^2)^n & \qquad\text{if } s=-1,
\end{cases}
\end{equation}
and its density function is
\begin{equation}
\dens_h(a_2=t) =
\begin{cases}
\left(\pi\sqrt{4-(r^2-t)^2}\right)^{-1} & \quad\text{if } (r^2-t)^2 < 4 \text{ and } s = 1,\\
\delta_{2-r^2} & \quad\text{if } s = -1,\\
0 & \quad\text{otherwise.}
\end{cases}
\end{equation}

For any particular Sato-Tate group $G$ the moment sequences and density functions of $a_1$ and $a_2$ are then computed by averaging over the components.
Taking $G=T$ as an example, averaging over the 12 components of $G$ yields
\[
\Exp[a_2^n] = \frac{1}{12}\left(b_{4,n}+3b_n+8b_{1,n}\right),
\]
as listed in table \ref{table:a2moments}.  A plot of the $a_2$ density function appears in Figure \ref{figure:Ta2hist}.

\subsubsection{Component distributions in the case \texorpdfstring{$G^0=\SU(2)$}{G0=SU2}}

In this case we have $a_1 = ra_{1,\SU(2)}$ when $s=1$, and $a_1=0$ otherwise.
The $n$th moment of $a_1$ on $hG^0$ is $\Exp_h[a_1^n]=r^nc_n$, and its density function is
\begin{equation}
\dens_h(a_1=t) =
\begin{cases}
\frac{1}{2\pi r^2}\sqrt{4r^2-t^2} & \quad\text{if } t^2 < 4r^2\text{ and } s = 1,\\
\delta_0 & \quad\text{if } r = 0 \text{ or } s = -1,\\
0 & \quad\text{otherwise.}
\end{cases}
\end{equation}

For $a_2$ we have $a_2=a_{1,\SU(2)}^2+r^2-2$ when $s=1$ and $a_2=2-a_{1,\SU(2)}^2$ otherwise.
The moments of $a_2$ on $hG^0$ are given by
\begin{equation}
\Exp_h[a_2^n]=
\begin{cases}
\sum_{k=0}^n\binom{n}{k}c_{2k}(r^2-2)^{n-k} = d_{r^2,n} &\qquad\text{if } s=1,\\
\sum_{k=0}^n\binom{n}{k}(-1)^kc_{2k}2^{n-k} = (-1)^nd_n &\qquad\text{if } s=-1,\\
\end{cases}
\end{equation}
and its density function is
\begin{equation}
\dens_h(a_2=t) =
\begin{cases}
\frac{1}{2\pi}\sqrt{4/(t-r^2+2) - 1} & \quad\text{if }|t-r^2| < 2\text{ and } s =1,\\
\frac{1}{2\pi}\sqrt{4/(2-t) - 1} & \quad\text{if }|t| < 2\text{ and } s =-1,\\
0 & \quad\text{otherwise.}
\end{cases}
\end{equation}

\subsection{Testing the refined conjecture for genus 2 curves}\label{subsection:computations}

Let us now see how the densities and moments computed in \S\ref{subsection:atlas} can be used to numerically test the refined Sato-Tate conjecture; this will also provide some indication of how we assembled the numerical
evidence needed to formulate the conjecture in the first place.
We limit ourselves to testing the equidistribution of normalized
$L$-polynomials (rather than the equidistribution of classes in $\Conj(\ST_A)$ described
in Conjecture~\ref{Refined Sato-Tate}), and we consider only Jacobians of curves of genus 2.

We first discuss how to test the conjecture for a single genus 2 curve $C/k$.
For a given bound $N$, we consider the primes $\mathfrak{p}$ of norm $q=\norm{\mathfrak{p}}\le N$ at which $C$ has good reduction.
For each $\mathfrak{p}$ we compute normalized $L$-polynomial coefficients
\begin{align}\label{eq:a1a2curvecount}
a_1(\mathfrak{p}) &= q^{-1/2}\left(\#C(\F_q) - q - 1\right),\\
a_2(\mathfrak{p}) &= q^{-1}\left(\#C(\F_{q^2}) + (\#C(\F_q)-q-1)^2 - q^2 - 1\right)/2,\notag
\end{align}
and then calculate \emph{moment statistics} $M_n(a_i)$ as the mean value of $a_i(\mathfrak{p})^n$ over
$\norm{\mathfrak{p}}\le N$.  The refined Sato-Tate conjecture implies that, as $N\to\infty$, each moment statistic $M_n(a_i)$
must converge to the moment $\Exp[a_i^n]$ for the Sato-Tate group $\ST_{\Jac(C)}\subseteq \USp(4)$.
We can test this numerically for the first several values of $n$ by comparing $M_n(a_i)$ to $\Exp[a_i^n]$ as we increase the bound $N$.
In addition, we may plot histograms of the $a_i(\mathfrak{p})$ and compare the results to the density function for $a_i$ in $\ST_{\Jac(C)}$.

For each of the example curves listed in Table \ref{table:curves}, we have prepared animated histograms of
the $a_1$ and $a_2$ distributions demonstrating the convergence to the conjectured Sato-Tate
distribution, at the level of both densities and moments.  These can be found online at
\begin{center}
\url{http://math.mit.edu/~drew}.
\end{center}
We give one representative example here, using the curve $C/\Q(\sqrt{-2})$ with $\ST_{\Jac(C)}=T$ listed in Table~\ref{table:curves}.
We computed moment statistics $M_n(a_i)$ for this curve using $N=2^k$, with $k$ ranging from 10 to 30, the first several of which are listed in Table~\ref{table:moment statistics}.
For comparison, the last line of the table lists the corresponding moments for the Sato-Tate group $T$.

\begin{table}
\begin{center}
\caption{Moment statistics for the curve $y^2 = x^6 + 6x^5 - 20x^4 + 20x^3 - 20x^2 - 8x + 8$ over the field $\Q(\sqrt{-2})$.}\label{table:moment statistics}
\begin{tabular}{@{}llllllllllll@{}}
&&\multicolumn{4}{c}{$a_1$}&&\multicolumn{5}{c}{$a_2$}\\
\cmidrule(r){3-6}\cmidrule(r){8-12}
$N$ && $M_2$ & $M_4$ & $M_6$ & $M_8$ && $M_1$ & $M_2$ & $M_3$ & $M_4$ & $M_5$\\
\midrule
10 && 1.79 & 11.74 & 134.59 & 1894.71 && 0.81 & 3.55 & 10.82 & 51.55 & 256.29\\
12 && 1.90 & 11.87 & 127.24 & 1714.90 && 0.91 & 3.81 & 11.46 & 52.13 & 246.15\\
14 && 2.01 & 12.54 & 130.43 & 1701.09 && 0.97 & 4.02 & 12.23 & 54.66 & 253.10\\
16 && 1.94 & 10.93 & 102.75 & 1266.03 && 0.97 & 3.86 & 11.05 & 46.56 & 203.16\\
18 && 1.96 & 11.45 & 111.89 & 1417.51 && 0.99 & 3.92 & 11.54 & 49.33 & 220.74\\
20 && 1.99 & 11.87 & 118.12 & 1513.12 && 1.00 & 3.98 & 11.88 & 51.37 & 232.49\\
22 && 1.99 & 11.89 & 118.64 & 1522.75 && 1.00 & 3.98 & 11.90 & 51.51 & 233.49\\
24 && 2.00 & 11.95 & 119.20 & 1528.43 && 1.00 & 3.99 & 11.96 & 51.74 & 234.53\\
26 && 2.00 & 11.99 & 119.80 & 1537.06 && 1.00 & 4.00 & 11.99 & 51.93 & 235.62\\
28 && 2.00 & 12.00 & 119.91 & 1538.51 && 1.00 & 4.00 & 12.00 & 51.97 & 235.83\\
30 && 2.00 & 12.00 & 120.02 & 1540.29 && 1.00 & 4.00 & 12.00 & 52.00 & 236.03\\
\midrule
&& 2 & 12 & 120 & 1540 && 1 & 4 & 12 & 52 & 236\\
\end{tabular}
\end{center}
\end{table}

The evident convergence of the moment statistics of $C$ to the moments of $\ST_{\Jac(C)}$ extends to moments well beyond the range of the table.
With $N=2^{30}$ the first 20 moment statistics of $a_1$ and $a_2$ for $C$ agree with the corresponding moments of $\ST_{\Jac(C)}$ with a relative error of approximately $0.1\%$ (ignoring the even moments $\Exp[a_1^{2k}]=0$).
By contrast, the best agreement one finds using the moments of any of the other genus 2 Sato-Tate groups is worse than $40\%$.
In particular, the corresponding moments of $\USp(4)$ are dramatically different: $\Exp[a_1^8]$ is only 84, rather than 1540, for example.

Histograms of $a_2(\mathfrak{p})$ values for $\norm{\mathfrak{p}}\le N=2^k$, with $k=12, 14, 16, \ldots, 30$ are shown in Figure \ref{figure:Ta2hist}, together with the $a_2$ density function for $\ST_{\Jac(C)}=T$.
One can see the histogram data steadily converging toward the density function.
Indeed, when $N=2^{30}$ the histogram data matches the density function so closely that it is difficult to distinguish the two.

\begin{center}
\begin{figure}
\includegraphics[scale=0.48]{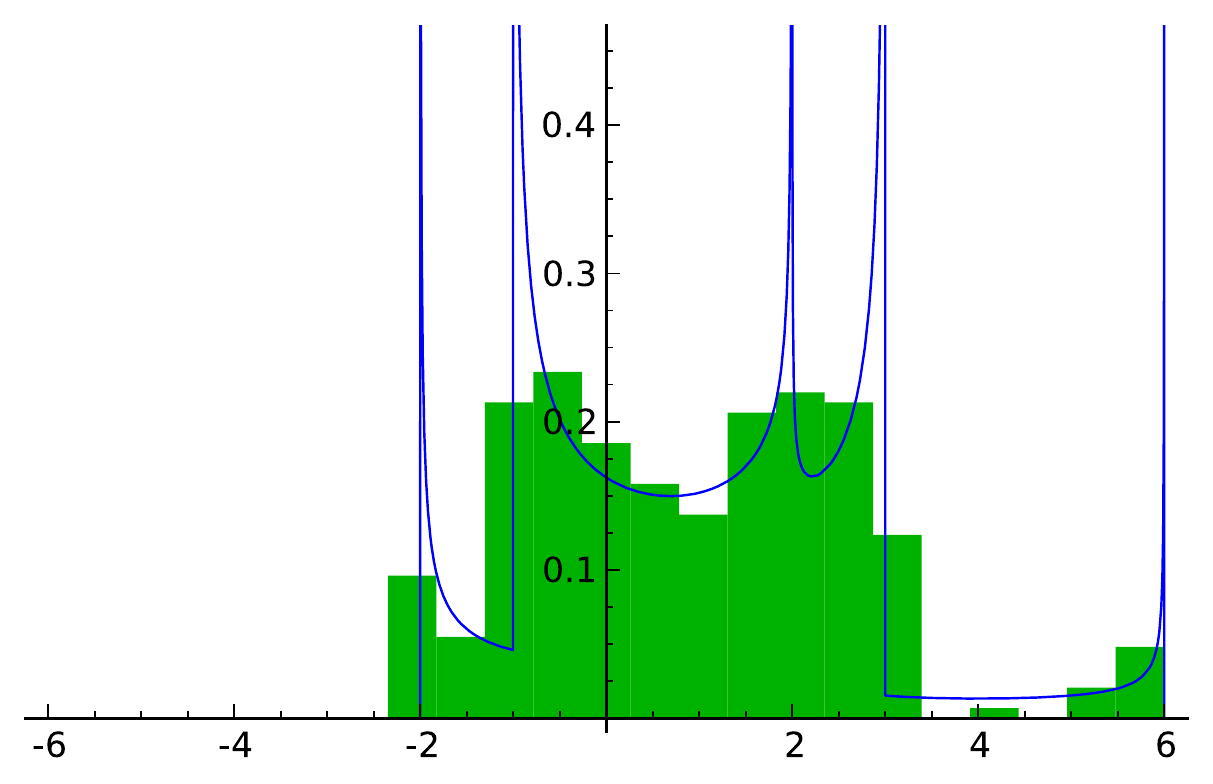}
\quad
\includegraphics[scale=0.48]{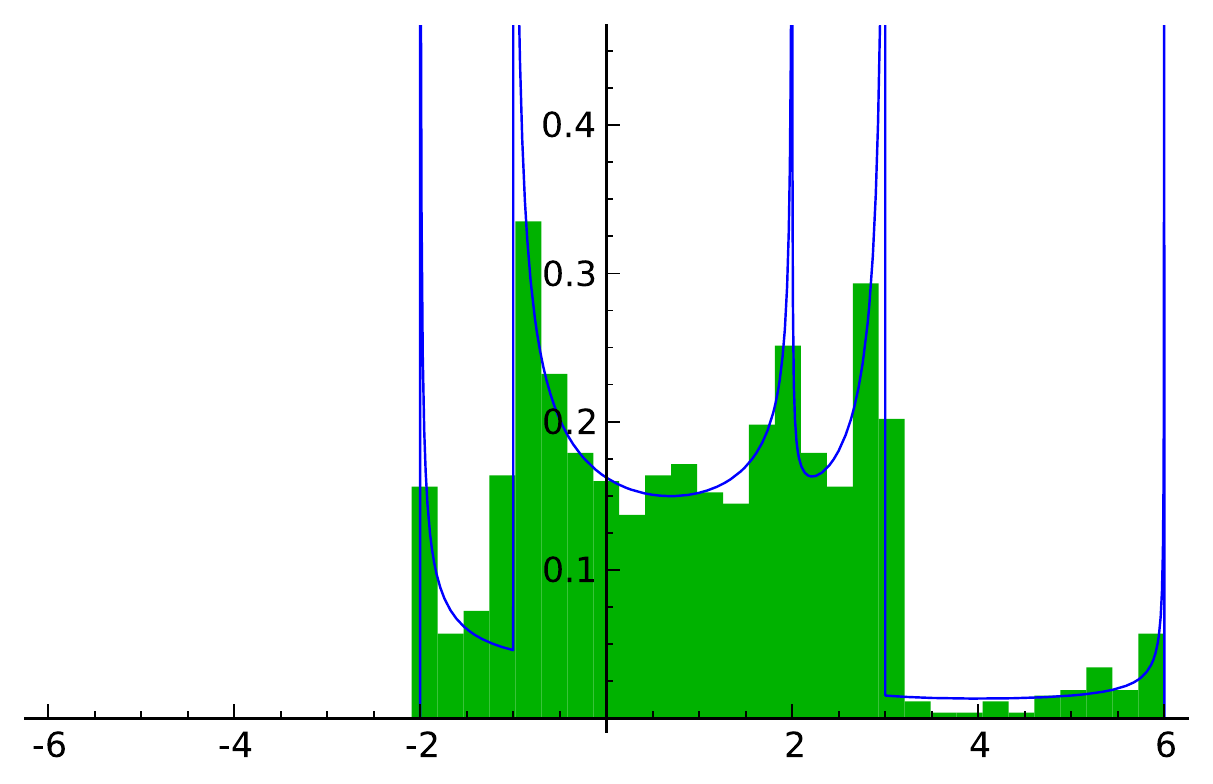}
\quad
\includegraphics[scale=0.48]{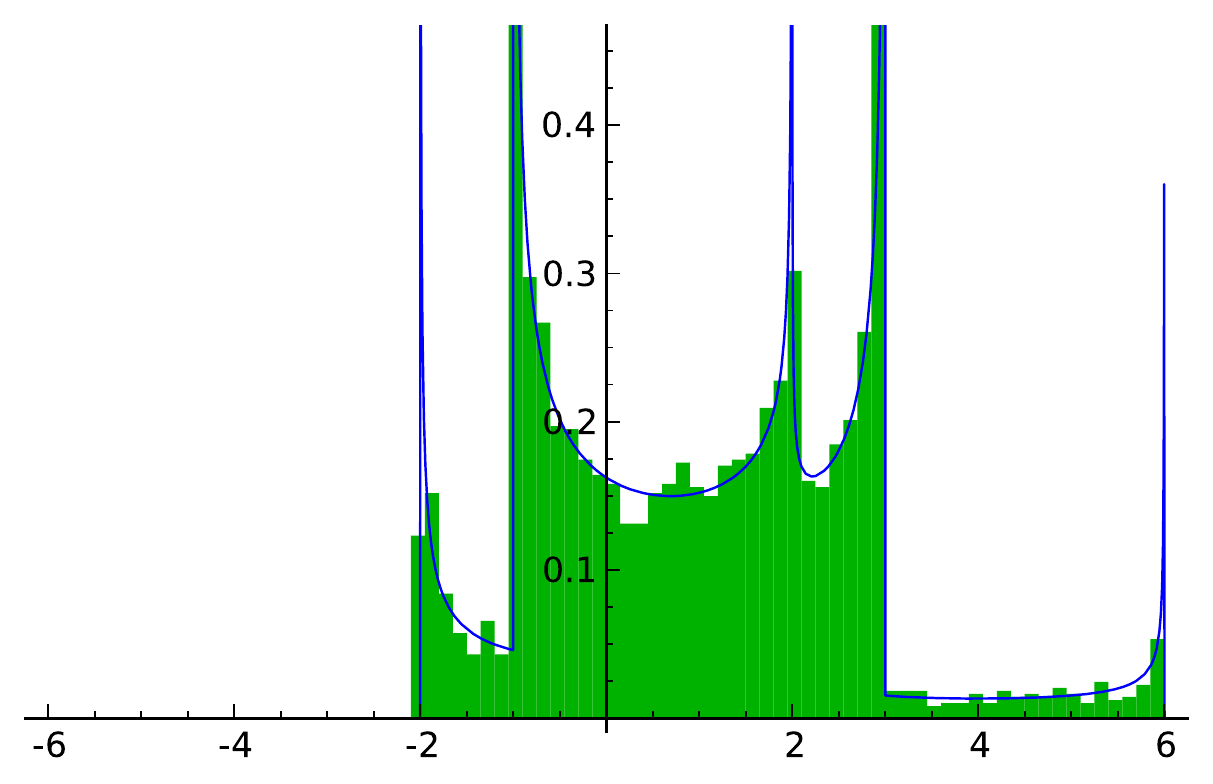}
\quad
\includegraphics[scale=0.48]{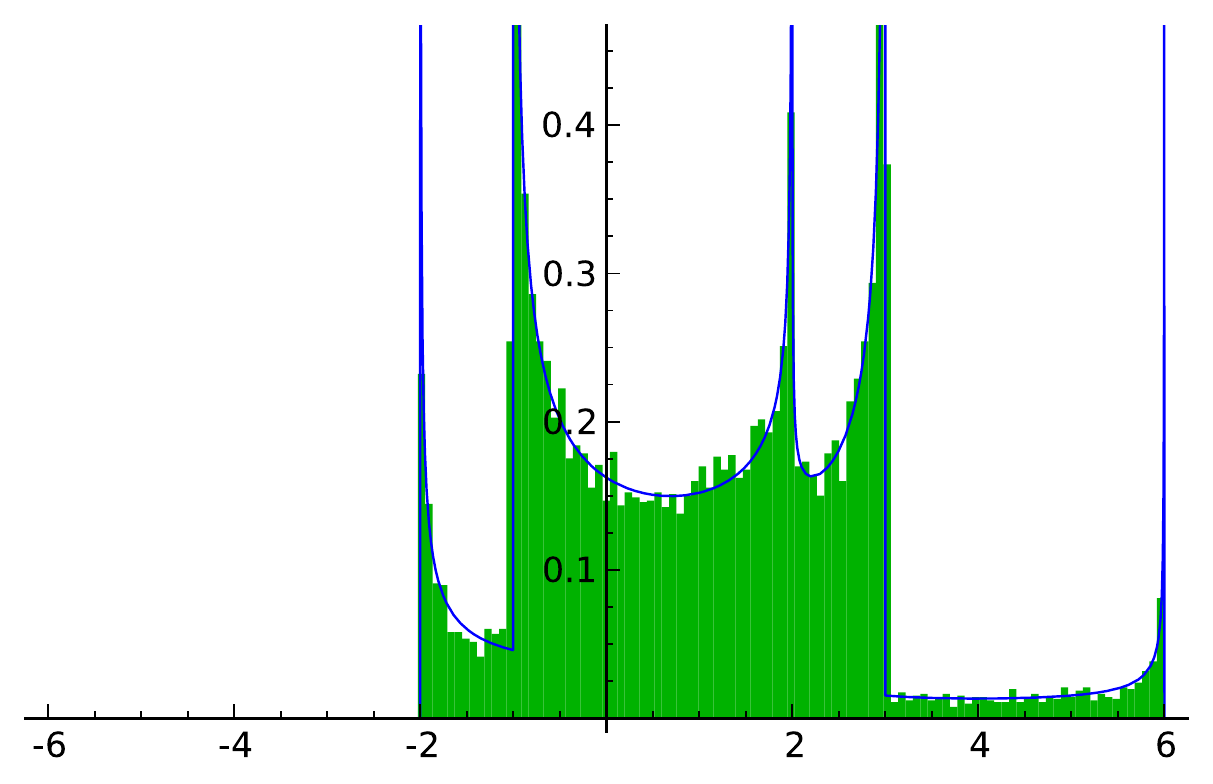}
\quad
\includegraphics[scale=0.48]{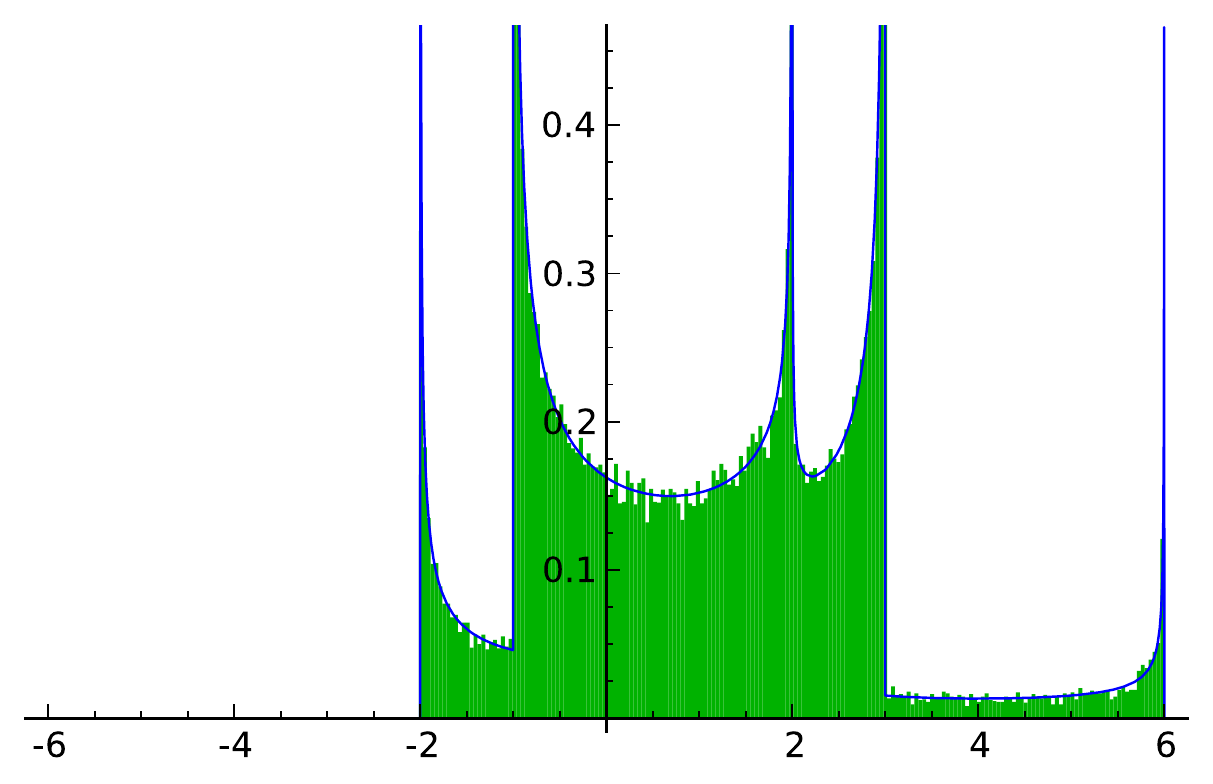}
\quad
\includegraphics[scale=0.48]{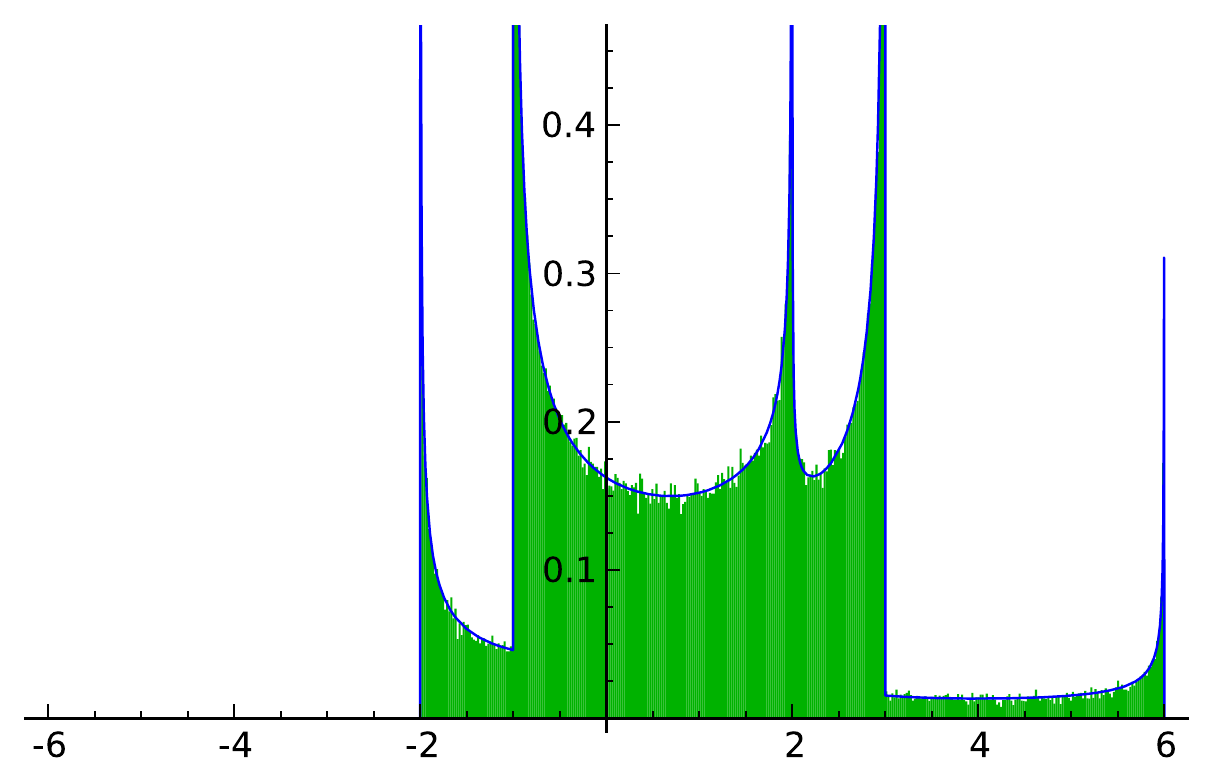}
\quad
\includegraphics[scale=0.48]{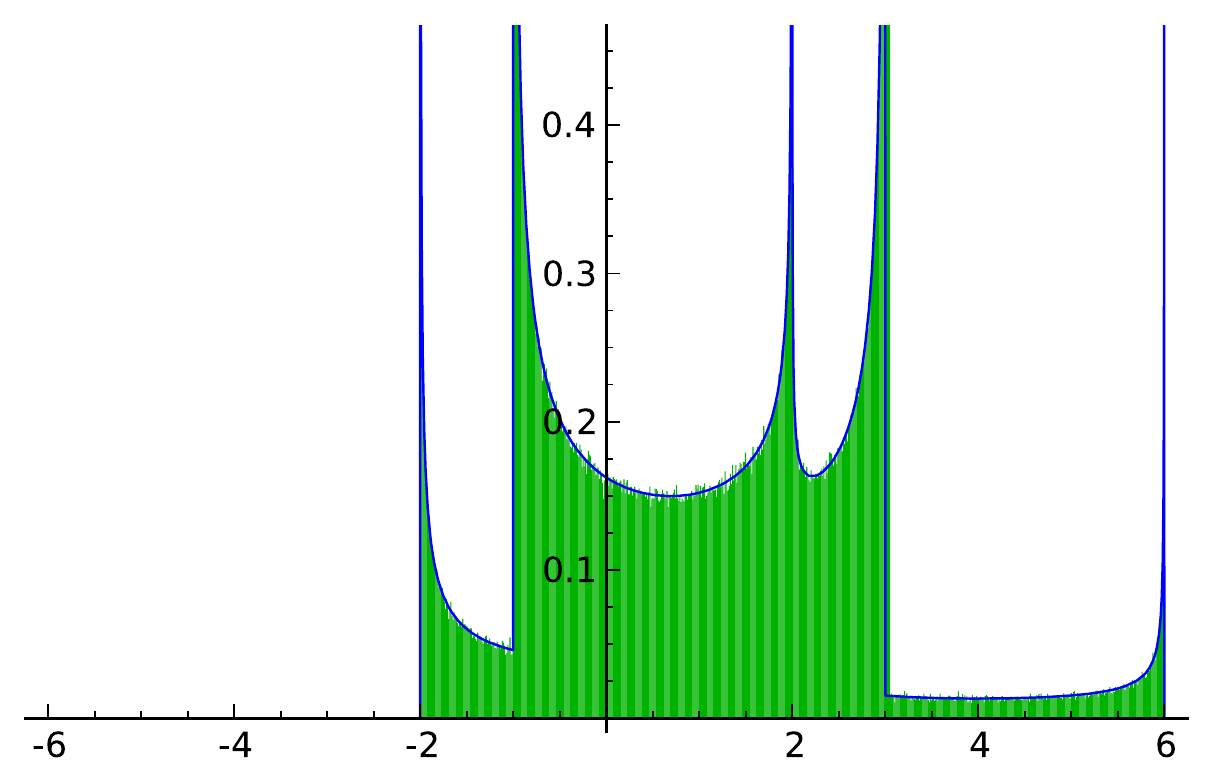}
\quad
\includegraphics[scale=0.48]{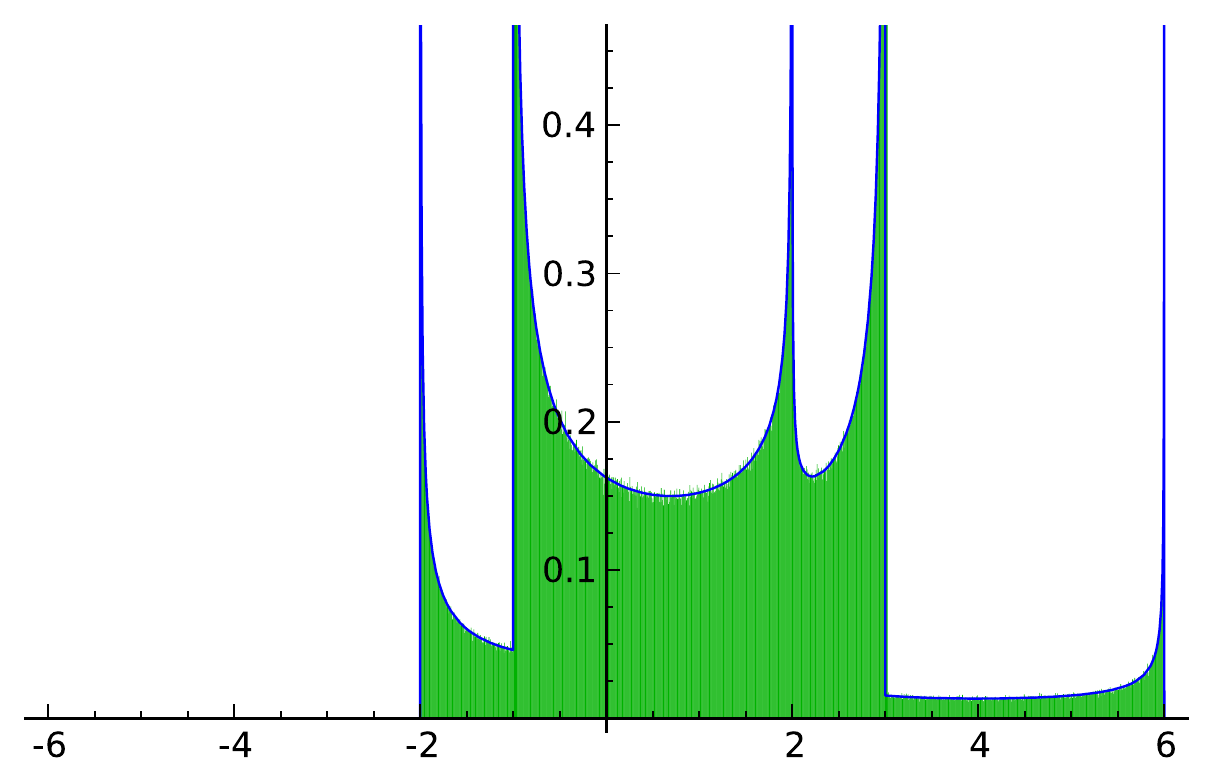}
\quad
\includegraphics[scale=0.48]{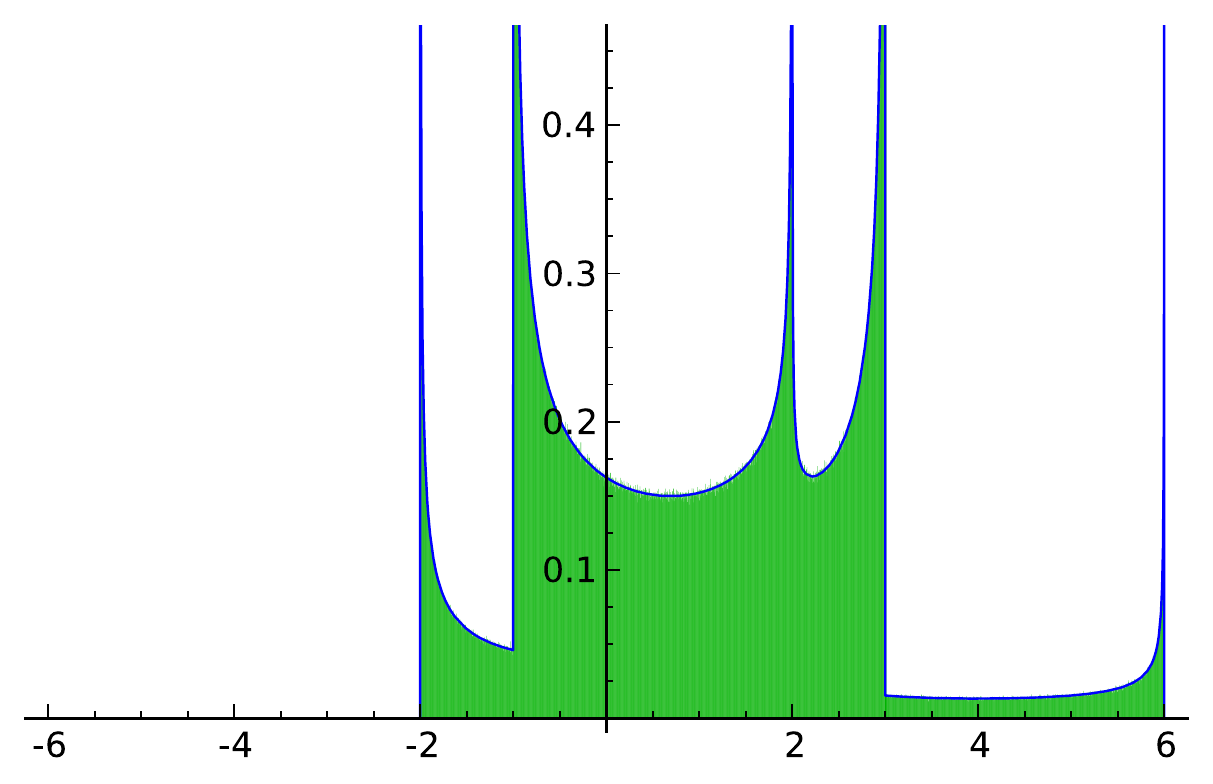}
\quad
\includegraphics[scale=0.48]{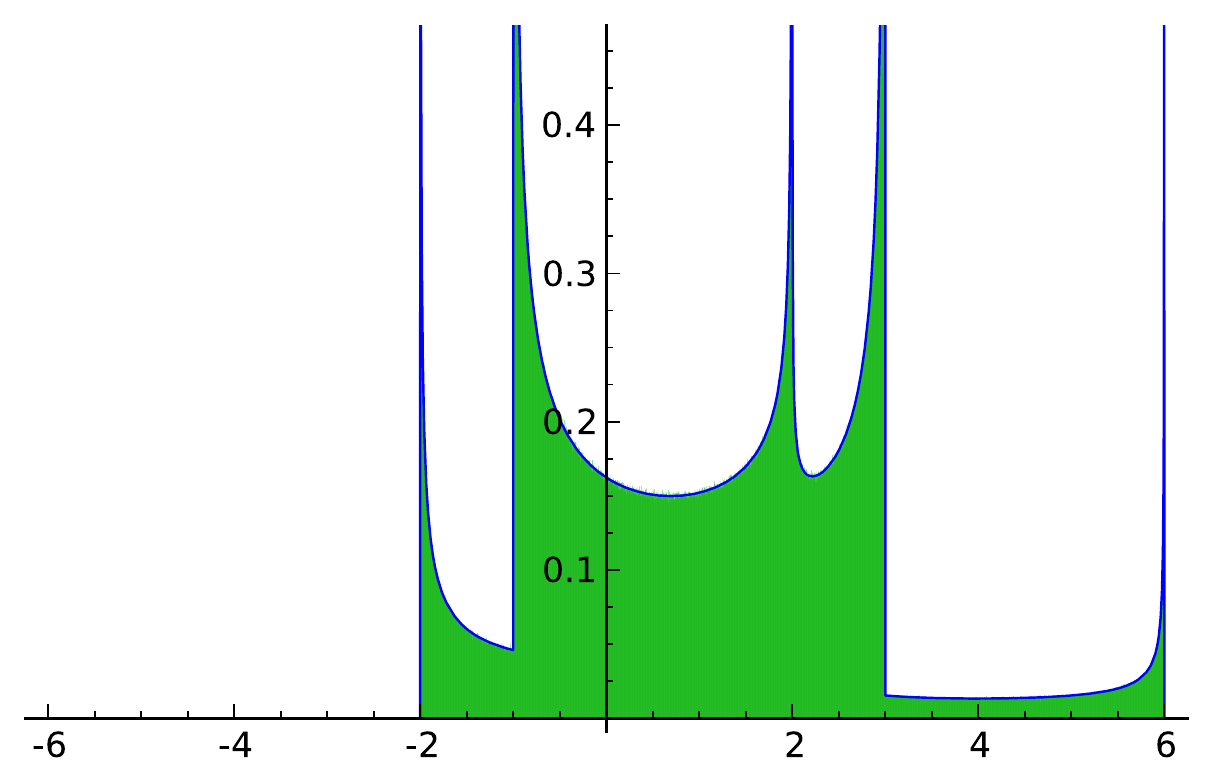}
\caption{$a_2$ density function for Sato-Tate group $T$ and $a_2(\mathfrak{p})$ histograms for the curve $y^2 = x^6 + 6x^5 - 20x^4 + 20x^3 - 20x^2 - 8x + 8$ over $\Q(\sqrt{-2})$ for $\norm{\mathfrak{p}}\le 2^N$, with $N=12,14,\ldots,30$.}\label{figure:Ta2hist}
\end{figure}
\end{center}

\begin{center}
\begin{figure}
\hspace{-30pt}
\includegraphics[scale=0.5]{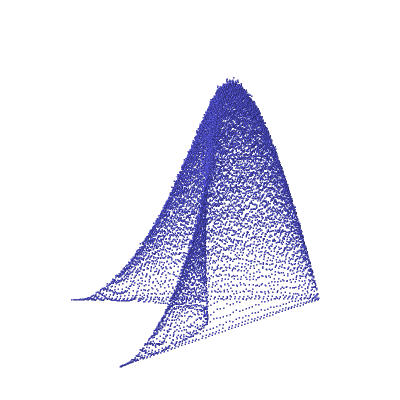}
\includegraphics[scale=0.5]{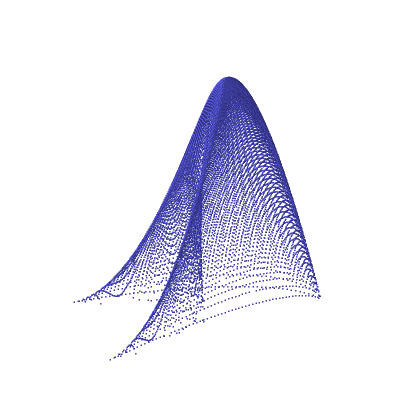}
\caption{Joint $a_1(\mathfrak{p})$ and $a_2(\mathfrak{p})$ statistics for the curve $y^2=x^5-x+1$ (left), and the joint density function for the Sato-Tate group $\USp(4)$ (right).
The vertical scale is exaggerated; the peak at $a_1=0$ and $a_2=2/3$ has height $8/\left(\pi^2\sqrt{27}\right)$.}\label{figure:USp4joint}

\end{figure}
\end{center}

We may also compare the joint statistics of $a_1(\mathfrak{p})$ and $a_2(\mathfrak{p})$ for a given curve~$C$ with the joint density function of $a_1$ and $a_2$ for the corresponding Sato-Tate group $\ST_{\Jac(C)}$.
Figure~\ref{figure:USp4joint} shows a plot of these joint statistics for the curve $y^2=x^5-x+1$, using the bound $\norm{\mathfrak{p}}\le 2^{30}$, alongside a plot of the joint density function for its Sato-Tate group $\USp(4)$, computed using (\ref{eq:rho_a1a2}) and plotted at the same number of points.

With $N=2^{30}$, we must compute $a_i(\mathfrak{p})$ for more than 54 million values of~$\mathfrak{p}$ (for each curve).
To make such computations practical, we employ the optimizations described in \cite{KS08}, as well as several further improvements recently incorporated in the \texttt{smalljac} software library \cite{S11b}; the most notable of these is the use of ideas in \cite{GHM08} to efficiently implement the group operation in the Jacobian of curves defined by a sextic equation.
For all but very small values of $\norm{\mathfrak{p}}$ we do not use (\ref{eq:a1a2curvecount}) to compute $a_i(\mathfrak{p})$, but instead apply
\begin{align}
a_1(\mathfrak{p}) &= q^{-1/2}\frac{\#\Jac(C)(\F_q)-\#\Jac(\tilde{C})(\F_q)}{2(q + 1)},\\
a_2(\mathfrak{p}) &= q^{-1}\frac{\#\Jac(C)(\F_q)+\#\Jac(\tilde{C})(\F_q)-2(q^2+1)}{2}.\notag
\end{align}
Here $\tilde{C}$ denotes a nonisomorphic quadratic twist of $C$ over $\F_q$.
The computations of the group orders $\#\Jac(C)$ and $\#\Jac(\tilde{C})$ are performed using generic group algorithms described in \cite{S07} and \cite{S11a}.
As discussed in \cite{KS08}, the asymptotically faster $p$-adic and $\ell$-adic methods available are not practically faster in genus 2 for the range of $N$ considered here.

In cases where $k\ne \Q$, we may take advantage of the fact that the moment statistics are essentially determined by the degree-1 primes $\mathfrak{p}$, allowing us to work entirely over prime fields.
We can also exploit the situation where $C/k$ is actually defined over $\Q$, in which case $a_i(\mathfrak{p})$ depends only on $p=\norm{\mathfrak{p}}$.
In this situation it suffices to compute $a_i(\mathfrak{p})$ for just one prime of norm $p$ and then weight it with the correct multiplicity, as determined by the number of linear factors of a defining polynomial for $k/\Q$ in $\F_p[x]$.

\subsection{An exhaustive search}\label{subsection:search}

The general methodology described above allows us to numerically test the Sato-Tate conjecture for individual curves.
Using more specialized techniques, we can efficiently analyze $L$-polynomial distributions for many curves at once.
This was originally done to provide an empirical conjecture for the classification of Sato-Tate groups;
it now serves as a partial check of the completeness of this classification.
Of course, one cannot really be convinced of the completeness without the proofs of the theorems;
after all, a similar search was described in \cite{KS09} and found considerably fewer groups.\footnote{The search in \cite{KS09} only looked at $a_1$-distributions of genus 2 curves over $\Q$, finding 23 of the 26 distributions identified here.}

To search for curves with exceptional $L$-polynomial distributions, we considered every nonsingular curve of the form $y^2=f(x)$ where $f$ is a monic polynomial of degree 5 or 6 whose coefficients lie in the interval $[-128,128)$.
This amounts to more than $2^{48}$ distinct curve equations, of which approximately $2^{47}$ are nonisomorphic.
This is much a larger range than was used in \cite{KS09}, which examined some $2^{35}$ curves, yet it actually required less computational effort.
Here we summarize some of the optimizations that made this possible.

Of the 34 Sato-Tate groups that can arise for a genus 2 curve over $\Q$, all but eight have a density $z_1/c \ge 1/2$ of zero traces.
The eight that do not, correspond to the first eight distributions listed in Tables 11 and 13 of \cite{KS09}, where one can already find representative curves with small coefficients.
Thus we chose to focus our search on curves with $z_1/c\ge 1/2$, allowing us to quickly discard curves that do not exhibit an abundance of zero traces at small primes.

For a suitably chosen bound $B$, we imposed the constraint
\[
\pi(B) - 2z(C,B)\le 3
\]
where $\pi(B)$ counts the primes $p\le B$ and $z(C,B)$ counts the primes $p\le B$ where $C$ has good reduction and $\#C(\F_p)=p+1$.
By initially checking this constraint with a small value of $B$, we can very quickly discard the vast majority of curves from further consideration.
This procedure causes us to ignore some curves with $z_1/c\ge 1/2$, but on average we expect to discard no more than half of the exceptional curves that we seek.

In order to distinguish exceptional curves, and to provisionally identify the Sato-Tate group $G$ for each curve $C$, we computed various statistics for $C$ up to a larger bound $B$ to obtain a ``signature" $\sigma(C,B)$, that can be compared to signatures $\sigma(G)$ derived from the group invariants defined in Section \ref{subsection:atlas}.
Let $\hat z_{i,j}$ denote the integer $48z_{i,j}/c$.  We define $\sigma(G)$ to be the tuple of integers
\[
\sigma(G) = \bigl(\hat z_{1,0}, \hat z_{2,-2}, \hat z_{2,-1}, \hat z_{2,0}, \hat z_{2,1}, \hat z_{2,2}, \Exp[a_1^2], \Exp[a_2^4], \Exp[a_2], \Exp[a_2^2], \Exp[a_2^3]\bigr),
\]
which suffices to uniquely distinguish all 55 of the Sato-Tate groups listed in Theorem~\ref{Sato-Tate axioms groups}.
Given a curve~$C$ and a bound~$B$, one can compute the analogous tuple $\sigma(C,B)$ by computing the corresponding statistics and rounding to the nearest integer.
We note that the number of components~$c$ is typically not known \emph{a priori}, but the ratios $z_{i,j}/c$ and the corresponding values of $\hat z_{i,j}$ can be computed without knowing $c$.

We now outline the search algorithm for exceptional curves $C$ of the form $y^2=f(x)$, where $f(x)=\sum f_ix^i$ is a monic sextic with $f_i\in I=[-R,R)$ and $f_5\ge 0$, using bounds $B_1$, $B_2$, and $B_3$.
For each combination of $f_2,f_3,f_4,f_5$ we perform the following three steps.

\begin{enumerate}
\item
For odd primes $p\le B_1$, count points on the curve $C/\F_p$ defined by $y^2=f(x)$ for every value of $f_0,f_1\in \F_p$, using the method of \cite[\S3]{KS08}.  Let $z_p(f_0,f_1)=1$ if $\#C(\F_p)=p+1$ and 0 otherwise.
For each $f_0,f_1\in I$, compute $z(C,B_1)=\sum_p z_p(f_0, f_1)$ for the curve $C/\Q$ defined by $y^2=f(x)$.
If $2z(C,B_1) < \pi(B_1)-3$, then reject $C$.

\item
For each remaining curve $C$, compute $z(C,B_2)$.\\
If $2z(C,B_2) < \pi(B_2)-3$ then reject $C$.

\item
For each remaining curve $C$, initialize $B$ to $B_3$ and compute $\sigma(C,B)$ using $L$-polynomial data for $C$ at primes $p\le B_3$.
Then increase $B$ by $50\%$ and repeat until $\sigma(C,B)$ is stable for three consecutive values of $B$.
\end{enumerate}

In our search, we used $R=128$ as the coefficient bound, and the prime bounds $B_1=83$, $B_2=1229$, and $B_3=2741$ in each of the 3 steps, values that were chosen after some initial performance testing.
Using $B_1=83$, fewer than~1 in 100,000 curves pass step 1, and the average time spent on each curve is very small: about 100 nanoseconds on a 3.0 GHz AMD Phenom II core.
With $B_2=1229$, fewer than 1 in 100 of the curves that pass step 1 also pass step~2.
Thus out of a total of $2^{48}$ curves, we only needed to compute signatures for some ten million curves.
On average, this takes 1--2 seconds per curve, although in particularly difficult cases, it may take as much as a minute.
Overall, we spent an average of less than 200 nanoseconds per curve.

The search found curves with matching signatures for all 26 of the 34 genus 2 Sato-Tate groups over $\Q$ that have $z/c \ge 1/2$.
Indeed, we found at least 3 curves for each group that are not isomorphic over $\Q$.
As can be seen in Table \ref{table:curves}, in each case we found a representative curve with integer coefficients of absolute value at most 60.
Thus \emph{a posteriori}, we see that we could have used $R=64$, rather than $R=128$, which would have reduced the search time dramatically.

\section{Tables}

In this section, we give the tables listing the Sato-Tate groups in genus 2 identified in \S\ref{section:STgroups}
(Table~\ref{table:STgroups}), the moments of $a_1$ and $a_2$ computed in \S\ref{subsection:moments}
(Tables~\ref{table:a1moments} and~\ref{table:a2moments}), the curves analyzed in
\S\ref{section:examples} realizing each Sato-Tate group
(Table~\ref{table:curves}),
and the automorphism data needed to verify these Sato-Tate groups
(Tables~\ref{table:automorphisms} and~\ref{table:auxpols}).

To make these tables more comprehensible, let us recall in each case what data is being tabulated.
In Table~\ref{table:STgroups}, each line corresponds to one of the 55 groups~$G$ named in Theorem~\ref{group classification}.
The quantities $d$ and $c$ indicate the dimension of $G$ and the order of the component group $G/G^0$, whose isomorphism
class is also given. To partially determine the Galois type, we list the $\R$-algebra $\End(A)_{\R}$, i.e., the fixed
subalgebra of $\End(A_K)_{\R}$ under the action of $\Gal(K/k)$ as determined using Proposition~\ref{Galois type from ST group}. (It is not necessary to list the fixed subalgebras under
subgroups of $\Gal(K/k)$, as this can be inferred from the rows of the table corresponding to those subgroups.)
We also list the label associated to the Galois type by Theorem~\ref{main-section4} (or * in the three cases that can not arise
from abelian surfaces).
The quantities $z_1$ and $z_2$ count components of $H$ on which $a_1$ and $a_2$ are constant; see
\S\ref{subsection:computations}. The quantities $M[a_1^2]$ and $M[a_2]$ are some initial terms of moment sequences
that are described more thoroughly in the tables that follow.

Tables~\ref{table:a1moments} and~\ref{table:a2moments} provide explicit formulas and initial terms for the~$a_1$ and~$a_2$
moment sequences associated to each group in Table~\ref{table:STgroups},
as computed using the methods of \S\ref{subsection:atlas}.
Note that the 55 groups only give rise to 37 distinct $a_1$ moment sequences, corresponding to 37 distinct Sato-Tate trace distributions, of which 26 can arise over $\Q$.
Each of these distributions has been assigned an identifier of the form $\#N$ consistent with the numbering used in \cite{KS09};
note that only the indices 1 to 23 correspond to distributions found in \cite{KS09}.
By contrast, the $a_2$ moment sequences in Table~\ref{table:a2moments} are all distinct with one exception: the groups $F_a$ and $F_{ab}$ have identical $a_2$ distributions
(but distinct $a_1$ distributions).

Table~\ref{table:curves} lists the example curves used in \S\ref{section:examples} to prove that there do exist 52 distinct
Galois types arising from abelian surfaces over general number fields, of which 34 do occur over $\Q$. For each curve,
we indicate the field of definition~$k$ and the minimal extension $K/k$ over which all endomorphisms of its Jacobian are defined.
It is proved in \S\ref{section:examples} that the field $K$ and the Galois type agree with the claimed values;
in most cases, the proof makes use of certain noncommuting automorphisms $\alpha$ and $\gamma$ of the curve.
These automorphisms are listed in Table~\ref{table:automorphisms}, together with the value of the intermediate field $M$
used in the alternate description of the Galois type in \S \ref{section:Galois}.
To make things more readable, some rather complicated polynomials appearing in the definitions of the automorphisms
have been moved to Table~\ref{table:auxpols}.
Note that in each formula in Tables~\ref{table:automorphisms} and~\ref{table:auxpols}, the symbols $a$ and $b$ represent elements
of $K$ as presented in the corresponding line of Table~\ref{table:curves}; consequently,
the meaning of these symbols varies from
line to line.
\newpage

\setlength{\evensidemargin}{0mm}
\setlength{\oddsidemargin}{8mm}
\setlength{\textwidth}{140mm}

\begin{table}
\begin{center}
\footnotesize
\setlength{\extrarowheight}{0.5pt}
\caption{Sato-Tate groups in genus 2} \label{table:STgroups}
\vspace{6pt}
\begin{tabular}{rrl|lll|rlll}
$d$ & $c$  & $G$          & $[G/G^0]$              & $\End(A)_{\R}$ & Galois type                         &$z_1$ & $z_2$       & $M[a_1^2]$  & $M[a_2]$\\\hline\vspace{-10pt}\\
$1$ & $1$  & $C_1$        & $\cyc{1}$              & $\M_2(\C)$     & $\bF[\cyc{1}]$                      &  $0$ & $0,0,0,0,0$ & $8,96,1280$ & $4,18,88$\\
$1$ & $2$  & $C_2$        & $\cyc{2}$              & $\C \times \C$ & $\bF[\cyc{2}]$                      &  $1$ & $0,0,0,0,0$ & $4,48,640$  & $2,10,44$\\
$1$ & $3$  & $C_3$        & $\cyc{3}$              & $\C \times \C$ & $\bF[\cyc{3}]$                      &  $0$ & $0,0,0,0,0$ & $4,36,440$  & $2,8,34$\\
$1$ & $4$  & $C_4$        & $\cyc{4}$              & $\C \times \C$ & $\bF[\cyc{4}]$                      &  $1$ & $0,0,0,0,0$ & $4,36,400$  & $2,8,32$\\
$1$ & $6$  & $C_6$        & $\cyc{6}$              & $\C \times \C$ & $\bF[\cyc{6}]$                      &  $1$ & $0,0,0,0,0$ & $4,36,400$  & $2,8,32$\\
$1$ & $4$  & $D_2$        & $\dih{2}$              & $\C$           & $\bF[\dih{2}]$ 		              &  $3$ & $0,0,0,0,0$ & $2,24,320$  & $1,6,22$\\
$1$ & $6$  & $D_3$        & $\dih{3}$              & $\C$           & $\bF[\dih{3}]$                      &  $3$ & $0,0,0,0,0$ & $2,18,220$  & $1,5,17$\\
$1$ & $8$  & $D_4$        & $\dih{4}$              & $\C$           & $\bF[\dih{4}]$                      &  $5$ & $0,0,0,0,0$ & $2,18,200$  & $1,5,16$\\
$1$ & $12$ & $D_6$        & $\dih{6}$              & $\C$           & $\bF[\dih{6}]$                      &  $7$ & $0,0,0,0,0$ & $2,18,200$  & $1,5,16$\\
$1$ & $12$ & $T$          & $\alt{4}$              & $\C$           & $\bF[\alt{4}]$                      &  $3$ & $0,0,0,0,0$ & $2,12,120$  & $1,4,12$\\
$1$ & $24$ & $O$          & $\sym{4}$              & $\C$           & $\bF[\sym{4}]$                      &  $9$ & $0,0,0,0,0$ & $2,12,100$  & $1,4,11$\\
$1$ & $2$  & $J(C_1)$     & $\cyc{2}$              & $\HH$          & $\bF[\cyc{2},\cyc{1},\HH]$          &  $1$ & $1,0,0,0,0$ & $4,48,640$  & $1,11,40$\\
$1$ & $4$  & $J(C_2)$     & $\dih{2}$              & $\C$           & $\bF[\dih{2},\cyc{2},\HH]$          &  $3$ & $1,0,0,0,1$ & $2,24,320$  & $1,7,22$\\
$1$ & $6$  & $J(C_3)$     & $\cyc{6}$              & $\C$           & $\bF[\cyc{6},\cyc{3},\HH]$          &  $3$ & $1,0,0,2,0$ & $2,18,220$  & $1,5,16$\\
$1$ & $8$  & $J(C_4)$     & $\cyc{4}\times\cyc{2}$ & $\C$           & $\bF[\cyc{4}\times\cyc{2},\cyc{4}]$ &  $5$ & $1,0,2,0,1$ & $2,18,200$  & $1,5,16$\\
$1$ & $12$ & $J(C_6)$     & $\cyc{6}\times\cyc{2}$ & $\C$           & $\bF[\cyc{6}\times\cyc{2},\cyc{6}]$ &  $7$ & $1,2,0,2,1$ & $2,18,200$  & $1,5,16$\\
$1$ & $8$  & $J(D_2)$     & $\dih{2}\times\cyc{2}$ & $\R$           & $\bF[\dih{2}\times\cyc{2},\dih{2}]$ &  $7$ & $1,0,0,0,3$ & $1,12,160$  & $1,5,13$\\
$1$ & $12$ & $J(D_3)$     & $\dih{6}$              & $\R$           & $\bF[\dih{6},\dih{3},\HH]$          &  $9$ & $1,0,0,2,3$ & $1,9,110$   & $1,4,10$\\
$1$ & $16$ & $J(D_4)$     & $\dih{4}\times\cyc{2}$ & $\R$           & $\bF[\dih{4}\times\cyc{2},\dih{4}]$ & $13$ & $1,0,2,0,5$ & $1,9,100$   & $1,4,10$\\
$1$ & $24$ & $J(D_6)$     & $\dih{6}\times\cyc{2}$ & $\R$           & $\bF[\dih{6}\times\cyc{2},\dih{6}]$ & $19$ & $1,2,0,2,7$ & $1,9,100$   & $1,4,10$\\
$1$ & $24$ & $J(T)$       & $\alt{4}\times\cyc{2}$ & $\R$           & $\bF[\alt{4}\times\cyc{2},\alt{4}]$ & $15$ & $1,0,0,8,3$ & $1,6,60$    & $1,3,7$\\
$1$ & $48$ & $J(O)$       & $\sym{4}\times\cyc{2}$ & $\R$           & $\bF[\sym{4}\times\cyc{2},\sym{4}]$ & $33$ & $1,0,6,8,9$ & $1,6,50$    & $1,3,7$\\
$1$ & $2$  & $C_{2,1}$    & $\cyc{2}$              & $\M_2(\R)$     & $\bF[\cyc{2},\cyc{1},\M_2(\R)]$     &  $1$ & $0,0,0,0,1$ & $4,48,640$  & $3,11,48$\\
$1$ & $4$  & $C_{4,1}$    & $\cyc{4}$              & $\C$           & $\bF[\cyc{4},\cyc{2}]$              &  $3$ & $0,0,2,0,0$ & $2,24,320$  & $1,5,22$\\
$1$ & $6$  & $C_{6,1}$    & $\cyc{6}$              & $\C$           & $\bF[\cyc{6},\cyc{3},\M_2(\R)]$     &  $3$ & $0,2,0,0,1$ & $2,18,220$  & $1,5,18$\\
$1$ & $4$  & $D_{2,1}$    & $\dih{2}$              & $\R \times \R$ & $\bF[\dih{2},\cyc{2},\M_2(\R)]$     &  $3$ & $0,0,0,0,2$ & $2,24,320$  & $2,7,26$\\
$1$ & $8$  & $D_{4,1}$    & $\dih{4}$              & $\R$           & $\bF[\dih{4},\dih{2}]$              &  $7$ & $0,0,2,0,2$ & $1,12,160$  & $1,4,13$\\
$1$ & $12$ & $D_{6,1}$    & $\dih{6}$              & $\R$           & $\bF[\dih{6},\dih{3},\M_2(\R)]$     &  $9$ & $0,2,0,0,4$ & $1,9,110$   & $1,4,11$\\
$1$ & $6$  & $D_{3,2}$    & $\dih{3}$              & $\R \times \R$ & $\bF[\dih{3},\cyc{3}]$              &  $3$ & $0,0,0,0,3$ & $2,18,220$  & $2,6,21$\\
$1$ & $8$  & $D_{4,2}$    & $\dih{4}$              & $\R \times \R$ & $\bF[\dih{4},\cyc{4}]$              &  $5$ & $0,0,0,0,4$ & $2,18,200$  & $2,6,20$\\
$1$ & $12$ & $D_{6,2}$    & $\dih{6}$              & $\R \times \R$ & $\bF[\dih{6},\cyc{6}]$              &  $7$ & $0,0,0,0,6$ & $2,18,200$  & $2,6,20$\\
$1$ & $24$ & $O_1$        & $\sym{4}$              & $\R$           & $\bF[\sym{4},\alt{4}]$              & $15$ & $0,0,6,0,6$ & $1,6,60$    & $1,3,8$\\
$3$ & $1$  & $E_1$        & $\cyc{1}$              & $\M_2(\R)$     & $\bE[\cyc{1}]$                      &  $0$ & $0,0,0,0,0$ & $4,32,320$  & $3,10,37$\\
$3$ & $2$  & $E_2$        & $\cyc{2}$              & $\C$           & $\bE[\cyc{2}, \C]$                  &  $1$ & $0,0,0,0,0$ & $2,16,160$  & $1,6,17$\\
$3$ & $3$  & $E_3$        & $\cyc{3}$              & $\C$           & $\bE[\cyc{3}]$                      &  $0$ & $0,0,0,0,0$ & $2,12,110$  & $1,4,13$\\
$3$ & $4$  & $E_4$        & $\cyc{4}$              & $\C$           & $\bE[\cyc{4}]$                      &  $1$ & $0,0,0,0,0$ & $2,12,100$  & $1,4,11$\\
$3$ & $6$  & $E_6$        & $\cyc{6}$              & $\C$           & $\bE[\cyc{6}]$                      &  $1$ & $0,0,0,0,0$ & $2,12,100$  & $1,4,11$\\
$3$ & $2$  & $J(E_1)$     & $\cyc{2}$              & $\R \times \R$ & $\bE[\cyc{2},\R\times\R]$           &  $1$ & $0,0,0,0,0$ & $2,16,160$  & $2,6,20$\\
$3$ & $4$  & $J(E_2)$     & $\dih{2}$              & $\R$           & $\bE[\dih{2}]$                      &  $3$ & $0,0,0,0,0$ & $1,8,80$    & $1,4,10$\\
$3$ & $6$  & $J(E_3)$     & $\dih{3}$              & $\R$           & $\bE[\dih{3}]$                      &  $3$ & $0,0,0,0,0$ & $1,6,55$    & $1,3,8$\\
$3$ & $8$  & $J(E_4)$     & $\dih{4}$              & $\R$           & $\bE[\dih{4}]$                      &  $5$ & $0,0,0,0,0$ & $1,6,50$    & $1,3,7$\\
$3$ & $12$ & $J(E_6)$     & $\dih{6}$              & $\R$           & $\bE[\dih{6}]$                      &  $7$ & $0,0,0,0,0$ & $1,6,50$    & $1,3,7$\\
$2$ & $1$  & $F_\nothing$ & $\cyc{1}$              & $\C \times \C$ & $\bD[\cyc{1}]$                      &  $0$ & $0,0,0,0,0$ & $4,36,400$  & $2,8,32$\\
$2$ & $2$  & $F_a$        & $\cyc{2}$              & $\R \times \C$ & $\bD[\cyc{2},\R\times\C]$           &  $0$ & $0,0,0,0,1$ & $3,21,210$  & $2,6,20$\\
$2$ & $2$  & $F_c$        & $\cyc{2}$              & *              & *                                   &  $1$ & $0,0,0,0,0$ & $2,18,200$  & $1,5,16$\\ 
$2$ & $2$  & $F_{ab}$     & $\cyc{2}$              & $\R \times \R$ & $\bD[\cyc{2},\R\times\R]$           &  $1$ & $0,0,0,0,1$ & $2,18,200$  & $2,6,20$\\
$2$ & $4$  & $F_{ac}$     & $\cyc{4}$              & $\R$           & $\bD[\cyc{4}]$                      &  $3$ & $0,0,2,0,1$ & $1,9,100$   & $1,3,10$\\
$2$ & $4$  & $F_{a,b}$    & $\dih{2}$              & $\R \times \R$ & $\bD[\dih{2}]$                      &  $1$ & $0,0,0,0,3$ & $2,12,110$  & $2,5,14$\\
$2$ & $4$  & $F_{ab,c}$   & $\dih{2}$              & *              & *                                   &  $3$ & $0,0,0,0,1$ & $1,9,100$   & $1,4,10$\\ 
$2$ & $8$  & $F_{a,b,c}$  & $\dih{4}$              & *              & *                                   &  $5$ & $0,0,2,0,3$ & $1,6,55$    & $1,3,7$\\  
$4$ & $1$  & $G_{1,3}$    & $\cyc{1}$              & $\R \times \C$ & $\bC[\cyc{1}]$                      &  $0$ & $0,0,0,0,0$ & $3,20,175$  & $2,6,20$\\
$4$ & $2$  & $N(G_{1,3})$ & $\cyc{2}$              & $\R \times \R$ & $\bC[\cyc{2}]$                      &  $0$ & $0,0,0,0,1$ & $2,11,90$   & $2,5,14$\\
$6$ & $1$  & $G_{3,3}$    & $\cyc{1}$              & $\R \times \R$ & $\bB[\cyc{1}]$                      &  $0$ & $0,0,0,0,0$ & $2,10,70$   & $2,5,14$\\
$6$ & $2$  & $N(G_{3,3})$ & $\cyc{2}$              & $\R$           & $\bB[\cyc{2}]$                      &  $1$ & $0,0,0,0,0$ & $1,5,35$    & $1,3,7$\\
$10$& $1$  & $\USp(4)$    & $\cyc{1}$              & $\R$           & $\bA[\cyc{1}]$                      &  $0$ & $0,0,0,0,0$ & $1,3,14$    & $1,2,4$\\\hline
\end{tabular}
\end{center}
\end{table}

\begin{table}
\begin{center}
\footnotesize
\setlength{\extrarowheight}{0.5pt}
\caption{Moments of $a_1$ for Sato-Tate groups in genus 2}\label{table:a1moments}
\vspace{6pt}
\begin{tabular}{llrrrrrl}
$G$ & $M_n = \Exp[a_1^n]$ & $M_2$ & $M_4$ & $M_6$ & $M_8$ & $M_{10}$ & Type [KS]\\\hline\vspace{-8pt}\\
$C_1$ & $2^nb_n$ & $8$ & $96$ & $1280$ & $17920$ & $258048$ &\#27\\
$C_2$ & $\nicefrac{1}{2}(2^n+0^n)b_n$ & $4$ & $48$ & $640$ & $8960$ & $129024$ & \#13\\
$C_3$ & $\nicefrac{1}{3}(2^n+2)b_n$ & $4$ & $36$ & $440$ & $6020$ & $86184$ & \#28\\
$C_4$ & $\nicefrac{1}{4}(2^n+0^n+2\cdot 2^{n/2})b_n$ & $4$ & $36$ & $400$ & $5040$ & $68544$ & \#29\\
$C_6$ & $\nicefrac{1}{6}(2^n+0^n+2+2\cdot 3^{n/2})b_n$ & $4$ & $36$ & $400$ & $4900$ & $63504$ & \#30\\
$D_2$ & $\nicefrac{1}{4}(2^n+3\cdot 0^n)b_n$ & $2$ & $24$ & $320$ & $4480$ & $64512$& \#21\\
$D_3$ & $\nicefrac{1}{6}(2^n+3\cdot 0^n+2)b_n$ & $2$ & $18$ & $220$ & $3010$ & $43092$ & \#12\\
$D_4$ & $\nicefrac{1}{8}(2^n+5\cdot 0^n+2\cdot 2^{n/2})b_n$ & $2$ & $18$ & $200$ & $2520$ & $34272$ & \#17\\
$D_6$ & $\nicefrac{1}{12}(2^n+7\cdot 0^n+2+2\cdot 3^{n/2})b_n$ & $2$ & $18$ & $200$ & $2450$ & $31752$ & \#15\\
$T$ & $\nicefrac{1}{12}(2^n+3\cdot 0^n+8)b_n$ & $2$ & $12$ & $120$ & $1540$ & $21672$ & \#31 \\
$O$ & $\nicefrac{1}{24}(2^n+9\cdot 0^n+8+6\cdot 2^{n/2})b_n$ & $2$ & $12$ & $100$ & $1050$ & $12852$ & \#32\\
$J(C_1)$ & $\nicefrac{1}{2}(2^n+0^n)b_n$ & $4$ & $48$ & $640$ & $8960$ & $129024$ & \#13\\
$J(C_2)$ & $\nicefrac{1}{4}(2^n+3\cdot 0^n)b_n$ & $2$ & $24$ & $320$ & $4480$ & $64512$ & \#21\\
$J(C_3)$ & $\nicefrac{1}{6}(2^n+3\cdot 0^n+2)b_n$ & $2$ & $18$ & $220$ & $3010$ & $43092$ & \#12\\
$J(C_4)$ & $\nicefrac{1}{8}(2^n+5\cdot 0^n+2\cdot 2^{n/2})b_n$ & $2$ & $18$ & $200$ & $2520$ & $34272$ & \#17\\
$J(C_6)$ & $\nicefrac{1}{12}(2^n+7\cdot 0^n+2+2\cdot 3^{n/2})b_n$ & $2$ & $18$ & $200$ & $2450$ & $31752$ & \#15\\
$J(D_2)$ & $\nicefrac{1}{8}(2^n+7\cdot 0^n)b_n$ & $1$ & $12$ & $160$ & $2240$ & $32256$ & \#23\\
$J(D_3)$ & $\nicefrac{1}{12}(2^n+9\cdot 0^n+2)b_n$ & $1$ & $9$ & $110$ & $1505$ & $21546$ & \#20 \\
$J(D_4)$ & $\nicefrac{1}{16}(2^n+13\cdot 0^n+2\cdot 2^{n/2})b_n$ & $1$ & $9$ & $100$ & $1260$ & $17136$ & \#22\\
$J(D_6)$ & $\nicefrac{1}{24}(2^n+19\cdot 0^n+2+2\cdot 3^{n/2})b_n$ & $1$ & $9$ & $100$ & $1225$ & $15876$ & \#24\\
$J(T)$ & $\nicefrac{1}{24}(2^n+15\cdot 0^n+8)b_n$ & $1$ & $6$ & $60$ & $770$ & $10836$ & \#25\\
$J(O)$ & $\nicefrac{1}{48}(2^n+33\cdot 0^n+8+6\cdot 2^{n/2})b_n$ & $1$ & $6$ & $50$ & $525$ & $6426$ & \#26\\
$C_{2,1}$ & $\nicefrac{1}{2}(2^n+0^n)b_n$ & $4$ & $48$ & $640$ & $8960$ & $129024$ & \#13\\
$C_{4,1}$ & $\nicefrac{1}{4}(2^n+3\cdot 0^n)b_n$ & $2$ & $24$ & $320$ & $4480$ & $64512$ & \#21\\
$C_{6,1}$ & $\nicefrac{1}{6}(2^n+3\cdot 0^n+2)b_n$ & $2$ & $18$ & $220$ & $3010$ & $43092$ & \#12\\
$D_{2,1}$ & $\nicefrac{1}{4}(2^n+3\cdot 0^n)b_n$ & $2$ & $24$ & $320$ & $4480$ & $64512$ & \#21\\
$D_{4,1}$ & $\nicefrac{1}{8}(2^n+7\cdot 0^n)b_n$ & $1$ & $12$ & $160$ & $2240$ & $32256$ & \#23\\
$D_{6,1}$ & $\nicefrac{1}{12}(2^n+9\cdot 0^n+2)b_n$ & $1$ & $9$ & $110$ & $1505$ & $21546$ & \#20\\
$D_{3,2}$ & $\nicefrac{1}{6}(2^n+3\cdot 0^n+2)b_n$ & $2$ & $18$ & $220$ & $3010$ & $43092$ & \#12\\
$D_{4,2}$ & $\nicefrac{1}{8}(2^n+5\cdot 0^n+2\cdot 2^{n/2})b_n$ & $2$ & $18$ & $200$ & $2520$ & $34272$ & \#17\\
$D_{6,2}$ & $\nicefrac{1}{12}(2^n+7\cdot 0^n+2+2\cdot 3^{n/2})b_n$ & $2$ & $18$ & $200$ & $2450$ & $31752$ & \#15\\
$O_1$ & $\nicefrac{1}{24}(2^n+15\cdot 0^n+8)b_n$ & $1$ & $6$ & $60$ & $770$ & $10836$ & \#25\\
$E_1$ & $2^nc_n$ & $4$ & $32$ & $320$ & $3584$ & $43008$ & \#5\\
$E_2$ & $\nicefrac{1}{2}(2^n+0^n)c_n$ & $2$ & $16$ & $160$ & $1792$ & $21504$ & \#11\\
$E_3$ & $\nicefrac{1}{3}(2^n+2)c_n$ & $2$ & $12$ & $110$ & $1204$ & $14364$ & \#4\\
$E_4$ & $\nicefrac{1}{4}(2^n+0^n+2\cdot 2^{n/2})c_n$ & $2$ & $12$ & $100$ & $1008$ & $11424$ & \#7\\
$E_6$ & $\nicefrac{1}{6}(2^n+0^n+2+2\cdot 3^{n/2})c_n$ & $2$ & $12$ & $100$ & $980$ & $10584$ & \#6\\
$J(E_1)$ & $\nicefrac{1}{2}(2^n+0^n)c_n$ & $2$ & $16$ & $160$ & $1792$ & $21504$ & \#11\\
$J(E_2)$ & $\nicefrac{1}{4}(2^n+3\cdot 0^n)c_n$ & $1$ & $8$ & $80$ & $896$ & $10752$ & \#18\\
$J(E_3)$ & $\nicefrac{1}{6}(2^n+3\cdot 0^n+2)c_n$ & $1$ & $6$ & $55$ & $602$ & $7182$ & \#10\\
$J(E_4)$ & $\nicefrac{1}{8}(2^n+5\cdot 0^n+2\cdot 2^{n/2})c_n$ & $1$ & $6$ & $50$ & $504$ & $5712$ & \#16\\
$J(E_6)$ & $\nicefrac{1}{12}(2^n+7\cdot 0^n+2+2\cdot 3^{n/2})c_n$ & $1$ & $6$ & $50$ & $490$ & $5292$ & \#14\\
$F_\nothing$ & $b_n^2$ & $4$ & $36$ & $400$ & $4900$ & $63504$ & \#33\\
$F_a$ & $\nicefrac{1}{2}(b_n+b_n^2)$ & $3$ & $21$ & $210$ & $2485$ & $31878$ & \#34\\
$F_c$ & $\nicefrac{1}{2}(b_n^2+0^n)$ & $2$ & $18$ & $200$ & $2450$ & $31752$ & \#35\\
$F_{ab}$ & $\nicefrac{1}{2}(b_n^2+0^n)$ & $2$ & $18$ & $200$ & $2450$ & $31752$ & \#35\\
$F_{ac}$ & $\nicefrac{1}{4}(b_n^2+3\cdot 0^n)$ & $1$ & $9$ & $100$ & $1225$ & $15876$ & \#19\\
$F_{a,b}$ & $\nicefrac{1}{4}(b_n^2+2b_n+0^n)$ & $2$ & $12$ & $110$ & $1260$ & $16002$ & \#8\\
$F_{ab,c}$ & $\nicefrac{1}{4}(b_n^2+3\cdot 0^n)$ & $1$ & $9$ & $100$ & $1225$ & $15876$ & \#19 \\
$F_{a,b,c}$ & $\nicefrac{1}{8}(b_n^2+2b_n+5\cdot 0^n)$ & $1$ & $6$ & $55$ & $630$ & $8001$ & \#37 \\
$G_{1,3}$ & $c_nb_{n+2}/2$ & $3$ & $20$ & $175$ & $1764$ & $19404$ & \#36\\
$N(G_{1,3})$ & $\nicefrac{1}{2}(c_nb_{n+2}/2+c_n)$ & $2$ & $11$ & $90$ & $889$ & $9723$ & \#3\\
$G_{3,3}$ & $c_nc_{n+2}$ & $2$ & $10$ & $70$ & $588$ & $5544$ & \#2\\
$N(G_{3,3})$ & $\nicefrac{1}{2}(c_nc_{n+2}+0^n)$ & $1$ & $5$ & $35$ & $294$ & $2772$ & \#9\\\vspace{2pt}
$\USp(4)$ & $c_nc_{n+4}-c_{n+2}^2$ & $1$ & $3$ & $14$ & $84$ & $594$ & \#1\\\hline
\end{tabular}
\end{center}
\end{table}

\begin{table}
\begin{center}
\footnotesize
\setlength{\extrarowheight}{0.5pt}
\caption{Moments of $a_2$ for Sato-Tate groups in genus 2}\label{table:a2moments}
\vspace{6pt}
\begin{tabular}{@{\extracolsep{-3pt}}llrrrrrr}
$G$ & $M_n = \Exp[a_2^n]$ & $M_1$ & $M_2$ & $M_3$ & $M_4$ & $M_5$\\\hline\vspace{-8pt}\\
$C_1$ & $b_{4,n}$ & $4$ & $18$ & $88$ & $454$ & $2424$\\
$C_2$ & $\nicefrac{1}{2}(b_{4,n}+b_n)$ & $2$ & $10$ & $44$ & $230$ & $1212$\\
$C_3$ & $\nicefrac{1}{3}(b_{4,n}+2b_{1,n})$ & $2$ & $8$ & $34$ & $164$ & $842$\\
$C_4$ & $\nicefrac{1}{4}(b_{4,n}+b_n+2b_{2,n})$ & $2$ & $8$ & $32$ & $150$ & $732$\\
$C_6$ & $\nicefrac{1}{6}(b_{4,n}+b_n+2b_{1,n}+2b_{3,n})$ & $2$ & $8$ & $32$ & $148$ & $712$\\
$D_2$ & $\nicefrac{1}{4}(b_{4,n}+3b_n)$ &$1$ & $6$ & $22$ & $118$ & $606$\\
$D_3$ & $\nicefrac{1}{6}(b_{4,n}+3b_n+2b_{1,n})$ & $1$ & $5$ & $17$ & $85$ & $421$\\
$D_4$ & $\nicefrac{1}{8}(b_{4,n}+5b_n+2b_{2,n})$& $1$ & $5$ & $16$ & $78$ & $366$\\
$D_6$ & $\nicefrac{1}{12}(b_{4,n}+7b_n+2b_{1,n}+2b_{3,n})$ & $1$ & $5$ & $16$ & $77$ & $356$\\
$T$ & $\nicefrac{1}{12}(b_{4,n}+3b_n+8b_{1,n})$ & $1$ & $4$ & $12$ & $52$ & $236$\\
$O$ & $\nicefrac{1}{24}(b_{4,n}+9b_n+8b_{1,n}+6b_{2,n})$& $1$ & $4$ & $11$ & $45$ & $181$\\
$J(C_1)$ & $\nicefrac{1}{2}(b_{4,n}+(-2)^n)$ & $1$ & $11$ & $40$ & $235$ & $1196$\\
$J(C_2)$ & $\nicefrac{1}{4}(b_{4,n}+b_n+2^n+(-2)^n)$ & $1$ & $7$ & $22$ & $123$ & $606$\\
$J(C_3)$ & $\nicefrac{1}{6}(b_{4,n}+2(b_{1,n}+1)+(-2)^n)$ & $1$ & $5$ & $16$ & $85$ & $416$\\
$J(C_4)$ & $\nicefrac{1}{8}(b_{4,n}+b_n+2(b_{2,n}+0^n)+2^n+(-2)^n)$ & $1$ & $5$ & $16$ & $79$ & $366$\\
$J(C_6)$ & $\nicefrac{1}{12}(b_{4,n}+b_n+2(b_{1,n}+b_{3,n}+1+(-1)^n)+2^n+(-2)^n)$& $1$ & $5$ & $16$ & $77$ & $356$\\
$J(D_2)$ & $\nicefrac{1}{8}(b_{4,n}+3(b_n+2^n)+(-2)^n)$ & $1$ & $5$ & $13$ & $67$ & $311$\\
$J(D_3)$ & $\nicefrac{1}{12}(b_{4,n}+3(b_n+2^n)+2(b_{1,n}+1)+(-2)^n)$ & $1$ & $4$ & $10$ & $48$ & $216$\\
$J(D_4)$ & $\nicefrac{1}{16}(b_{4,n}+5(b_n+2^n)+2(b_{2,n}+0^n)+(-2)^n)$ & $1$ & $4$ & $10$ & $45$ & $191$\\
$J(D_6)$ & $\nicefrac{1}{24}(b_{4,n}+2(b_{1,n}+b_{3,n}+1+(-1)^n)+7(b_n+2^n)+(-2)^n)$ & $1$ & $4$ & $10$ & $44$ & $186$\\
$J(T)$ & $\nicefrac{1}{24}(b_{4,n}+3(b_n+2^n)+8(b_{1,n}+1)+(-2)^n)$& $1$ & $3$ & $7$ & $29$ & $121$\\
$J(O)$ & $\nicefrac{1}{48}(b_{4,n}+9(b_n+2^n)+8(b_{1,n}+1)+6(b_{2,n}+0^n)+(-2)^n)$ & $1$ & $3$ & $7$ & $26$ & $96$\\
$C_{2,1}$ & $\nicefrac{1}{2}(b_{4,n}+2^n)$ & $3$ & $11$ & $48$ & $235$ & $1228$\\
$C_{4,1}$ & $\nicefrac{1}{4}(b_{4,n}+b_n+2\cdot 0^n)$ & $1$ & $5$ & $22$ & $115$ & $606$\\
$C_{6,1}$ & $\nicefrac{1}{6}(b_{4,n}+2(b_{1,n}+(-1)^n)+2^n)$ & $1$ & $5$ & $18$ & $85$ & $426$\\
$D_{2,1}$ & $\nicefrac{1}{4}(b_{4,n}+b_n+2^{n+1})$& $2$ & $7$ & $26$ & $123$ & $622$\\
$D_{4,1}$ & $\nicefrac{1}{8}(b_{4,n}+3b_n+2(2^n+0^n))$ & $1$ & $4$ & $13$ & $63$ & $311$\\
$D_{6,1}$ & $\nicefrac{1}{12}(b_{4,n}+3b_n+2(b_{1,n}+(-1)^n)+2^{n+2})$ & $1$ & $4$ & $11$ & $48$ & $221$\\
$D_{3,2}$ & $\nicefrac{1}{6}(b_{4,n}+2b_{1,n}+3\cdot 2^n)$ & $2$ & $6$ & $21$ & $90$ & $437$\\
$D_{4,2}$ & $\nicefrac{1}{8}(b_{4,n}+b_n+2b_{2,n}+2^{n+2})$ & $2$ & $6$ & $20$ & $83$ & $382$\\
$D_{6,2}$ & $\nicefrac{1}{12}(b_{4,n}+b_n+2(b_{1,n}+b_{3,n})+6\cdot 2^n)$ & $2$ & $6$ & $20$ & $82$ & $372$\\
$O_1$ & $\nicefrac{1}{24}(b_{4,n}+3b_n+8b_{1,n}+6(2^n+0^n))$& $1$ & $3$ & $8$ & $30$ & $126$\\
$E_1$ & $d_{4,n}$ & $3$ & $10$ & $37$ & $150$ & $654$\\
$E_2$ & $\nicefrac{1}{2}(d_{4,n}+d_n)$ & $1$ & $6$ & $17$ & $78$ & $322$\\
$E_3$ & $\nicefrac{1}{3}(d_{4,n}+2d_{1,n})$& $1$ & $4$ & $13$ & $52$ & $222$\\
$E_4$ & $\nicefrac{1}{4}(d_{4,n}+d_n+2d_{2,n})$ & $1$ & $4$ & $11$ & $46$ & $182$\\
$E_6$ & $\nicefrac{1}{6}(d_{4,n}+d_n+2(d_{1,n}+d_{3,n}))$ & $1$ & $4$ & $11$ & $44$ & $172$\\
$J(E_1)$ & $\nicefrac{1}{2}(d_{4,n}+(-1)^nd_n)$ & $2$ & $6$ & $20$ & $78$ & $332$\\
$J(E_2)$ & $\nicefrac{1}{4}(d_{4,n}+d_n+2(-1)^nd_n)$& $1$ & $4$ & $10$ & $42$ & $166$\\
$J(E_3)$ & $\nicefrac{1}{6}(d_{4,n}+2d_{1,n}+3(-1)^nd_n)$ & $1$ & $3$ & $8$ & $29$ & $116$\\
$J(E_4)$ & $\nicefrac{1}{8}(d_{4,n}+d_n+2d_{2,n}+4(-1)^nd_n)$ & $1$ & $3$ & $7$ & $26$ & $96$\\
$J(E_6)$ & $\nicefrac{1}{12}(d_{4,n}+d_n+2(d_{1,n}+d_{3,n})+6(-1)^nd_n)$ & $1$ & $3$ & $7$ & $25$ & $91$\\
$F_\nothing$ & $\bbn$ & $2$ & $8$ & $32$ & $148$ & $712$\\
$F_a$ & $\nicefrac{1}{2}(\bbn+2^n)$& $2$ & $6$ & $20$ & $82$ & $372$\\
$F_c$ & $\nicefrac{1}{2}(\bbn+b_n)$ & $1$ & $5$ & $16$ & $77$ & $356$\\
$F_{ab}$ & $\nicefrac{1}{2}(\bbn+2^n)$& $2$ & $6$ & $20$ & $82$ & $372$\\
$F_{ac}$ & $\nicefrac{1}{4}(\bbn+2\cdot 0^n+2^n)$ & $1$ & $3$ & $10$ & $41$ & $186$\\
$F_{a,b}$ & $\nicefrac{1}{4}(\bbn+3\cdot 2^n)$ & $2$ & $5$ & $14$ & $49$ & $202$\\
$F_{ab,c}$ & $\nicefrac{1}{4}(\bbn+2b_n+2^n)$ & $1$ & $4$ & $10$ & $44$ & $186$\\
$F_{a,b,c}$ & $\nicefrac{1}{8}(\bbn+2(b_n+0^n) + 3\cdot 2^n)$ & $1$ & $3$ & $7$ & $26$ & $101$\\
$G_{1,3}$ & $\bcn$ & $2$ & $6$ & $20$ & $76$ & $312$\\
$N(G_{1,3})$ & $\nicefrac{1}{2}(\bcn+2^n)$ & $2$ & $5$ & $14$ & $46$ & $172$\\
$G_{3,3}$ & $\ccn$ & $2$ & $5$ & $14$ & $44$ & $152$\\
$N(G_{3,3})$ & $\nicefrac{1}{2}(\ccn+c_n)$ & $1$ & $3$ & $7$ & $23$ & $76$\\\vspace{2pt}
$\USp(4)$ & $\sum_k\binom{n}{k}2^{n-k}(c_kc_{k+2}-c_{k+1}^2)$ & $1$ & $2$ & $4$ & $10$ & $27$\\\hline
\end{tabular}
\end{center}
\end{table}

\begin{table}
\begin{center}
\footnotesize
\setlength{\extrarowheight}{0.5pt}
\caption{Genus 2 curves realizing Sato-Tate groups}\label{table:curves}
\vspace{6pt}
\begin{tabular}{llll}
$G$ & Curve $y^2 = f(x)$ & $k$ & $K$\\\hline\vspace{-8pt}\\
$C_1$     & $x^6+1$ & $\Q(\sqrt{-3})$ & $\Q(\sqrt{-3})$\\
$C_2$     & $x^5-x$ & $\Q(\sqrt{-2})$ & $\Q(i,\sqrt{2})$\\
$C_3$     & $x^6 + 4$ & $\Q(\sqrt{-3})$ & $\Q(\sqrt{-3},\sqrt[3]{2})$ \\
$C_4$     & $x^6 + x^5 - 5x^4 - 5x^2 - x + 1$ & $\Q(\sqrt{-2})$ & $\Q(\sqrt{-2},a)$; $a^4 + 17a^2 + 68=0$\\
$C_6$     & $x^6 + 2$ & $\Q(\sqrt{-3})$ & $\Q(\sqrt{-3},\sqrt[6]{2})$ \\
$D_2$     & $x^5 + 9x$ & $\Q(\sqrt{-2})$ & $\Q(i,\sqrt{2},\sqrt{3})$ \\
$D_3$     & $x^6 + 10x^3 - 2$ & $\Q(\sqrt{-2})$ & $\Q(\sqrt{-3},\sqrt[6]{-2})$ \\
$D_4$     & $x^5 + 3x$ & $\Q(\sqrt{-2})$ & $\Q(i,\sqrt{2},\sqrt[4]{3})$ \\
$D_6$     & $x^6 + 3x^5 + 10x^3 - 15x^2 + 15x - 6$ & $\Q(\sqrt{-3})$ & $\Q(i,\sqrt{2},\sqrt{3},a)$; $a^3+3a-2=0$ \\
$T$       & $x^6 + 6x^5 - 20x^4 + 20x^3 - 20x^2 - 8x + 8$ & $\Q(\sqrt{-2})$ & $\Q(\sqrt{-2},a,b)$; \\
&&& \hspace{6pt} $a^3-7a+7 = b^4+4b^2+8b+8 = 0$ \\
$O$       & $x^6 - 5x^4 + 10x^3 - 5x^2 + 2x - 1$ & $\Q(\sqrt{-2})$ & $\Q(\sqrt{-2},\sqrt{-11},a,b);$ \\ &&& \hspace{6pt} $a^3-4a+4 = b^4+22b+22=0$ \\
$J(C_1)$  & $x^5-x$ & $\Q(i)$ & $\Q(i,\sqrt{2})$ \\
$J(C_2)$  & $x^5 - x$ & $\Q$ & $\Q(i,\sqrt{2})$ \\
$J(C_3)$  & $x^6 + 10x^3 - 2$ & $\Q(\sqrt{-3})$ & $\Q(\sqrt{-3},\sqrt[6]{-2})$ \\
$J(C_4)$  & $x^6 + x^5 - 5x^4 - 5x^2 - x + 1$ & $\Q$ & see entry for $C_4$ \\
$J(C_6)$  & $x^6 - 15x^4 - 20x^3 + 6x + 1$ & $\Q$ & $\Q(i,\sqrt{3},a)$; $a^3+3a^2-1=0$ \\
$J(D_2)$  & $x^5 + 9x$ & $\Q$ & $\Q(i,\sqrt{2},\sqrt{3})$ \\
$J(D_3)$  & $x^6 + 10x^3 - 2$ & $\Q$ & $\Q(\sqrt{-3},\sqrt[6]{-2})$ \\
$J(D_4)$  & $x^5 + 3x$ & $\Q$ & $\Q(i,\sqrt{2},\sqrt[4]{3})$ \\
$J(D_6)$  & $x^6 + 3x^5 + 10x^3 - 15x^2 + 15x - 6$ & $\Q$ & see entry for $D_6$ \\
$J(T)$    & $x^6 + 6x^5 - 20x^4 + 20x^3 - 20x^2 - 8x + 8$ & $\Q$ & see entry for $T$ \\
$J(O)$    & $x^6 - 5x^4 + 10x^3 - 5x^2 + 2x - 1$ & $\Q$ & see entry for $O$ \\
$C_{2,1}$ & $x^6+1$ & $\Q$ & $\Q(\sqrt{-3})$\\
$C_{4,1}$ & $x^5 + 2x$ & $\Q(i)$ &$\Q(i,\sqrt[4]{2})$ \\
$C_{6,1}$ & $x^6 + 6x^5 - 30x^4 + 20x^3 + 15x^2 - 12x + 1$ & $\Q$ & $\Q(\sqrt{-3},a)$; $a^3-3a+1=0$ \\
$D_{2,1}$ & $x^5 + x$ & $\Q$ & $\Q(i,\sqrt{2})$ \\
$D_{4,1}$ & $x^5 + 2x$ & $\Q$ & $\Q(i,\sqrt[4]{2})$ \\
$D_{6,1}$ & $x^6 + 6x^5 - 30x^4 - 40x^3 + 60x^2 + 24x - 8$ & $\Q$ & $\Q(\sqrt{-2},\sqrt{-3},a)$; $a^3-9a+6 = 0$ \\
$D_{3,2}$   & $x^6 + 4$ & $\Q$ & $\Q(\sqrt{-3},\sqrt[3]{2})$ \\
$D_{4,2}$   & $x^6 + x^5 + 10x^3 + 5x^2 + x - 2$ & $\Q$ & $\Q(\sqrt{-2},a)$; $a^4 - 14a^2 + 28a - 14=0$ \\
$D_{6,2}$   & $x^6 + 2$ & $\Q$ & $\Q(\sqrt{-3},\sqrt[6]{2})$ \\
$O_1$     & $x^6 + 7x^5 + 10x^4 + 10x^3 + 15x^2 + 17x + 4$ & $\Q$ & $\Q(\sqrt{-2},a,b)$; \\ &&& \hspace{6pt} $a^3+5a+10=b^4+4b^2+8b+2=0$ \\
$F_\nothing$& $x^6 + 3x^4 + x^2 - 1$ & $\Q(i,\sqrt{2})$ & $\Q(i,\sqrt{2})$ \\
$F_a$     & $x^6 + 3x^4 + x^2 - 1$ & $\Q(i)$ & $\Q(i,\sqrt{2})$ \\
$F_{ab}$  & $x^6 + 3x^4 + x^2 - 1$ & $\Q(\sqrt{2})$ & $\Q(i,\sqrt{2})$ \\
$F_{ac}$  & $x^5 + 1$ & $\Q$ & $\Q(a)$; $a^4+5a^2+5 = 0$ \\
$F_{a,b}$ & $x^6 + 3x^4 + x^2 - 1$ & $\Q$ & $\Q(i,\sqrt{2})$ \\
$E_1$     & $x^6 + x^4 + x^2 + 1$ & $\Q$ & $\Q$ \\
$E_2$     & $x^6 + x^5 + 3x^4 + 3x^2 - x + 1$ & $\Q$ & $\Q(\sqrt{2})$ \\
$E_3$     & $x^5 + x^4 - 3x^3 - 4x^2 - x$ & $\Q$ & $\Q(a)$; $a^3-3a+1=0$ \\
$E_4$     & $x^5 + x^4 + x^2 - x$ & $\Q$ & $\Q(a)$; $a^4-5a^2+5=0$ \\
$E_6$     & $x^5 + 2x^4 - x^3 - 3x^2 - x$ & $\Q$ & $\Q(\sqrt{7},a)$; $a^3 - 7a - 7 = 0$ \\
$J(E_1)$  & $x^5 + x^3 + x$ & $\Q$ & $\Q(i)$\\
$J(E_2)$  & $x^5 + x^3 - x$ & $\Q$ & $\Q(i,\sqrt{2})$ \\
$J(E_3)$  & $x^6 + x^3 + 4$ & $\Q$ & $\Q(\sqrt{-3},\sqrt[3]{2})$ \\
$J(E_4)$  & $x^5 + x^3 + 2x$ & $\Q$ & $\Q(i,\sqrt[4]{2})$ \\
$J(E_6)$  & $x^6 + x^3 - 2$ & $\Q$ & $\Q(\sqrt{-3},\sqrt[6]{-2})$ \\
$G_{1,3}$   & $x^6 + 3x^4 - 2$ & $\Q(i)$ & $\Q(i)$ \\
$N(G_{1,3})$& $x^6 + 3x^4 - 2$ & $\Q$ & $\Q(i)$ \\
$G_{3,3}$   & $x^6 + x^2 + 1$ & $\Q$ & $\Q$ \\
$N(G_{3,3})$& $x^6 + x^5 + x - 1$ & $\Q$ & $\Q(i)$ \\
$\USp(4)$ & $x^5 - x + 1$ & $\Q$ & $\Q$ \\\hline
\end{tabular}
\end{center}
\end{table}

\begin{table}
\begin{center}
\footnotesize
\setlength{\extrarowheight}{0.5pt}
\caption{Some automorphisms of the curves in Table~\ref{table:curves}}
\label{table:automorphisms}
\vspace{6pt}
\begin{tabular}{lllll}
$G$  & $\alpha$ & $\gamma$ & $M$ 
\\\hline\vspace{-8pt}\\[4pt]
$C_1,\,C_{2,1}$      & $(-x,y)$ & $\left(\frac{1}{x},\frac{1}{x^3} y\right)$ & $\Q(\sqrt{-3})$ 
\\[4pt]
$C_2,\,J(C_1),\,J(C_2)$   & $\left(\frac{i}{x},\frac{\zeta_8^3}{x^3}y\right)$ & $(-x,iy)$ & $\Q(\sqrt{-2})$ 
\\[4pt]
$C_3, D_{3,2}$  & $(-x,y)$ & $\left(\frac{4^{1/3}}{x}, \frac{2}{x^3}y \right)$ & $\Q(\sqrt{-3})$ 
\\[4pt]
$C_4,\,J(C_4)$  & $\left(\frac{-(2a^2+13)x+1}{x+2a^2+13},\frac{Q_{\alpha,C_4}}{(x+2a^2+13)^3}y\right)$ & $\left(\frac{-x-1}{x-1},\frac{-2\sqrt{-2}}{(x-1)^3}y\right)$ & $\Q(\sqrt{-2})$ 
\\[4pt]
$C_6,\,D_{6,2}$      & $(-x,y)$ & $\left(\frac{2^{1/3}}{x},\frac{2^{1/2}}{x^3} y\right)$ & $\Q(\sqrt{-3})$ 
\\[4pt]
$D_2,\,J(D_2)$    & $\left(\frac{3}{x},\frac{3^{3/2}}{x^3} y\right)$ & $\left(-x,iy\right)$  & $\Q(\sqrt{-2})$ 
\\[4pt]
$D_3,\,J(D_3),\,J(C_3)$     & $\left(\frac{-2^{1/3}}{x},\frac{\sqrt{-2}}{x^3}y\right)$ & $\left(\frac{(1-\sqrt{-3})2^{1/3}}{2x},\frac{\sqrt{-2}}{x^3}y\right)$ & $\Q(\sqrt{-2})$ 
\\[4pt]
$D_4,\,J(D_4)$   & $\left(\frac{\sqrt 3}{x},\frac{3^{3/4}}{x^3}y\right)$ & $\left(-x,iy\right)$ &  $\Q(\sqrt{-2})$ 
\\[4pt]
$D_6,\,J(D_6)$    & $\left(\frac{x+1}{x-1},\frac{2\sqrt{2}}{(x-1)^3}y\right)$ & $\left(\frac{P_{\gamma,D_6}}{2x+(a^2+3)},\frac{Q_{\gamma,D_6}}{(2x+(a^2+3))^3}y\right)$  & $\Q(\sqrt{-3})$ 
\\[4pt]
$T,\,J(T)$     & $\left( \frac{P_{\alpha,T}}{R_{\alpha,T}}, \frac{Q_{\alpha,T}}
{R_{\alpha,T}^3}y \right)$ & $\left(\frac{P_{\gamma,T}}{x+1-a},\frac{Q_{\gamma,T}}{(x+1-a)^3}y\right)$ & $\Q(\sqrt{-2})$ 
\\[4pt]

$O,\,J(O)$  & $\left( \frac{P_{\alpha,O}}{R_{\alpha,O}}, \frac{Q_{\alpha,O}}{R_{\alpha,O}^3}y \right)$ & $\left(\frac{ax+a^2-2}{2x-a},\frac{Q_{\gamma,O}}{(2x-a)^3}y\right)$ & $\Q(\sqrt{-2})$ 
\\[4pt]
$J(C_6)$  & $\left(\frac{-x-2}{2x+1},\frac{3\sqrt{-3}}{(2x+1)^3}y\right)$ & $\left(\frac{-(a+1)x-a}{x+a+1},\frac{-3ia^2-3ia}{(x+a+1)^3}y\right)$ & $\Q(\sqrt{-3})$ 
\\[4pt]

$C_{4,1},\,D_{4,1}$  & $\left(\frac{\sqrt 2}{x},\frac{2^{3/4}}{x^3}y\right)$ & $\left(-x,iy\right)$ &  $\Q(\sqrt{-2})$ 
\\[4pt]
$C_{6,1}$  & $\left(\frac{(1-a)x+a}{x+(a-1)},\frac{-3a^2+3a}{(x+(a-1))^3}y\right) $ & $\left(\frac{x-1}{x},\frac{-1}{x^3}y\right)$ & $\Q(\sqrt{-3})$ 
\\[4pt]
$D_{2,1}$  & $\left(\frac{1}{x},\frac{1}{x^3}y\right)$ & $(-x,iy)$ & $\Q(\sqrt{-2})$ 
\\[4pt]
$D_{6,1}$ & $\left(\frac{-2}{x},\frac{-2\sqrt{-2}}{x^3}y\right)$ & $\left(\frac{(a-1)x+2}{x+1-a},\frac{Q_{\gamma,D_{6,1}}}{(x+1-a)^3}y\right)$ & $\Q(\sqrt{-3})$ 
\\[4pt]
$D_{4,2}$  & $\left(\frac{P_{\alpha,D_{4,2}}}{R_{\alpha,D_{4,2}}},\frac{Q_{\alpha,D_{4,2}}}{R_{\alpha,D_{4,2}}^3}y\right)$ & $\left(\frac{P_{\gamma,D_{4,2}}}{x-\sqrt{-2}+1},\frac{-8}{(x-\sqrt{-2}+1)^3}y\right)$ & $\Q(\sqrt{-2})$ 
\\[4pt]
$O_1$  & $\left( \frac{P_{\alpha,O_1}}{R_{\alpha,O_1}}, \frac{Q_{\alpha,O_1}}{R_{\alpha,O_1}^3}y \right)$ & $\left(\frac{P_{\gamma,O_1}}{R_{\gamma,O_1}},\frac{Q_{\gamma,O_1}}{R_{\gamma,O_1}^3}y\right)$ & $\Q(\sqrt{-2})$ 
\\[4pt]

$E_1$   & $(-x,y)$ & $\left(\frac{-1}{x},\frac{1}{x^3} y\right)$ & $\Q$ 
\\[4pt]
$E_2$   & $\left(\frac{x+1}{x-1},\frac{-2\sqrt 2}{(x-1)^3}y\right)$ & $\left(\frac{-1}{x},\frac{1}{x^3}y\right)$ & $\Q$ 
\\[4pt]
$E_3$     & $\left(\frac{-ax-a+1}{x+a},\frac{3a^2-3a}{(x+a)^3}y\right)$ & $\left(\frac{-x-1}{x},\frac{1}{x^3}y\right)$ & $\Q$ 
\\[4pt] 
$E_4$      & $\left(\frac{(a^2-3)x+1}{x-a^2+3},\frac{-3a^3+10a}{(x-a^2+3)^3}y\right)$ & $\left(\frac{-1}{x},\frac{1}{x^3}y\right)$ & $\Q$ 
\\[4pt] 
$E_6$    & $\left(\frac{P_{\alpha,E_6}}{R_{\alpha,E_6}},\frac{Q_{\alpha,E_6}}{R_{\alpha,E_6}^3}y\right)$ & $\left(\frac{-1}{x+1},\frac{1}{(x+1)^3}y\right)$ & $\Q$ 
\\[4pt] 
$J(E_1)$   & $\left(\frac{1}{x},\frac{1}{x^3}y\right)$ & $\left(-x,iy\right)$ & $\Q$ 
\\[4pt]
$J(E_2)$   & $\left(\frac{-i}{x},\frac{i-1}{\sqrt 2 x^3}y\right)$ & $(-x,iy)$ & $\Q$ 
\\[4pt]
$J(E_3)$  & $\left(\frac{4^{1/3}}{x},\frac{2}{x^3}y\right)$ & $(\zeta_3x,y)$ & $\Q$ 
\\[4pt]
$J(E_4)$  & $\left(\frac{\sqrt 2}{x},\frac{2^{3/4}}{x^3}y\right)$ & $(-x,iy)$ & $\Q$ 
\\[4pt]
$J(E_6)$   & $\left(\frac{(-2)^{1/3}}{x},\frac{\sqrt 2}{x^3}y\right)$ & $(\zeta_3x,y)$ & $\Q$ 
\\\hline
\end{tabular}
\end{center}
\end{table}

\begin{table}
\begin{center}
\footnotesize
\setlength{\extrarowheight}{0.5pt}
\caption{Polynomials describing the automorphisms of Table~\ref{table:automorphisms}}\label{table:auxpols}
\vspace{6pt}
\begin{tabular}{lll}
\hline\\[-4pt]
$Q_{\alpha,C_4}$ & $=$ & $\sqrt{-2}(29a^3+187a)$\\[4pt]

$P_{\gamma,D_6}$ & $=$ & $-(a^2+3)x+2(a^2+2)$\\[4pt]

$Q_{\gamma,D_6}$ & $=$ & $-21ia^2-6ia-83i$\\[4pt]

$P_{\alpha,T}$ & $=$ & $\bigl((\frac{-a^2}{7} - \frac{a}{6} + \frac{2}{3})b^3 + (\frac{3a^2}{14} + \frac{a}{3} - \frac{7}{6})b^2 + (\frac{-10}{21}a^2- \frac{2}{3}a + \frac{7}{3})b + (\frac{-2}{21}a^2 + \frac{2}{3})\bigr)x $\\[4pt]

& & $+(\frac{1}{21}a^2 - \frac{1}{3})b^3 + (\frac{-5}{21}a^2 - \frac{1}{3}a + \frac{4}{3})b^2 + (\frac{-2}{7}a^2 - \frac{2}{3}a + \frac{2}{3})b + (\frac{10}{21}a^2 - \frac{8}{3})$ \\[4pt]

$Q_{\alpha,T}$ & $=$ & $\sqrt{-2}\bigl((\frac{10}{27}a^2 + \frac{11}{27}a - \frac{49}{27})b^3 + (\frac{8}{27}a^2 + \frac{22}{27}a - \frac{49}{27})b^2 $\\[4pt]

& &$+ (\frac{32}{27}a^2 + \frac{46}{27}a - \frac{182}{27})b+ (\frac{76}{27}a^2 +  \frac{110}{27}a - \frac{392}{27})\bigr)$ \\[4pt]

$R_{\alpha,T}$ & $=$ & $x + (\frac{a^2}{7} + \frac{a}{6} - \frac{2}{3})b^3 + (\frac{-3}{14}a^2 - \frac{a}{3} + \frac{7}{6})b^2 +(\frac{10}{21}a^2 + \frac{2}{3}a - \frac{7}{3})b + (\frac{2}{21}a^2 - \frac{2}{3})$ \\[4pt]

$P_{\gamma,T}$ & $=$ & $(a-1)x+2a^2+2a-8$\\[4pt]

$Q_{\gamma,T}$ & $=$ & $(-a^2 + \frac{3}{2}a)b^3 + (\frac{a^2}{2} - \frac{a}{2})b^2 + (-5a^2 + 12a - 7)b - 5a^2 + 8a$\\[4pt]

$P_{\alpha,O}$ & $=$ & $((\frac{19}{429}a^2 + \frac{23}{429}a - \frac{58}{429})b^3 + (\frac{-5}{78}a^2 - \frac{2}{39}a + \frac{3}{13})b^2 +(\frac{11}{78}a^2 + \frac{7}{39}a - \frac{4}{13})b $\\[4pt]

& & $+ (\frac{31}{78}a^2 + \frac{5}{13}a - \frac{35}{39}))x+(\frac{-23}{429}a^2 - \frac{6}{143}a + \frac{76}{429})b^3 + (\frac{2}{39}a^2 + \frac{1}{39}a -\frac{10}{39})b^2$\\[4pt]

& & $+ (\frac{-7}{39}a^2 - \frac{10}{39}a + \frac{22}{39})b + (\frac{-23}{26}a^2 - \frac{9}{13}a +\frac{101}{39})$\\[4pt]

$Q_{\alpha,O}$ & $=$ & $\sqrt{-2}\bigl((\frac{a^2}{39} + \frac{20}{351}a - \frac{22}{351})b^3 + (\frac{-11}{702}a^2 - \frac{44}{351}a - \frac{22}{117})b^2  $\\[4pt]

& & $+ (\frac{107}{351}a^2 + \frac{232}{351}a - \frac{146}{351})b+(\frac{11}{26}a^2 +\frac{110}{117}a - \frac{121}{117})\bigr)$\\[4pt]

$R_{\alpha,O}$ & $=$ & $x + (\frac{-19}{429}a^2 - \frac{23}{429}a + \frac{58}{429})b^3 + (\frac{5}{78}a^2 + \frac{2}{39}a -
   \frac{3}{13})b^2 $\\[4pt]

& & $+ (\frac{-11}{78}a^2 - \frac{7}{39}a + \frac{4}{13})b + (\frac{-31}{78}a^2 - \frac{5}{13}a + \frac{35}{39})$\\[4pt]

$Q_{\gamma,O}$ & $=$ & $8\sqrt{-11}\bigl((\frac{1}{286}a^2 + \frac{3}{143}a - \frac{7}{143})b^3 + (\frac{-7}{572}a^2 + \frac{5}{286}a + \frac{5}{143})b^2 $\\[4pt]

& & $+ (\frac{1}{13}a^2 - \frac{1}{26}a - \frac{1}{13})b + (\frac{3}{52}a^2 + \frac{9}{26}a - \frac{21}{26})\bigr)$\\[4pt]

$Q_{\gamma,D_{6,1}}$ & $=$ & $\sqrt{-3}(3a^2-6a+1) $\\[4pt]

$P_{\alpha,D_{4,2}}$ & $=$ & $(a^3+a^2-14a+17)x+(2a^3+2a^2-28a+29)$\\[4pt]

$Q_{\alpha,D_{4,2}}$ & $=$ & $-160a^3-168a^2+1904a-2184$\\[4pt]

$R_{\alpha,D_{4,2}}$ & $=$ & $5x-(a^3+a^2-14a+17)$\\[4pt]

$P_{\gamma,D_{4,2}}$ & $=$ & $-(\sqrt{-2}+1)x+1$\\[4pt]

$P_{\alpha,O_1}$ & $=$ & $\bigl((\frac{-3}{580}a^2 - \frac{13}{116}a - \frac{7}{58})b^3 + (\frac{-1}{1160}a^2 + \frac{15}{232}a + \frac{17}{116})b^2 +  (\frac{-17}{232}a^2 - \frac{117}{232}a - \frac{63}{116})b $\\[4pt]

& &  $+ (\frac{13}{290}a^2 - \frac{21}{58}a - \frac{18}{29})\bigr)x+(\frac{-7}{580}a^2 - \frac{11}{116}a + \frac{3}{58})b^3 + (\frac{23}{580}a^2 + \frac{3}{116}a + \frac{15}{58})b^2  $\\[4pt]

& & $+ (\frac{-15}{116}a^2 - \frac{35}{116}a - \frac{1}{58})b + (\frac{109}{580}a^2 - \frac{11}{116}a +   \frac{61}{58})$\\[4pt]

$Q_{\alpha,O_1}$ & $=$ & $(\frac{-2277}{24389}a^2 + \frac{7186}{24389}a - \frac{495}{24389})b^3 + (\frac{1287}{48778}a^2 -
   \frac{34813}{48778}a - \frac{13115}{24389})b^2 $\\[4pt]

& &  $+ (\frac{-26733}{48778}a^2 + \frac{83435}{48778}a + \frac{26255}{24389})b+(\frac{-12375}{24389}a^2 + \frac{8303}{24389}a - \frac{29200}{24389})$\\[4pt]

$R_{\alpha,O_1}$ & $=$ & $x + (\frac{3}{580}a^2 + \frac{13}{116}a + \frac{7}{58})b^3 + (\frac{1}{1160}a^2 - \frac{15}{232}a - \frac{17}{116})b^2 $ \\[4pt]

& & $+ (\frac{17}{232}a^2 + \frac{117}{232}a + \frac{63}{116})b + (\frac{-13}{290}a^2 + \frac{21}{58}a + \frac{18}{29})$\\[4pt]

$P_{\gamma,O_1}$ & $=$ & $(-a^2+a-8)x-a^2-a-6$\\[4pt]

$Q_{\gamma,O_1}$ & $=$ & $4\sqrt{-2}\bigl((-2a^2+4a-10)b^3+(-3a^2+4a-25)b^2$\\[4pt]

&& $+(-2a^2+10a)b-18a^2+32a-110\bigr)$\\[4pt]

$R_{\gamma,O_1}$ & $=$ & $4x+a^2-a+8$\\[4pt]

$P_{\alpha,E_6}$ & $=$ & $(a^2-a-5)x+a^2-a-4$\\[4pt]

$Q_{\alpha,E_6}$ & $=$ & $\sqrt 7(-a^2+2a+6)$\\[4pt]

$R_{\alpha,E_6}$ & $=$ & $x-a^2+a+5$\\[4pt]\hline

\end{tabular}
\end{center}
\end{table}

\end{document}